\documentclass[final]{siamltex}
\usepackage{amssymb}
\usepackage{array}
\usepackage{amsmath}
\usepackage{epstopdf}
\usepackage{epsfig}
\usepackage{subfigure}
\usepackage[margin=0.6in]{geometry}
\usepackage{pstricks}
\newtheorem{remark}{Remark}

\newtheorem{Definition}{Definition}

\title{The role of numerical  boundary  procedures in the stability of  perfectly matched layers}
\author{Kenneth Duru\thanks{Department of Geophysics, Stanford University, Stanford, CA ({\tt kduru@stanford.edu})}}

\begin{document}

\maketitle

\begin{abstract} 
In this paper we address the temporal energy growth associated with numerical approximations of the perfectly matched layer (PML) for Maxwell's equations in first order form.  In the literature, several studies  have shown that a numerical method which is stable in the absence of the PML can become unstable when the PML is introduced. We demonstrate in this paper that this instability  can be  directly related to  numerical treatment of boundary conditions in the PML.  First, at the continuous level, we establish the  stability of the constant coefficient initial boundary value problem  for the PML. To enable  the construction of  stable numerical  boundary procedures, we derive energy estimates for the variable coefficient PML.  Second,  we develop a high order accurate and stable numerical approximation for the PML  using summation--by--parts finite difference operators to approximate spatial derivatives and weak enforcement of boundary conditions using penalties. By constructing analogous discrete energy estimates we show  discrete stability and convergence of the numerical method. Numerical experiments verify the  theoretical results. 
\end{abstract}
\pagestyle{myheadings}
\thispagestyle{plain}
\markboth{KENNETH DURU}{BOUNDARY PROCEDURES AND DISCRETE STABILITY OF THE PML}
\section{Introduction}
Numerical approximations of the perfectly matched layer (PML) \cite{Abarbanel1998, Berenger1994} for Maxwell's equations in first order form often exhibit  a late time linear or exponential energy growth, see for example \cite{Abarbanel2002, Abarbanel2009,  Becache2002, Becache2004}.  This unwanted growth can spread into the computational domain and ruin the accuracy of a numerical simulation.   Growth is often seen long after a wave has exited the domain or when a wave has been trapped in the domain for a longtime. The later  situation can occur, for example, in the simulation of an electromagnetic waveguide. This has generated substantial interests.  One explanation for this  behavior of the PML was offered in \cite{Abarbanel1997}. It is  shown that the split--field  PML is weakly hyperbolic, and   numerical approximations  can lead to  perturbations that excite explosive modes in the PML.   Because of the result in \cite{Abarbanel1997}, several unsplit  (strongly hyperbolic) PMLs for Maxwell's equations  were developed \cite{Abarbanel1998}. However,  series of numerical experiments presented in the literature, see for instance    \cite{Abarbanel2002, Abarbanel2009}, using  standard finite difference schemes and \cite{Abenius2005}  for finite element methods,  suggest that these  unsplit PMLs also suffer from the same late time exponential energy growth as the split field PML. It was later shown in \cite{Abarbanel2002} that the source of the late time exponential growth was   characterized by the multiplicity of eigenvalues and the degeneracy of eigenvectors of the lower order undifferentiated terms. In \cite{Becache2004, Becache2002}, however,  it was shown, for a given computational setup, that the growth generated in the PML, if any,  may not propagate back into the computational domain. It is also important to note that a partial  remedy for this problem in odd (1 and 3)  spatial   dimensions   has been proposed in \cite{Qasimov2008}, using a lacunae based time integration. Nonetheless, a recent  analysis of a PML for Maxwell's equations in \cite{DuKrApnum2013} shows that  the corresponding initial value problem (IVP) is  asymptotically stable.  In this study, we argue that many of the late time energy growth of the PML for Maxwell's equations observed in computations are caused by  inadequate numerical boundary treatments in the PML.  

We denote that the well--posedness and temporal stability analysis of the PML has been a topic of several works, see for instance \cite{Becache2003, LionsMetralVacus2001, HalpernPetit-BergezRauch2011, Abenius2005, Abarbanel1997, Abarbanel2002, Abarbanel2009, Abarbanel1998, Qasimov2008,  KrDuBIT2013,  DuAcoustic, DuKrApnum2013, DuKr, Becache2002, Becache2004,  AppeloHagstromKreiss2006} and many others. In \cite{HalpernPetit-BergezRauch2011}, perfect matching, well-posedness and temporal stability for  PMLs as well as some other classes of layers are discussed for general systems. Unfortunately, there is no guarantee that all solutions decay with time. In \cite{Becache2003} however, the geometric stability condition was introduce to characterize the temporal stability of IVP for the  PML. If this condition is not satisfied, then there are modes of high spatial frequencies with temporally growing amplitudes. In this paper, we consider a two space dimensional model problem satisfying the geometric stability condition but augmented with boundary conditions in the tangential and normal directions. Our primary objective in this paper is the analysis and development of high order accurate and provably stable numerical approximations of the corresponding initial boundary value  problem (IBVP).

It  is a well known fact  that the  theory and numerical techniques to solve an IBVP are  more elaborate and complicated than that of  a corresponding IVP. For a well--posed partial differential equation which is stable in the absence of boundaries can support unstable solutions, or become ill--posed, when boundaries are introduced. In fact, the analysis of the IVP is a necessary (first) step towards the analysis of the IBVP  \cite{Kreiss1989, Kreiss1970, MotamedKreiss}. For symmetric  hyperbolic systems such as the Maxwell's equations  and certain boundary conditions, using the energy method, it is possible to obtain a uniform bound on the L$_2$--norm of the  solutions (electric and magnetic fields) in terms of the L$_2$--norm of   the initial data and a boundary norm of the   boundary  data \cite{Friedrichs1967, Friedrichs1958, Friedrichs1954}. In numerical approximations of  IBVPs, most of the difficulties arise from the boundaries. Similarly, a numerical method which is stable in the absence of boundaries can support growth when boundary conditions are imposed. In the discrete setting,  it is possible though  to design stable numerical boundary procedures   by mimicking the continuous estimates  \cite{Abarbanel2000_1, Abarbanel2000_2, Nordstrom2008, Bert_Gust, Bert_Gust1998, Matt2003}. 

Effective boundary treatments are also important when PMLs are used in computations.  If there are physical boundaries, for instance in waveguides \cite{DuKr, DuWaveGuidesPML}, the corresponding boundary conditions must be accurately extended into the PML. In practice, the PML must  be solved on a bounded computational domain.  A stable   and accurate PML boundary closure becomes essential, since it enables efficient numerical computations. There is an ongoing discussion though on how to effectively terminate the PML. It is also important to point out that a numerical method which is stable in the absence of the PML can support growth when the PML is included.  A major difficulty is that in general the PML (for symmetric systems) is asymmetric. Therefore, deriving energy estimates for the PML that are useful in designing stable and accurate numerical boundary procedures can be very challenging. Linear and/or exponential temporal growth are frequently seen in numerical simulations of PMLs \cite{Abarbanel2002, Abarbanel2009, Qasimov2008},  particularly if higher order methods  are considered.
Often,  artificial numerical dissipations are added to the scheme in order to eliminate these unwanted temporal growth. Artificial numerical dissipations certainly helps in many ways, but  can also introduce  some interesting, and unfortunately undesirable, effects. In standard computational electromagnetic codes, discrete stability of the numerical method is usually investigated in the absence of the PML. Once the interior scheme is confirmed stable, the PML is then included by discretizing the accompanying auxiliary differential equations and appending the auxiliary functions accordingly, without further discrete stability analysis. As in \cite{Abarbanel2002, Abarbanel2009, Qasimov2008}, any eventual growth  seen when the absorption function is present is then attributed to the PML.  We emphasize that a more complete analysis of discrete stability must include  all the equations and variables in the PML.

In bounded and semi--bounded domains the stability of the PML has  been analyzed for certain classes of problems and boundary conditions \cite{DuWaveGuidesPML, DuBoundaryPML} using normal mode analysis. These results  \cite{DuWaveGuidesPML, DuBoundaryPML} are mostly for second order systems. In the continuous setting, it is possible  to extend some of these theoretical  results to the  PML for first order systems. However, numerical treatments for first order systems are essentially different from numerical methods for second order systems. For instance, standard boundary conditions for first order hyperbolic systems can only contain a  linear combination of the unknown variables,  while boundary conditions for second order systems often involve  combinations of the unknown variables and their spatial and temporal derivatives. Therefore, for first order systems different numerical techniques and,  probably,   new theories   are  required.

In the present paper, we consider  unsplit PMLs \cite{Abarbanel2002, Abarbanel2009, DuKrApnum2013} and the classical split--field PML \cite{Berenger1994} , for Maxwell's equations in first order form  with boundary conditions set by an electromagnetic waveguide. We are particularly interested in high order accurate and efficient methods for numerical simulation of Maxwell's equations.  
Our first objective is to  extend the continuous analysis for Cauchy problems \cite{DuKrApnum2013} to the PML in  semi--bounded and bounded domains.   The first sets of results are the proofs of  stability of the constant coefficient PML in a lower half--plane and a left half--plane. The second results in the paper are the derivation of energy estimates for constant coefficients in the Laplace space and for variable coefficients PML in the time domain.   

We show that a  stable numerical boundary procedure for the undamped problem can become unstable when the PML is introduced.   

The second objective of the paper is to develop a high order accurate and provably stable numerical approximation  for the PML in a  bounded domain. We use summation--by--parts (SBP) finite difference  operators \cite{Bert_Gust, Strand94}   to approximate spatial derivatives.  Boundary conditions are enforced weakly    using a penalty technique, referred to as the simultaneous approximation term (SAT) methodology \cite{CarpenterGottliebAbarbanel1994, Matt2003, Nordstrom2008}. By  mimicking the continuous  energy estimates we construct accurate and stable boundary procedures. The  convergence of the numerical method follows. 
Finally, we  present numerical experiments, using SBP finite difference  operators in space with interior order of accuracy 2, 4 and 6,  and the classical fourth order accurate Runge--Kutta scheme in time.  The numerical examples illustrate  high order accuracy and longtime stability of the PML. 

Note that in \cite{Abarbanel2002, Abarbanel2009} it was pointed out that the discrete PML for Maxwell's equations in first order form can suffer severely  from longtime exponential energy growth  when spatial derivatives are discretized with high order accurate central finite difference operators and and the time derivative  discretized using Runge--Kutta schemes. This indicates that an appropriate remedy is required. The results obtained in this paper provide a cure for this problem without any additional cost.  We believe that the results obtained  in this paper can be extended to other situations, not only for finite difference electromagnetic  codes, but also for finite volume methods, finite element methods and spectral  methods, and  for more complicated systems such as the elastic wave equation in isotropic media and the  linearized Euler equation with vanishing mean flow.

The remainder of the paper is planned as follows. In the next section, we present some preliminaries, introducing Maxwell's equations in two space dimensions and a brief description of a discrete approximation  of the Maxwell's equations in  a bounded  domain using the SBP-SAT methodology.
 In section 3, we introduce the PMLs and present some instructive numerical experiments, motivating the continuous analysis in section  4 and the discrete analysis in section 5. In section 5, we construct discrete approximations for the PML, derive discrete energy estimates and prove asymptotic stability.
 More numerical experiments are presented in section 6, verifying the analysis of previous sections.
 A summary of the paper and suggestions for future work are offered in the last section. 
\section{Preliminary}
Here, we introduce the Maxwell's equations in a two  dimensional rectangular domain. We end the section with a   brief description of  the SBP--SAT methodology for Maxwell's equations.
\subsection{The Maxwell's equations in a rectangular domain}
Consider the transverse magnetic $(TM_z)$ case of the Maxwell equation  in a source free, homogeneous isotropic medium,
\begin{equation}\label{eq:Maxwell}
\begin{split}
 &\frac{\partial{E_z}}{\partial t} = -\frac{\partial{H_y}}{\partial x} + \frac{\partial{H_x}}{\partial y},\\
 &\frac{\partial{H_y}}{\partial t} = -\frac{\partial{E_z}}{\partial x}, \\
 &\frac{\partial{H_x}}{\partial t} = \frac{\partial{E_z}}{\partial y}.\\
  \end{split}
  \end{equation}
Here,
$
\textbf{H} = \left(H_x, H_y, 0\right)^T, \quad \textbf{E} = \left(0, 0, E_z\right)^T
 $
are the magnetic fields and electric fields respectively.  
At the initial   time $t = 0$ we set the smooth initial data
\[ 
  \left(E_z\left(x, y, 0\right), H_y\left(x, y, 0\right), H_x\left(x, y, 0\right)\right)^T = \left(f_1(x,y), f_2(x,y), f_3(x,y) \right)^T. 
\]
Consider now the two space dimensional rectangular domain $\left(x, y\right) \in [-x_0-\delta, x_0+\delta] \times [-y_0, y_0]$, with $x_0, y_0, \delta > ~0$. 
The  length and width of the rectangle are $\Gamma_x := \{ x: -x_0-\delta \le x\le x_0+\delta\} $ and  $\Gamma_y := \{ y: -y_0 \le y\le y_0\}$. At the edges of the rectangle, at $x = \pm \left(x_0+\delta\right)$ and $y = \pm y_0$, we set the general boundary conditions, expressed as  a relation between the incoming and outgoing  characteristics  
\begin{equation}\label{eq:wall_bc_x}
 \frac{1}{2}\left({E_z \mp   H_y}\right) = \frac{R_x}{2} \left(E_z \pm H_y \right) \iff \frac{\left(1-R_x\right)}{2} E_z \mp \frac{\left(1+R_x\right)}{2}H_y = 0, \quad \text{at} \quad x = \pm \left(x_0+\delta\right),
\end{equation}
\begin{equation}\label{eq:wall_bc_y_0}
\frac{1}{2} \left(E_z \pm H_x \right)= \frac{R_y}{2} \left(E_z \mp H_x \right)  \iff \frac{\left(1-R_y\right)}{2} E_z \pm \frac{\left(1+R_y\right)}{2}H_x = 0, \quad \text{at} \quad y = \pm y_0.
\end{equation}
Here, the reflection coefficients $R_j$  $(j = x, y),$ with $|R_j| \le 1$, are real constants. 
Note that the boundary condition \eqref{eq:wall_bc_x}, \eqref{eq:wall_bc_y_0} can support boundary waves: waves whose amplitudes are maximum on the boundaries, at $x = \pm \left(x_0+\delta\right)$, $y = \pm y_0$, but decay exponentially into the domain. In particular if $R_j = 1$, then the boundary conditions \eqref{eq:wall_bc_x}, \eqref{eq:wall_bc_y_0} model insulated electromagnetic walls.   Glancing waves, constant modes normal to the boundaries,  can be supported.  If $R_j = -1$, then the boundary conditions \eqref{eq:wall_bc_x}, \eqref{eq:wall_bc_y_0} model  perfect electric conductor  (PEC) boundary conditions. And if $R_j = 0$, then the boundary conditions \eqref{eq:wall_bc_x}, \eqref{eq:wall_bc_y_0} correspond to  local one dimensional absorbing boundary conditions, which are also equivalent to setting the incoming characteristics (Riemann invariant) at the boundaries to zero.

Note  that if  $R_j\ne -1$ then we have
\begin{equation}\label{eq:wall_bc_x_0}
 \gamma_x E_z \mp H_y = 0, \quad \text{at} \quad x = \pm \left(x_0+\delta\right), \quad  \gamma_x =  \frac{\left(1-R_x\right)}{\left(1+R_x\right)} \ge 0,
\end{equation}
\begin{equation}\label{eq:wall_bc_y}
 \gamma_y E_z \pm H_x = 0, \quad \text{at} \quad y = \pm y_0, \quad  \gamma_y =  \frac{\left(1-R_y\right)}{\left(1+R_y\right)} \ge 0.
\end{equation}
The special case  of  $R_j= -1$ in  \eqref{eq:wall_bc_x}, \eqref{eq:wall_bc_y_0} corresponds to  PEC boundary conditions 
\begin{equation}\label{eq:wall_bc_x_pec}
  E_z  = 0, \quad \text{at} \quad x = \pm \left(x_0+\delta\right),
\end{equation}
\begin{equation}\label{eq:wall_bc_y_pec}
  E_z  = 0, \quad \text{at} \quad y = \pm y_0.
\end{equation}

Let the standard $L_2$ inner product and the corresponding norm be defined by
\begin{equation}
\left(u, v\right) = \int_{\Omega}v^* u \mathrm{d}x\mathrm{d}y, \quad \|u\|^2 = \left(u, u\right).
\end{equation}
We also introduce the boundary norms  $\|\cdot\|^2_{\Gamma_x}$, $\|\cdot\|^2_{\Gamma_y}$, defined by
\begin{equation}
\left(u_0, v_0\right)_{\Gamma_z} = \int_{\Gamma_z}v_0^* u_0\mathrm{d}z, \quad \|u_0\|^2_{\Gamma_z} = \left(u_0, u_0\right)_{\Gamma_z}, \quad z \in \Gamma_z, \quad z = x, y.
\end{equation}
Here, $v^*, v_0^*$ are the complex conjugates of $v, v_0$ and $u_0(z), v_0(z)$ depend on one spatial variable $z = x$ or $z = y$.
Using the energy method and integration by parts, it is  straightforward to show that the Maxwell's equation \eqref{eq:Maxwell} subject to the boundary conditions \eqref{eq:wall_bc_x_0}, \eqref{eq:wall_bc_y},  with $\gamma_j \ge 0$, satisfies
\begin{equation}\label{eq:strict}
\begin{split}
&\|E_z\left(t\right)\|^2  + \|H_y\left(t\right)\|^2  + \|H_x\left(t\right)\|^2  + 2\gamma_x \int_0^t \left(\|E_z\left(t', x_0+\delta\right)\|^2_{\Gamma_y} + \|E_z\left(t', -x_0-\delta\right)\|^2_{\Gamma_y}\right)dt' \\
&+ 2\gamma_y \int_0^t \left(\|E_z\left(t', y_0\right)\|^2_{\Gamma_x} + \|E_z\left(t', -y_0\right)\|^2_{\Gamma_x}\right)dt'= \|f_1\|^2  + \|f_2\|^2  + \|f_3\|^2.
\end{split}
\end{equation}
In particular if  $\gamma_j = 0$ in \eqref{eq:wall_bc_x_0}, \eqref{eq:wall_bc_y} or $R_j = -1$ (we consider the PEC condition) in \eqref{eq:wall_bc_x},  \eqref{eq:wall_bc_y_0} then 
\begin{equation}\label{eq:conserve}
\|E_z\left(t\right)\|^2  + \|H_y\left(t\right)\|^2  + \|H_x\left(t\right)\|^2  = \|f_1\|^2  + \|f_2\|^2  + \|f_3\|^2.
\end{equation}
Note that the $L_2$-norm of the solutions (electric and magnetic fields) are uniformly bounded in time by the $L_2$-norm of the initial data.
\subsection{A stable discrete approximation of the Maxwell's equations using the SBP--SAT methodology}\label{sec:SBP_SAT}
This subsection serves as a brief introduction to the SBP--SAT scheme for IBVPs (Maxwell's equations in bounded domains). The main idea is to approximate all spatial derivatives in \eqref{eq:Maxwell} using SBP finite difference operators. The boundary conditions  \eqref{eq:wall_bc_x_0}, \eqref{eq:wall_bc_y} are then enforced weakly using penalties. However, penalty parameters must be chosen such that discrete energy estimates, analogous to the continuous estimates \eqref{eq:strict}, \eqref{eq:conserve}, can be derived. Thus, proving asymptotic stability for the discretization. 
More elaborate discussions on this topic can be found in \cite{Abarbanel2000_1, Abarbanel2000_2, Bert_Gust, Bert_Gust1998, CarpenterGottliebAbarbanel1994, Matt2003, Nordstrom2008}, see also the recent review paper \cite{SvardNordstrom2014}.

To begin with, consider the uniform discretization of the unit interval with $N_x$ number of grid points and the spatial step, $h >0$,
\begin{align}\label{eq:grid}
x_i = \left(i-1\right)h, \quad i = 1, 2, \dots, N_x, \quad h = \frac{1}{N_x-1}.
\end{align}
Let $D_x$ denote a differentiation matrix approximating  the  first derivative,
 $
 D_x \approx {\partial}/{\partial x},
 $
on the grid \eqref{eq:grid}.
The matrix $D_x$ is an  SBP operator (in one space dimension), see \cite{Bert_Gust, Strand94},  if   the following properties hold
\begin{equation}\label{eq:SBPDx}
\begin{split}
&D_x = \mathrm{P}_x^{-1}Q_x, \quad Q_x^{T} + Q_x = E_{Rx}-E_{Lx}.
\end{split}
\end{equation}
Here, $ E_{Lx} = \mathrm{diag}(1, 0, 0, \ldots, 0)$, $ E_{Rx} = \mathrm{diag}(0, 0, 0, \ldots, 1)$ pick out the left and right boundary values. The matrix $\mathrm{P}_x$ is symmetric and positive definite, and thus defines a discrete norm.  The matrix $Q_x$ is almost skew-symmetric.  

It is important to note that SBP operators can be constructed using Galerkin finite/spectral element approaches \cite{Gassner2013, KormannKronbichlerMuller}. 
 The one space dimensional  SBP operator $D_x$ can be extended to higher space dimensions using  the Kronecker product,  $\otimes$, defined by
\begin{Definition}[Kronecker Products]\label{def:section26}
Let A be an m-by-n matrix and B be a p-by-q matrix. Then $A \otimes B$, the Kronecker product of A and B, is the (mp)-by-(nq) matrix
\\
\begin{displaymath}
A \otimes B = \begin{pmatrix} 
a_{11}B&a_{12}B&\ldots&a_{1n}B \\
\hdotsfor[4]{4}\\
\hdotsfor[4]{4}\\
a_{m1}B&a_{m2}B&\hdots&a_{mn}B
\end{pmatrix}.
\end{displaymath}
\end{Definition}
The following properties hold for Kronecker products
\begin{enumerate}
\item Assume that the products AC and BD are well defined then\\
$(A \otimes B)  (C \otimes D) = (A C) \otimes (B  D)$,
\item If A and B are invertible, then $(A \otimes B)^{-1} = A^{-1} \otimes B^{-1}$,
\item $(A \otimes B)^{T} = A^{T} \otimes B^{T}$.
\end{enumerate}

Now consider the Maxwell's equation in a two dimensional rectangular domain.
To begin, we discretize  the spatial domain in the $x$- and $y$-directions using $N_x, N_y$ grid points with constant spatial steps $h_x, h_y$, respectively. 
A two dimensional scalar grid function, $\mathbf{u} = [u_{ij}]$, is stacked as a vector of length $N_xN_y$,
\[
 \mathbf{u} = \left({u}_{11}, {u}_{12}, \cdots, {u}_{N_xN_y}\right)^T.
 \]
Spatial derivatives are discretized using SBP operators,
and  the boundary conditions, \eqref{eq:wall_bc_y_0} and \eqref{eq:wall_bc_x},  are impose weakly using SAT \cite{CarpenterGottliebAbarbanel1994, Matt2003, Nordstrom2008}.  The SBP--SAT approximation of the Maxwell's equation \eqref{eq:Maxwell} subject the boundary conditions, \eqref{eq:wall_bc_y_0} and \eqref{eq:wall_bc_x},  is
\begin{subequations}\label{eq:Maxwell_WaveGuide_Discrete}
    \begin{alignat}{2}
     \frac{\mathrm{d}{\mathbf{E}_z}}{\mathrm{d t}} &= -\left(D_x \otimes I_y\right)\mathbf {H}_y + \left(I_x \otimes D_y\right)\mathbf {H}_x   \underbrace{-\alpha_x\left(\frac{1-R_x}{2}\left(\mathrm{P}_x^{-1}\left(E_{Rx}+E_{Lx}\right)\otimes I_y \right)\mathbf {E}_z  
 - \frac{1+R_x}{2}\left(\mathrm{P}_x^{-1}\left(E_{Rx}-E_{Lx}\right)\otimes I_y \right)\mathbf {H}_y\right)}_{\mathrm{SAT}_x} \notag \\
 & \underbrace{-\alpha_y \left( \frac{1-R_y}{2}\left(I_x \otimes \mathrm{P}_y^{-1}\left(E_{Ry}+E_{Ly}\right)\right)\mathbf {E}_z +\frac{1+R_y}{2}\left(I_x \otimes \mathrm{P}_y^{-1}\left(E_{Ry}-E_{Ly}\right)\right)\mathbf {H}_x\right)}_{\mathrm{SAT}_y}, \label{eq:Maxwell_WaveGuide_Discrete_1} \\
     \frac{\mathrm{d}{\mathbf{H}_y}}{\mathrm{d t}}  &= -\left(D_x \otimes I_y\right)\mathbf {E}_z \underbrace{+\theta_x\left(\frac{1-R_x}{2}\left(\mathrm{P}_x^{-1}\left(E_{Rx}-E_{Lx}\right)\otimes I_y \right)\mathbf {E}_z  
 - \frac{1+R_x}{2}\left(\mathrm{P}_x^{-1}\left(E_{Rx}+E_{Lx}\right)\otimes I_y \right)\mathbf {H}_y\right)}_{\mathrm{SAT}_x}, \label{eq:Maxwell_WaveGuide_Discrete_2} \\
     \frac{\mathrm{d}{\mathbf{H}_x}}{\mathrm{d t}}  &= \left(I_x \otimes D_y\right)\mathbf {E}_z \underbrace{-\theta_y \left( \frac{1-R_y}{2}\left(I_x \otimes \mathrm{P}_y^{-1}\left(E_{Ry}-E_{Ly}\right)\right)\mathbf {E}_z +\frac{1+R_y}{2}\left(I_x \otimes \mathrm{P}_y^{-1}\left(E_{Ry}+E_{Ly}\right)\right)\mathbf {H}_x\right)}_{\mathrm{SAT}_y}.\label{eq:Maxwell_WaveGuide_Discrete_3}
    \end{alignat}
  \end{subequations}
Here, the operators $I_x, I_y$ are identity matrices, and $\alpha_x, \alpha_y, \theta_x, \theta_y$ are  the penalty parameters chosen by requiring stability. Note that the boundary conditions   \eqref{eq:wall_bc_y}, \eqref{eq:wall_bc_x} are enforced weakly by the SAT--terms, $\mathrm{SAT}_x$ and $\mathrm{SAT}_y$, defined on the right hand side of equation
\eqref{eq:Maxwell_WaveGuide_Discrete_1}. This is   opposed to strong enforcement of boundary conditions,  where boundary conditions are  explicitly imposed by injection. 

By construction the SBP--SAT scheme \eqref{eq:Maxwell_WaveGuide_Discrete} is a consistent approximation of the IBVP,  \eqref{eq:Maxwell}, \eqref{eq:wall_bc_y_0} and \eqref{eq:wall_bc_x}.
The critical aspect of the scheme lies in  chosing penalties and ensuring numerical stability. Since the IBVP,  \eqref{eq:Maxwell}, \eqref{eq:wall_bc_y_0}, \eqref{eq:wall_bc_x}, does not support any temporal growth, it is imperative that penalties parameters must be chosen in a manner that does not allow any growth in time. The penalty parameters, $\alpha_x, \alpha_y, \theta_x, \theta_y$, are not necessarily unique. Often the choice of stable  penalty parameters are guided by constructing discrete energy estimates which, as far as possible, mimic the continuous energy estimates \eqref{eq:conserve}, \eqref{eq:strict}. In particular, penalty parameters are chosen such that the corresponding discrete boundary terms dissipate energy, at least, as fast as the continuous problem. However, from a practical point of view, for efficient explicit time integration, penalty parameters must also be chosen in order to avoid numerical stiffness. Note that if $|\alpha_x|, |\alpha_y|, |\theta_x|, |\theta_y| \gg 1$, then the ordinary differential equation (ODE)  \eqref{eq:Maxwell_WaveGuide_Discrete} is increasingly stiff.

 For more elaborate discussions and results on numerical boundary procedure, please see \cite{Bert_Gust, Bert_Gust1998,CarpenterGottliebAbarbanel1994, Matt2003, Nordstrom2008, SvardNordstrom2014}.  Note that the  weak boundary treatment  \eqref{eq:Maxwell_WaveGuide_Discrete_1}--\eqref{eq:Maxwell_WaveGuide_Discrete_3}  is often used for high order accurate discontinuous Galerkin (dG) finite element  methods \cite{Gassner2013, JanTim} and  also for pseudo spectral methods \cite{FunaroGottlieb}. 

 We will now chose the penalty parameters $\alpha_x, \alpha_y, \theta_x, \theta_y$ such that the semi--discrete approximation \eqref{eq:Maxwell_WaveGuide_Discrete_1}--\eqref{eq:Maxwell_WaveGuide_Discrete_3} is stable. 
To begin with, we define the discrete scalar product and the corresponding discrete  norm
\begin{equation}
\left\langle \mathbf{E}, \mathbf{H}\right\rangle_{\mathbf{P}_{xy}} = \mathbf{H}^T{\mathbf{P}_{xy}}\mathbf{E}, \quad \left\|  \mathbf{E} \right\|^2_{\mathbf{P}_{xy}} =  \left\langle \mathbf{E}, \mathbf{E}\right\rangle_{\mathbf{P}_{xy}} , \quad {\mathbf{P}_{xy}} = \mathrm{P}_x \otimes \mathrm{P}_y.
\end{equation}
Note that from the  SBP property \eqref{eq:SBPDx}  we have 
\begin{equation}\label{eq:SBPD2x}
\begin{split}
&D_x = \mathrm{P}_x^{-1}Q_x =  \mathrm{P}_x^{-1}\left(-\left(\mathrm{P}_x^{-1}Q_x\right)^{T} \mathrm{P}_x + E_{Rx}-E_{Lx}\right) = -\mathrm{P}_x^{-1}D_x^{T}\mathrm{P}_x + \mathrm{P}_x^{-1}\left(E_{Rx}-E_{Lx}\right).
\end{split}
\end{equation}
Thus, using the SBP property \eqref{eq:SBPDx}  we   can rewrite the  right hand side of  \eqref{eq:Maxwell_WaveGuide_Discrete_1} to obtain
\begin{subequations}\label{eq:Maxwell_WaveGuide_Discrete_sbp}
    \begin{alignat}{2}
     \frac{\mathrm{d}{\mathbf{E}_z}}{\mathrm{d t}} &= {\mathbf{P}_{xy}^{-1}}\left(D_x^T \otimes I_y\right){\mathbf{P}_{xy}}\mathbf {H}_y - {\mathbf{P}_{xy}^{-1}}\left(I_x \otimes D_y^T\right){\mathbf{P}_{xy}}\mathbf {H}_x     -\alpha_x\frac{1-R_x}{2}{\mathbf{P}_{xy}^{-1}}\left(\left(E_{Rx}+E_{Lx}\right)\otimes \mathrm{P}_y \right)\mathbf {E}_z  \notag \\
& - \frac{2-\alpha_x(1+R_x)}{2}{\mathbf{P}_{xy}^{-1}}\left(\left(E_{Rx}-E_{Lx}\right)\otimes \mathrm{P}_y \right)\mathbf {H}_y 
  -\alpha_y  \frac{1-R_y}{2}{\mathbf{P}_{xy}^{-1}}\left(\mathrm{P}_x \otimes \mathrm{I}_y\left(E_{Ry}+E_{Ly}\right)\right)\mathbf {E}_z  \notag \\
  & +\frac{2-\alpha_y(1+R_y)}{2}{\mathbf{P}_{xy}^{-1}}\left(\mathrm{P}_x \otimes \left(E_{Ry}-E_{Ly}\right)\right)\mathbf {H}_x, \label{eq:Maxwell_WaveGuide_Discrete_1_sbp} \\
     \frac{\mathrm{d}{\mathbf{H}_y}}{\mathrm{d t}}  &= -\left(D_x \otimes I_y\right)\mathbf {E}_z \underbrace{+\theta_x\left(\frac{1-R_x}{2}{\mathbf{P}_{xy}^{-1}}\left(\left(E_{Rx}-E_{Lx}\right)\otimes \mathrm{P}_y \right)\mathbf {E}_z  
 - \frac{1+R_x}{2}{\mathbf{P}_{xy}^{-1}}\left(\left(E_{Rx}+E_{Lx}\right)\otimes \mathrm{P}_y \right)\mathbf {H}_y\right)}_{\mathrm{SAT}_x}, \label{eq:Maxwell_WaveGuide_Discrete_2_sbp} \\
     \frac{\mathrm{d}{\mathbf{H}_x}}{\mathrm{d t}}  &= \left(I_x \otimes D_y\right)\mathbf {E}_z \underbrace{-\theta_y \left( \frac{1-R_y}{2}{\mathbf{P}_{xy}^{-1}}\left(\mathrm{P}_x \otimes \left(E_{Ry}-E_{Ly}\right)\right)\mathbf {E}_z +\frac{1+R_y}{2}{\mathbf{P}_{xy}^{-1}}\left(\mathrm{P}_x \otimes \left(E_{Ry}+E_{Ly}\right)\right)\mathbf {H}_x\right)}_{\mathrm{SAT}_y}.\label{eq:Maxwell_WaveGuide_Discrete_3_sbp}
    \end{alignat}
  \end{subequations}
By applying the energy method to  \eqref{eq:Maxwell_WaveGuide_Discrete_1_sbp}--\eqref{eq:Maxwell_WaveGuide_Discrete_3_sbp},  that is, we multiply \eqref{eq:Maxwell_WaveGuide_Discrete_1_sbp}, \eqref{eq:Maxwell_WaveGuide_Discrete_2_sbp}, \eqref{eq:Maxwell_WaveGuide_Discrete_3_sbp} from the left with $\mathbf{E}_z^T\mathbf{P}_{xy}$,  $\mathbf{H}_y^T\mathbf{P}_{xy}$, $\mathbf{H}_x^T\mathbf{P}_{xy}$, respectively, and add the transpose of the products, we obtain
\begin{equation}\label{eq:etimate_step1}
 \frac{d}{dt}\left(\left\|  \mathbf{E}_z\left(t\right) \right\|^2_{\mathbf{P}_{xy}} +  \left\|  \mathbf{H}_y\left(t\right) \right\|^2_{\mathbf{P}_{xy}}  +  \left\|  \mathbf{H}_x\left(t\right) \right\|^2_{\mathbf{P}_{xy}}\right) = - \textbf{BT}_s .
\end{equation}
Therefore, the discrete approximation \eqref{eq:Maxwell_WaveGuide_Discrete_1}--\eqref{eq:Maxwell_WaveGuide_Discrete_3} is asymptotically stable if the boundary terms $\textbf{BT}_s$ are nonnegative.
If $R_x \ne -1$ and $R_y \ne -1$, then we can choose 
\[
\alpha_x = \frac{2}{1+R_x}, \quad \alpha_y = \frac{2}{1+R_y}, \quad \theta_x = \frac{2\bar{\theta}_x}{1+R_x}, \quad \theta_y =  \frac{2\bar{\theta}_y}{1+R_y},
\]
to obtain
\begin{align*}
\textbf{BT}_s(t) &= 
\underbrace{- \sum_{j=1}^{N_y} \begin{pmatrix} \mathbf{E}_{z1j}\\  \mathbf{H}_{y1j} \end{pmatrix}^T {A}_x^{-}\begin{pmatrix} \mathbf{E}_{z1j}\\  \mathbf{H}_{y1j} \end{pmatrix}P_{yjj} - \sum_{j=1}^{N_y} \begin{pmatrix} \mathbf{E}_{zN_xj}\\  \mathbf{H}_{yN_xj} \end{pmatrix}^T  {A}_x^{+}\begin{pmatrix} \mathbf{E}_{zN_xj}\\  \mathbf{H}_{yN_xj} \end{pmatrix}P_{yjj}}_{\text{dissipative boundary terms in the x--axis}}\\
&\underbrace{- \sum_{i=1}^{N_x} \begin{pmatrix} \mathbf{E}_{zi1}\\  \mathbf{H}_{yi1} \end{pmatrix}^T  {A}_y^{-} \begin{pmatrix} \mathbf{E}_{zi1}\\  \mathbf{H}_{yi1} \end{pmatrix} P_{xii}- \sum_{i=1}^{N_x} \begin{pmatrix} \mathbf{E}_{ziN_y}\\  \mathbf{H}_{yiN_y} \end{pmatrix}^T  {A}_y^{+} \begin{pmatrix} \mathbf{E}_{ziN_y}\\  \mathbf{H}_{yiN_y} \end{pmatrix} P_{xii}}_{\text{dissipative boundary terms in the y--axis}}
\end{align*}
where
\begin{align*}
 {A}_x^{-} = \begin{pmatrix}
\gamma_x & -\frac{\bar{\theta}_x\gamma_x}{2} \\
-\frac{\bar{\theta}_x\gamma_x}{2} & \bar{\theta}_x
\end{pmatrix}, \quad
A_x^{+} = \begin{pmatrix}
\gamma_x & \frac{\bar{\theta}_x\gamma_x}{2} \\
\frac{\bar{\theta}_x\gamma_x}{2} & \bar{\theta}_x
\end{pmatrix}, \quad
 {A}_y^{-} = \begin{pmatrix}
\gamma_y & \frac{\bar{\theta}_y\gamma_y}{2} \\
\frac{\bar{\theta}_y\gamma_y}{2} & \bar{\theta}_y
\end{pmatrix},
\quad
A_y^{+} = \begin{pmatrix}
\gamma_y & -\frac{\bar{\theta}_y\gamma}{2} \\
-\frac{\bar{\theta}_y\gamma_y}{2} & \bar{\theta}_y
\end{pmatrix},
\end{align*}
are $2\times 2$ symmetric matrices and 
\[
\gamma_x = \frac{1-R_x}{1+R_x} \ge 0, \quad \gamma_y = \frac{1-R_y}{1+R_y} \ge 0.
\]
 To ensure numerical stability the symmetric matrices ${A}_x^{\pm}, {A}_y^{\pm}$ defined above must have nonnegative eigenvalues.
The eigenvalues of the matrices ${A}_x^{\pm}, {A}_y^{\pm}$ are
\[
\lambda_{x}^{\pm} = \frac{\left(\gamma_x+\bar{\theta}_x\right) \pm \sqrt{\left(\gamma_x+\bar{\theta}_x\right)^2 -4\gamma_x\bar{\theta}_x\left(1-\frac{\gamma_x\bar{\theta}_x}{4}\right) }}{2},
\quad
\lambda_{y}^{\pm} = \frac{\left(\gamma_y+\bar{\theta}_y\right) \pm \sqrt{\left(\gamma_y+\bar{\theta}_y\right)^2 -4\gamma_y\theta_y\left(1-\frac{\gamma_y\bar{\theta}_y}{4}\right) }}{2}.
\]
Clearly, the eigenvalues are real and nonnegative if
\[
0\le \bar{\theta}_x\le \frac{4}{\gamma_x}, \quad 0\le \bar{\theta}_y\le \frac{4}{\gamma_y}.
\]
By comparing $\mathbf{BT}_s$ to the continuous analog in the right hand side of \eqref{eq:strict} we can determine how fast  energy is dissipated by the numerical boundary treatment.
Note that with  $\bar{\theta}_x = \bar{\theta}_y = 0$, the boundary treatment dissipates energy as fast as the continuous problem, and we have
\[
\textbf{BT}_s  = 2\mathbf{E}_z^T \left(\left(E_{Rx}+E_{Lx}\right)\otimes \gamma_x\mathrm{P}_y  + \gamma_y \mathrm{P}_x\otimes \left(E_{Ry}+E_{Ly}\right)\right)\mathbf {E}_z,
\]
 thus, implying
\begin{equation}\label{eq:energy_estimate2}
\begin{split}
& \left\|  \mathbf{E}_z\left(t\right) \right\|^2_{\mathbf{P}_{xy}} +  \left\|  \mathbf{H}_y\left(t\right) \right\|^2_{\mathbf{P}_{xy}}  +  \left\|  \mathbf{H}_x\left(t\right) \right\|^2_{\mathbf{P}_{xy}} + 2\int_0^t\left(\mathbf{E}_z^T \left(\left(E_{Rx}+E_{Lx}\right)\otimes \gamma_x\mathrm{P}_y  + \gamma_y \mathrm{P}_x\otimes \left(E_{Ry}+E_{Ly}\right)\right)\mathbf {E}_z\right) dt'  \\
 &=   
 \left\| \mathbf{f}_1 \right\|^2_{\mathbf{P}_{xy}} +  \left\|  \mathbf{f}_2\right\|^2_{\mathbf{P}_{xy}}  +  \left\|  \mathbf{f}_3 \right\|^2_{\mathbf{P}_{xy}}.
 \end{split}
\end{equation}
Note that \eqref{eq:energy_estimate2} is completely analogous to \eqref{eq:strict}.
  Note also that standard finite element and nodal finite volume methods for Maxwell's equation \eqref{eq:Maxwell} are  constructed to satisfy  analogous discrete equations \eqref{eq:Maxwell_WaveGuide_Discrete_1_sbp}--\eqref{eq:Maxwell_WaveGuide_Discrete_3_sbp} and  energy estimates \eqref{eq:etimate_step1}  and \eqref{eq:energy_estimate2}.
  
If we choose $\theta_x > 0,$ $ \theta_y >  0$, then the numerical boundary treatment  will dissipate energy  faster than the continuous problem.

 More generally, for any $|R_x| \le 1$ and $|R_y| \le 1$, then we can choose 
\[
\alpha_x =  \alpha_y =  \theta_x = \theta_y = 1,
\]
to obtain
\begin{align*}
\textbf{BT}_s  &= \mathbf{E}_z^T \left(\left({1-R_x}\right)\left(E_{Rx}+E_{Lx}\right)\otimes \mathrm{P}_y  + \left(1-R_y\right) \mathrm{P}_x\otimes \left(E_{Ry}+E_{Ly}\right)\right)\mathbf {E}_z + \left({1+R_x}\right) \mathbf{H}_y^T \left(\left(E_{Rx}+E_{Lx}\right)\otimes \mathrm{P}_y\right)\mathbf {H}_y \\
&+  \left(1+R_y\right)\mathbf{H}_x^T \left( \mathrm{P}_x\otimes \left(E_{Ry}+E_{Ly}\right)\right)\mathbf {H}_x.
\end{align*}
Note that $\textbf{BT}_s \ge 0$ and 
  \begin{equation}\label{eq:energy_estimate3}
\begin{split}
& \left\|  \mathbf{E}_z\left(t\right) \right\|^2_{\mathbf{P}_{xy}} +  \left\|  \mathbf{H}_y\left(t\right) \right\|^2_{\mathbf{P}_{xy}}  +  \left\|  \mathbf{H}_x\left(t\right) \right\|^2_{\mathbf{P}_{xy}} + 2\int_0^t\textbf{BT}_s(t') dt'  \\
 &=   
 \left\| \mathbf{f}_1 \right\|^2_{\mathbf{P}_{xy}} +  \left\|  \mathbf{f}_2\right\|^2_{\mathbf{P}_{xy}}  +  \left\|  \mathbf{f}_3 \right\|^2_{\mathbf{P}_{xy}}.
 \end{split}
\end{equation}

 Thus, the discrete solutions of  \eqref{eq:Maxwell_WaveGuide_Discrete_1}--\eqref{eq:Maxwell_WaveGuide_Discrete_3} are uniformly bounded in time by the initial data. Therefore, the discrete approximation \eqref{eq:Maxwell_WaveGuide_Discrete_1}--\eqref{eq:Maxwell_WaveGuide_Discrete_3} with $\alpha_x =  \alpha_y =  \theta_x = \theta_y = 1$ is also asymptotically stable.

\begin{remark}
The SBP operators used in the present study  are high order accurate finite difference  operators  with  diagonal norms $\mathrm{P}_x$, where the boundary stencils are $r$th order accurate and the interior accuracy is $2r.$
\end{remark}

As will be seen in the next section, the PML will be placed in the $x$-direction. 
The primary goal of this paper is to study the numerical difficulties, such as stability and stiffness, arising from  the numerical boundary procedure \eqref{eq:Maxwell_WaveGuide_Discrete} when the PML is introduced. We will also investigate the optimal penalty parameters for the PML.
The PML will be terminated by the characteristic boundary condition, \eqref{eq:wall_bc_x} with $R_x = 0$. It is possible though to  terminate the PML with the PEC boundary condition, \eqref{eq:wall_bc_x} with $R_x = -1$.  However, the characteristic boundary condition at the external boundary have been shown to yield smaller spurious reflections, due to the modeling error, when combined with the PML than the PEC  backed PML \cite{JinChew}. 
In order to closely make comparisons with the results published in the literature \cite{Abarbanel1998, Abarbanel2002, Abarbanel2009} we will  impose the characteristic boundary condition, \eqref{eq:wall_bc_y_0} with $R_y = 0$, in the $y$--direction. Nevertheless, it is  possible though to use the PEC boundary condition, \eqref{eq:wall_bc_y_0} with $R_y = -1$, also in the $y$--direction.

\section{Termination of a computational domain using a PML}
Let us now consider  the solutions of  the Maxwell's equations  \eqref{eq:Maxwell}, in  the truncated domain $-x_0 \le x \le ~{x_0},  -y_0 \le y \le y_0 $, where $x = \pm x_0$ are  artificial boundaries introduced to limit the computational domain. In order to absorb outgoing waves at $x = \pm x_0$,  we have introduced two layers, of width $\delta > 0$,  having $x_0\le |x|~\le x_0+\delta$  in which the PML equations are solved.  Note that in the $x$-direction the  boundary condition is \eqref{eq:wall_bc_x} and   in the  $y$-direction   the boundary condition  is \eqref{eq:wall_bc_y_0}.  This setup can model an infinitely long electromagnetic waveguide truncated by the PML. 

In this section, we will derive PML models that truncate the computational domain in the $x$--direction  using two standard approaches,  the complex coordinate stretching technique \cite{Chew1994, DuKrApnum2013} and Berenger's splitting method \cite{Berenger1994}. The two approaches are mathematically analogous but yield two different formulations of the PML, the unsplit PML and  split--field PML, respectively.  The difference lies on the choice of variables in the time domain. As we will see below, by appropriate choice of variables the split-field \cite{Berenger1994}  can be formulated as the unsplit  \cite{DuKrApnum2013}. This shows that the general solutions of the two PML formulations are identical.

Let the Laplace transform of ${u}\left(x,y,t\right)$ be defined by
\begin{equation}
\widehat{{u}}(x,y,s)  = \int_0^{\infty}e^{-st}{u}\left(x,y,t\right)\text{dt},  \quad s = a + ib, \quad \Re{s} = a > 0.
\end{equation}
In order to derive the PML we  assume homogeneous initial data and take the Laplace transform  of the Maxwell's equations  \eqref{eq:Maxwell} in time, having  
\begin{equation}\label{eq:Maxwell_Laplace}
\begin{split}
 &s{\widehat{E}_z} = -\frac{\partial{\widehat{H}_y}}{\partial x} + \frac{\partial{\widehat{H}_x}}{\partial y},\\
 &s{\widehat{H}_y} = -\frac{\partial{\widehat{E}_z}}{\partial x}, \\
 &s{\widehat{H}_x} = \frac{\partial{\widehat{E}_z}}{\partial y}.\\
  \end{split}
  \end{equation}
Here $s$, with the positive real part $\Re{s} > 0$,  is the dual variable to time. Note that we have tacitly assumed homogeneous initial data.
\subsection{Unsplit PML and complex coordinate stretching}
 A general approach to deriving a PML was introduced in \cite{Chew1994}. 
 The main idea is to analytically continue the Maxwell's equation \eqref{eq:Maxwell_Laplace}  to a complex spatial coordinate system where all spatially oscillating solutions are turned into  exponentially decaying solutions. By using the complex coordinate stretching  technique it is straightforward to construct the PML for many wave equations. However, in order to localize the PML in time auxiliary variables are often introduced,  to avoid explicit convolution operations. 
 
 A PML  truncating the boundary in the $x$--direction can be derived by considering the, Laplace transformed, Maxwell's equation  \eqref{eq:Maxwell_Laplace}   and applying the complex coordinate transformation 
 \begin{equation}\label{eq:Complex_Transform}
\begin{split}
 \frac{\partial}{\partial x} \to \frac{1}{S_x}\frac{\partial}{\partial x},
  \end{split}
  \end{equation}
  we have
  \begin{subequations}\label{eq:Maxwell_Laplace_PML_Unsplit}
\begin{alignat}{2}
 &s{\widehat{E}_z} = -\frac{1}{S_x}\frac{\partial{\widehat{H}_y}}{\partial x} + \frac{\partial{\widehat{H}_x}}{\partial y},\label{eq:Maxwell_Laplace_PML_Unsplit_1}\\
 &s{\widehat{H}_y} = -\frac{1}{S_x}\frac{\partial{\widehat{E}_z}}{\partial x}, \label{eq:Maxwell_Laplace_PML_Unsplit_2}\\
 &s{\widehat{H}_x} = \frac{\partial{\widehat{E}_z}}{\partial y}, \label{eq:Maxwell_Laplace_PML_Unsplit_3}
  \end{alignat}
 \end{subequations}
 where $S_x$ is a PML complex metric. In this study, we will use the complex metric    $S_x = 1 + \sigma\left(x\right)/s$ corresponding to the standard PML model.  Here, $\sigma(x)\ge0$ is a positive real function that vanishes in $ |x|~\le x_0$ and $s$ is the dual variable to $t$. More complicated complex metrics can be derived by using a modal ansatz  \cite{Hagstrom2003, AppeloHagstromKreiss2006},  constructed to support only spatially decaying modes.
 
 The next step is to invert  the Laplace transforms in \eqref{eq:Maxwell_Laplace_PML_Unsplit} back to the time domain. To do this we  multiply equations \eqref{eq:Maxwell_Laplace_PML_Unsplit_1}--\eqref{eq:Maxwell_Laplace_PML_Unsplit_2} by  $S_x$ and 
 introduce the auxiliary variable
 \[
 s\widehat{E}_z^{(y)} = \frac{\partial{\widehat{H}_x}}{\partial y}.
 \]
 Inverting the Laplace transforms yields  the time--dependent  PML model
 \begin{subequations}\label{eq:Maxwell_PML_Split_2}
    \begin{alignat}{2}
 &\frac{\partial{E_z}}{\partial t} = -\frac{\partial{{H}_y}}{\partial x} + \frac{\partial{{H}_x}}{\partial y} + \sigma(x){{E}_z^{(y)}} - \sigma(x){{E}_z} , \label{eq:Maxwell_PML_Split_21}\\
 &\frac{\partial{H}_y}{\partial t} = -\frac{\partial{E_z}}{\partial x} - \sigma(x){{H}_y}, \label{eq:Maxwell_PML_Split_22} \\
 &\frac{\partial{H}_x}{\partial t} = \frac{\partial{E_z}}{\partial y} , \label{eq:Maxwell_PML_Split_23}\\
 &\frac{\partial{E}_z^{(y)}}{\partial t} =  \frac{\partial{{H}_x}}{\partial y} . \label{eq:Maxwell_PML_Split_24}
  \end{alignat}
  \end{subequations}
  Using instead the auxiliary variable 
   \[
 s\widehat{H}_x^{*} = \sigma(x)\frac{\partial{\widehat{H}_x}}{\partial y},
 \]
 and inverting the Laplace transform yields the time--dependent  PML model
  \begin{subequations}\label{eq:Maxwell_PML_WaveGuide}
    \begin{alignat}{2}
      \frac{\partial{ E_z}}{\partial t} &= -\frac{\partial{ H_y}}{\partial x} + \frac{\partial{ H_x}}{\partial y} + H^{*}_x - \sigma(x) E_z, \label{eq:Maxwell_PML_WaveGuide_1}\\
     \frac{\partial{ H_y}}{\partial t}  &= -\frac{\partial{ E_z}}{\partial x} - \sigma(x) H_y, \label{eq:Maxwell_PML_WaveGuide_2}\\
     \frac{\partial{ H_x}}{\partial t} &= \frac{\partial{ E_z}}{\partial y}, \label{eq:Maxwell_PML_WaveGuide_3}\\
    \frac{\partial{ H_x^*}}{\partial t} &=   \sigma(x)\frac{\partial{ H_x}}{\partial y}
    \label{eq:Maxwell_PML_WaveGuide_4}.
    \end{alignat}
  \end{subequations}
 
 We can also choose auxiliary variables to obtain the physically motivated PML used in  \cite{Abarbanel1998, Abarbanel2002, Abarbanel2009}.
 To begin, we consider \eqref{eq:Complex_Transform} and define the auxiliary variables 
 \[
 \widehat{\widetilde{H}}_x = S_x\widehat{H}_x, \quad s\widehat{P} = \frac{\sigma(x)}{S_x}\widehat{\widetilde{H}}_x.
 \]
 Inverting the Laplace transform and omitting the tilde signs  ( ${\widetilde{}}$ )   results in the physically motivated 
 unsplit PML of  \cite{Abarbanel1998, Abarbanel2002, Abarbanel2009}
\begin{subequations}\label{eq:Maxwell_PML_Ababarnel}
    \begin{alignat}{2}
      \frac{\partial{ E_z}}{\partial t} &= -\frac{\partial{ H_y}}{\partial x} + \frac{\partial{ H_x}}{\partial y}  - \sigma\left(x\right)E_z, \label{eq:Maxwell_PML_Ababarnel_1}
 \\
     \frac{\partial{ H_y}}{\partial t}  &= -\frac{\partial{ E_z}}{\partial x} - \sigma\left(x\right) H_y, \label{eq:Maxwell_PML_Ababarnel_2}\\
     \frac{\partial{ H_x}}{\partial t} &= \frac{\partial{ E_z}}{\partial y} + \sigma\left(x\right)\left( H_x -  P\right), \label{eq:Maxwell_PML_Ababarnel_3}\\
     \frac{\partial{ P}}{\partial t}  &=  \sigma\left(x\right)\left( H_x -  P\right).\label{eq:Maxwell_PML_Ababarnel_4}
    \end{alignat}
  \end{subequations}
Note that   the PML  \eqref{eq:Maxwell_PML_Ababarnel} was originally derived using physical arguments  \cite{Abarbanel1998, Abarbanel2002}.
There are certainly  other ways to choose auxiliary variables, see for example \cite{Abarbanel1998, Abarbanel2002, Abarbanel2009, AppeloHagstromKreiss2006, Gedney1996}. However, all resulting PML models can be shown to be linearly equivalent to \eqref{eq:Maxwell_PML_WaveGuide_1}--\eqref{eq:Maxwell_PML_WaveGuide_4}, \eqref{eq:Maxwell_PML_Ababarnel_1}--\eqref{eq:Maxwell_PML_Ababarnel_4} and \eqref{eq:Maxwell_PML_Split_21}--\eqref{eq:Maxwell_PML_Split_24}.
Since in the Laplace space the PML models \eqref{eq:Maxwell_PML_WaveGuide}, \eqref{eq:Maxwell_PML_Split_2},
 \eqref{eq:Maxwell_PML_Ababarnel} correspond to \eqref{eq:Maxwell_Laplace_PML_Unsplit}, a complex change of coordinates, it follows that the equations are perfectly matched  to the Maxwell's equation \eqref{eq:Maxwell}  by construction \cite{AppeloHagstromKreiss2006}.
\subsection{Berenger's  split-field PML}
Berenger's original idea \cite{Berenger1994} was to split the electric field and the corresponding equations  into two artificial tangential components before special lower order terms that simulate the absorption of waves are added. That is, consider the Maxwell's equation \eqref{eq:Maxwell_Laplace} in the Laplace space. Introduce the split variables ${\widehat{E}_z} = {\widehat{E}_z^{(x)}} + {\widehat{E}_z^{(y)}}$ and  equations  
\begin{subequations}\label{eq:Maxwell_Laplace_PML_Split}
    \begin{alignat}{2}
 &s{\widehat{E}_z^{(x)}} = -\frac{\partial{\widehat{H}_y}}{\partial x} - \sigma(x){\widehat{E}_z^{(x)}},\\
 &s{\widehat{H}_y} = -\frac{\partial{\left(\widehat{E}_z^{(x)} + \widehat{E}_z^{(y)}\right)}}{\partial x} - \sigma(x){\widehat{H}_y}, \\
 &s{\widehat{H}_x} = \frac{\partial{\left(\widehat{E}_z^{(x)} + \widehat{E}_z^{(y)}\right)}}{\partial y},\\
 &s{\widehat{E}_z^{(y)}} =  \frac{\partial{\widehat{H}_x}}{\partial y} ,\\
  \end{alignat}
  \end{subequations}
  where $\sigma(x)\ge0$ is a positive real function that vanishes in $ |x|~\le x_0$. Note that equation \eqref{eq:Maxwell_Laplace_PML_Split} can be derived by splitting the variables and equations, and then introducing the  complex coordinate transformation \eqref{eq:Complex_Transform}.
  This shows that the two approaches are mathematically analogous. However, unlike the unsplit PML \eqref{eq:Maxwell_Laplace_PML_Unsplit}, the split-field PML \eqref{eq:Maxwell_Laplace_PML_Split} introduces the additional degrees of  before the PML transformation is introduced.  It can be shown that the restriction of the general solution to the PML \eqref{eq:Maxwell_Laplace_PML_Split} in $ |x|~\le x_0$ coincides with the general solution to the Maxwell's equation \eqref{eq:Maxwell_Laplace}, see \cite{AppeloHagstromKreiss2006, HalpernPetit-BergezRauch2011}. Therefore,  the  PML equation \eqref{eq:Maxwell_Laplace_PML_Split} is perfectly matched to the Maxwell's equation \eqref{eq:Maxwell_Laplace} by construction.
  By inverting the Laplace transforms in \eqref{eq:Maxwell_Laplace_PML_Split} we obtain the time-dependent split-field PML  \cite{Berenger1994}
\begin{subequations}\label{eq:Maxwell_PML_Split_1}
    \begin{alignat}{2}
 &\frac{\partial{E_z^{(x)}}}{\partial t} = -\frac{\partial{{H}_y}}{\partial x} - \sigma(x){{E}_z^{(x)}}, \label{eq:Maxwell_PML_Split_11}\\
 &\frac{\partial{H}_y}{\partial t} = -\frac{\partial{\left({E}_z^{(x)} + {E}_z^{(y)}\right)}}{\partial x} - \sigma(x){{H}_y}, \label{eq:Maxwell_PML_Split_13}\\
 &\frac{\partial{H}_x}{\partial t} = \frac{\partial{\left({E}_z^{(x)} + {E}_z^{(y)}\right)}}{\partial y} \label{eq:Maxwell_PML_Split_14},\\
 &\frac{\partial{E}_z^{(y)}}{\partial t} =  \frac{\partial{{H}_x}}{\partial y}  \label{eq:Maxwell_PML_Split_12}.
  \end{alignat}
  \end{subequations}
  
  We will show  that the split-field PML \eqref{eq:Maxwell_PML_Split_1} can be reformulated as  the unsplit PML models \eqref{eq:Maxwell_PML_WaveGuide}, \eqref{eq:Maxwell_PML_Split_2}.
  Following \cite{LionsMetralVacus2001}, it is possible to re-introduce the electric field $E_z$ in the PML equations while eliminating one of the split variables. By summing \eqref{eq:Maxwell_PML_Split_11}, \eqref{eq:Maxwell_PML_Split_12}  together and using the identity ${{E}_z} =  {{E}_z^{(x)}} + {{E}_z^{(y)}} \implies {{E}_z^{(x)}} = {{E}_z} - {{E}_z^{(y)}}$ we again obtain the PML equation \eqref{eq:Maxwell_PML_Split_2}.
 We can also rewrite the equation \eqref{eq:Maxwell_PML_Split_2} in the form \eqref{eq:Maxwell_PML_WaveGuide}.
 Multiplying equation \eqref{eq:Maxwell_PML_Split_24} by $\sigma\left(x\right)$ and introducing $H_x^* =  \sigma\left(x\right){{E}_z^{(y)}}$
 we obtain  \eqref{eq:Maxwell_PML_Split_2}.

 Note that in \cite{DuKrApnum2013} the Berenger's PML \eqref{eq:Maxwell_PML_Split_1} and the modal unsplit PML,  in the form                          
\eqref{eq:Maxwell_PML_WaveGuide}, gave identical numerical results. The turning point is in efficient numerical implementation. Implementing the split--field  PML also requires the splitting of the electric field in the interior. The unsplit PML appears as a lower order modification of the Maxwell equations and does not require such unphysical splitting of the field variables.
 
 In the Laplace space, the PML models \eqref{eq:Maxwell_PML_Split_11}--\eqref{eq:Maxwell_PML_Split_14}, \eqref{eq:Maxwell_PML_WaveGuide_1}--\eqref{eq:Maxwell_PML_WaveGuide_4}, \eqref{eq:Maxwell_PML_Ababarnel_1}--\eqref{eq:Maxwell_PML_Ababarnel_4} and \eqref{eq:Maxwell_PML_Split_21}--\eqref{eq:Maxwell_PML_Split_24}   are mathematically equivalent \cite{DuKrApnum2013}. The difference in the time domain lies on the choice of auxiliary variables. By standard definitions  the physically motivated PML \eqref{eq:Maxwell_PML_Ababarnel_1}--\eqref{eq:Maxwell_PML_Ababarnel_4} satisfies the definition of a strongly hyperbolic system.   On the order hand, also by standard definitions the modal PMLs \eqref{eq:Maxwell_PML_WaveGuide_1}--\eqref{eq:Maxwell_PML_WaveGuide_4}, \eqref{eq:Maxwell_PML_Split_21}--\eqref{eq:Maxwell_PML_Split_24}  and the split--field PML  \eqref{eq:Maxwell_PML_Split_11}--\eqref{eq:Maxwell_PML_Split_14} are weakly hyperbolic. It is a bit strange though to say that  the strongly hyperbolic system \eqref{eq:Maxwell_PML_Ababarnel_1}--\eqref{eq:Maxwell_PML_Ababarnel_4} is  equivalent to the weakly hyperbolic problems  \eqref{eq:Maxwell_PML_WaveGuide_1}--\eqref{eq:Maxwell_PML_WaveGuide_4}, \eqref{eq:Maxwell_PML_Split_21}--\eqref{eq:Maxwell_PML_Split_24}, \eqref{eq:Maxwell_PML_Split_11}--\eqref{eq:Maxwell_PML_Split_14}. 
However, as we will see in the next section, it is possible to rewrite the PMLs \eqref{eq:Maxwell_PML_WaveGuide_1}--\eqref{eq:Maxwell_PML_WaveGuide_4}, \eqref{eq:Maxwell_PML_Split_21}--\eqref{eq:Maxwell_PML_Split_24}, \eqref{eq:Maxwell_PML_Split_11}--\eqref{eq:Maxwell_PML_Split_14} as a second order system and derive an energy estimate  for the solutions in a certain Sobolev norm. Thus, showing  well--posedness  also for the  time--dependent PML  models \eqref{eq:Maxwell_PML_WaveGuide_1}--\eqref{eq:Maxwell_PML_WaveGuide_4}, \eqref{eq:Maxwell_PML_Split_21}--\eqref{eq:Maxwell_PML_Split_24}, \eqref{eq:Maxwell_PML_Split_11}--\eqref{eq:Maxwell_PML_Split_14}.
 \subsection{Some instructive numerical examples}\label{sect:num_example}
 In this study, we will investigate discrete stability of  three PMLs resulting from the modifications of   the discrete equations  \eqref{eq:Maxwell_WaveGuide_Discrete_1}--\eqref{eq:Maxwell_WaveGuide_Discrete_3}.
  The  first  candidate is an unsplit modal PML  recently derived in \cite{DuKrApnum2013} and  given by \eqref{eq:Maxwell_PML_WaveGuide_1}--\eqref{eq:Maxwell_PML_WaveGuide_4}. 
Note that the results obtained for the PML \eqref{eq:Maxwell_PML_WaveGuide_1}--\eqref{eq:Maxwell_PML_WaveGuide_4} can be easily extended to  the  PML model \eqref{eq:Maxwell_PML_Split_21}--\eqref{eq:Maxwell_PML_Split_24}. We will not consider the  PML model \eqref{eq:Maxwell_PML_Split_21}--\eqref{eq:Maxwell_PML_Split_24} in this study.
 The second   is the so called  physically motivated unsplit PML of  \cite{Abarbanel1998, Abarbanel2002, Abarbanel2009} defined by \eqref{eq:Maxwell_PML_Ababarnel_1}--\eqref{eq:Maxwell_PML_Ababarnel_4},
and the third  candidate is the classical  Berenger's \cite{Berenger1994}  split--field PML  \eqref{eq:Maxwell_PML_Split_11}--\eqref{eq:Maxwell_PML_Split_14}.

Here, the damping function is given by the third degree monomial
\begin{equation}\label{eq:Damping_x}
\begin{split}
&\sigma(x) = \left \{
\begin{array}{rl}
0 \quad {}  \quad {}& \text{if} \quad |x| \le x_0,\\
d_0\Big(\frac{|x|-x_0}{\delta}\Big)^3  & \text{if}  \quad |x| \ge x_0.
\end{array} \right.
\end{split}
\end{equation}
  The damping coefficient  is $d_0 =  4/\left(2\times \delta \right)\log\left(1/\text{tol}\right)$, which allows a relative modeling error of magnitude $\sim \text{tol}$ from the outer boundaries, see \cite{DuKrApnum2013, DuAcoustic,  SjPe}. Note that $\text{tol}$ is a user input parameter.

We will now perform some instructive numerical experiments which are the motivations of  the analysis in the coming sections. We consider a typical test problem as in \cite{Abarbanel2002, Abarbanel2009} with $R_x = R_y = 0$ in \eqref{eq:wall_bc_y_0}, \eqref{eq:wall_bc_x} and we set $x_0 = 50$, $y_0 = 50$, and the PML width is $\delta = 10.$ We discretize the domain with a uniform spatial step $h_x=h_y =h = 1$. 
The PML equations \eqref{eq:Maxwell_PML_WaveGuide_1}--\eqref{eq:Maxwell_PML_WaveGuide_4}  are discretized straightforwardly. That is,  we replace the spatial derivative in \eqref{eq:Maxwell_PML_WaveGuide_4} with a finite difference SBP operator and append  the  discrete auxiliary functions  to the discrete equations \eqref{eq:Maxwell_WaveGuide_Discrete_1}--\eqref{eq:Maxwell_WaveGuide_Discrete_3}, we have
\begin{subequations}\label{eq:Maxwell_PML1_WaveGuide}
    \begin{alignat}{2}
       \frac{\mathrm{d}{\mathbf{E}_z}}{\mathrm{d t}} &= -\left(D_x \otimes I_y\right)\mathbf {H}_y + \left(I_x \otimes D_y\right)\mathbf {H}_x +\mathbf {H}_x^* - \mathbf{\sigma}\mathbf{E}_z \notag \\
    &   \underbrace{-\alpha_x\left(\frac{1-R_x}{2}\left(\mathrm{P}_x^{-1}\left(E_{Rx}+E_{Lx}\right)\otimes I_y \right)\mathbf {E}_z  
 - \frac{1+R_x}{2}\left(\mathrm{P}_x^{-1}\left(E_{Rx}-E_{Lx}\right)\otimes I_y \right)\mathbf {H}_y\right)}_{\mathrm{SAT}_x} \notag \\
 & \underbrace{-\alpha_y \left( \frac{1-R_y}{2}\left(I_x \otimes \mathrm{P}_y^{-1}\left(E_{Ry}+E_{Ly}\right)\right)\mathbf {E}_z +\frac{1+R_y}{2}\left(I_x \otimes \mathrm{P}_y^{-1}\left(E_{Ry}-E_{Ly}\right)\right)\mathbf {H}_x\right)}_{\mathrm{SAT}_y},
\label{eq:Maxwell_PML_WaveGuide_Discrete_11} \\
      \frac{\mathrm{d}{\mathbf{H}_y}}{\mathrm{d t}}  &= -\left(D_x \otimes I_y\right)\mathbf {E}_z- \mathbf{\sigma}\mathbf {H}_y \underbrace{+\theta_x\left(\frac{1-R_x}{2}\left(\mathrm{P}_x^{-1}\left(E_{Rx}-E_{Lx}\right)\otimes I_y \right)\mathbf {E}_z  
 - \frac{1+R_x}{2}\left(\mathrm{P}_x^{-1}\left(E_{Rx}+E_{Lx}\right)\otimes I_y \right)\mathbf {H}_y\right)}_{\mathrm{SAT}_x},\label{eq:Maxwell_PML_WaveGuide_Discrete_21} \\
     \frac{\mathrm{d}{\mathbf{H}_x}}{\mathrm{d t}}  &= \left(I_x \otimes D_y\right)\mathbf {E}_z \underbrace{-\theta_y \left( \frac{1-R_y}{2}\left(I_x \otimes \mathrm{P}_y^{-1}\left(E_{Ry}-E_{Ly}\right)\right)\mathbf {E}_z +\frac{1+R_y}{2}\left(I_x \otimes \mathrm{P}_y^{-1}\left(E_{Ry}+E_{Ly}\right)\right)\mathbf {H}_x\right)}_{\mathrm{SAT}_y}, \label{eq:Maxwell_PML_WaveGuide_Discrete_31}\\ 
   \frac{\mathrm{d}{\mathbf{H}_x^*}}{\mathrm{d t}}   &= \sigma \left(I_x \otimes D_y\right)  \mathbf {H}_x .\label{eq:Maxwell_PML_WaveGuide_Discrete_41}
    \end{alignat}
  \end{subequations}

 Note that for the PML model  \eqref{eq:Maxwell_PML_Ababarnel_1}--\eqref{eq:Maxwell_PML_Ababarnel_4},  the auxiliary differential equation \eqref{eq:Maxwell_PML_Ababarnel_4} is governed by an ordinary differential equation. The corresponding discrete PML is 
\begin{subequations}\label{eq:Maxwell_PML2_WaveGuide}
    \begin{alignat}{2}
       \frac{\mathrm{d}{\mathbf{E}_z}}{\mathrm{d t}} &= -\left(D_x \otimes I_y\right)\mathbf {H}_y + \left(I_x \otimes D_y\right)\mathbf {H}_x - \mathbf{\sigma}\mathbf{E}_z \notag \\
    &   \underbrace{-\alpha_x\left(\frac{1-R_x}{2}\left(\mathrm{P}_x^{-1}\left(E_{Rx}+E_{Lx}\right)\otimes I_y \right)\mathbf {E}_z  
 - \frac{1+R_x}{2}\left(\mathrm{P}_x^{-1}\left(E_{Rx}-E_{Lx}\right)\otimes I_y \right)\mathbf {H}_y\right)}_{\mathrm{SAT}_x} \notag \\
 & \underbrace{-\alpha_y \left( \frac{1-R_y}{2}\left(I_x \otimes \mathrm{P}_y^{-1}\left(E_{Ry}+E_{Ly}\right)\right)\mathbf {E}_z +\frac{1+R_y}{2}\left(I_x \otimes \mathrm{P}_y^{-1}\left(E_{Ry}-E_{Ly}\right)\right)\mathbf {H}_x\right)}_{\mathrm{SAT}_y}
  \label{eq:Maxwell_PML_Ababarnel_Discrete_11}\\
      \frac{\mathrm{d}{\mathbf{H}_y}}{\mathrm{d t}}  &= -\left(D_x \otimes I_y\right)\mathbf {E}_z- \mathbf{\sigma}\mathbf {H}_y \underbrace{+\theta_x\left(\frac{1-R_x}{2}\left(\mathrm{P}_x^{-1}\left(E_{Rx}-E_{Lx}\right)\otimes I_y \right)\mathbf {E}_z  
 - \frac{1+R_x}{2}\left(\mathrm{P}_x^{-1}\left(E_{Rx}+E_{Lx}\right)\otimes I_y \right)\mathbf {H}_y\right)}_{\mathrm{SAT}_x},\label{eq:Maxwell_PML_Ababarnel_Discrete_21} \\
     \frac{\mathrm{d}{\mathbf{H}_x}}{\mathrm{d t}}  &= \left(I_x \otimes D_y\right)\mathbf {E}_z +\sigma \left(\mathbf {H}_x - \mathbf {P}\right) \underbrace{-\theta_y \left( \frac{1-R_y}{2}\left(I_x \otimes \mathrm{P}_y^{-1}\left(E_{Ry}-E_{Ly}\right)\right)\mathbf {E}_z +\frac{1+R_y}{2}\left(I_x \otimes \mathrm{P}_y^{-1}\left(E_{Ry}+E_{Ly}\right)\right)\mathbf {H}_x\right)}_{\mathrm{SAT}_y}, \label{eq:Maxwell_PML_Ababarnel_Discrete_31} \\ 
   \frac{\mathrm{d}{\mathbf{P}}}{\mathrm{d t}}   &= \sigma \left(\mathbf {H}_x - \mathbf {P}\right).\label{eq:Maxwell_PML_Ababarnel_Discrete_41}
    \end{alignat}
  \end{subequations}
The corresponding semi-discretization for the split--field PML   \eqref{eq:Maxwell_PML_Split_11}--\eqref{eq:Maxwell_PML_Split_14} is 
\begin{subequations}\label{eq:Maxwell_PML3_WaveGuide}
    \begin{alignat}{2}
       \frac{\mathrm{d}{\mathbf{E}_z^{(x)}}}{\mathrm{d t}} &= -\left(D_x \otimes I_y\right)\mathbf {H}_y  - \mathbf{\sigma}\mathbf{E}_z^{(x)}
       \notag \\
    &   \underbrace{-\alpha_x\left(\frac{1-R_x}{2}\left(\mathrm{P}_x^{-1}\left(E_{Rx}+E_{Lx}\right)\otimes I_y \right)\left(\mathbf {E}_z^{(x)} + \mathbf {E}_z^{(y)}  \right)  
 - \frac{1+R_x}{2}\left(\mathrm{P}_x^{-1}\left(E_{Rx}-E_{Lx}\right)\otimes I_y \right)\mathbf {H}_y\right)}_{\mathrm{SAT}_x} \notag \\
 & \underbrace{-\alpha_y \left( \frac{1-R_y}{2}\left(I_x \otimes \mathrm{P}_y^{-1}\left(E_{Ry}+E_{Ly}\right)\right)\left(\mathbf {E}_z^{(x)} + \mathbf {E}_z^{(y)}  \right) +\frac{1+R_y}{2}\left(I_x \otimes \mathrm{P}_y^{-1}\left(E_{Ry}-E_{Ly}\right)\right)\mathbf {H}_x\right)}_{\mathrm{SAT}_y}
 %
 %
  \label{eq:Maxwell_PML_Split_Discrete_11}\\
      \frac{\mathrm{d}{\mathbf{H}_y}}{\mathrm{d t}}  &= -\left(D_x \otimes I_y\right)\left(\mathbf {E}_z^{(x)} + \mathbf {E}_z^{(y)}  \right)- \mathbf{\sigma}\mathbf {H}_y \notag\\
      &\underbrace{+\theta_x\left(\frac{1-R_x}{2}\left(\mathrm{P}_x^{-1}\left(E_{Rx}-E_{Lx}\right)\otimes I_y \right)\left(\mathbf {E}_z^{(x)} + \mathbf {E}_z^{(y)}  \right)  
 - \frac{1+R_x}{2}\left(\mathrm{P}_x^{-1}\left(E_{Rx}+E_{Lx}\right)\otimes I_y \right)\mathbf {H}_y\right)}_{\mathrm{SAT}_x}, \label{eq:Maxwell_PML_Split_Discrete_31} \\
     \frac{\mathrm{d}{\mathbf{H}_x}}{\mathrm{d t}}  &= \left(I_x \otimes D_y\right)\left(\mathbf {E}_z^{(x)} + \mathbf {E}_z^{(y)}  \right) \underbrace{-\theta_y \left( \frac{1-R_y}{2}\left(I_x \otimes \mathrm{P}_y^{-1}\left(E_{Ry}-E_{Ly}\right)\right)\left(\mathbf {E}_z^{(x)} + \mathbf {E}_z^{(y)}  \right) +\frac{1+R_y}{2}\left(I_x \otimes \mathrm{P}_y^{-1}\left(E_{Ry}+E_{Ly}\right)\right)\mathbf {H}_x\right)}_{\mathrm{SAT}_y},\label{eq:Maxwell_PML_Split_Discrete_41} \\
     \frac{\mathrm{d}{\mathbf{E}_z^{(y)}}}{\mathrm{d t}} &=  \left(I_x \otimes D_y\right)\mathbf {H}_x. \label{eq:Maxwell_PML_Split_Discrete_21}
    \end{alignat}
  \end{subequations}
The penalty parameters are
\[
\alpha_x = {2}, \quad \alpha_y = {2}, \quad \theta_x = \theta_y = 0.
\]
Similarly, as in the continuous case, using  the identity $\mathbf{E}_z  = \mathbf{E}_z^{(x)} + \mathbf{E}_z^{(y)} \implies \mathbf{E}_z^{(x)}  = \mathbf{E}_z - \mathbf{E}_z^{(y)}$ we can eliminate the auxiliary variable $\mathbf{E}_z^{(x)}$ and rewrite the split--field PML \eqref{eq:Maxwell_PML_Split_Discrete_11}--\eqref{eq:Maxwell_PML_Split_Discrete_21} exactly as  the discrete modal PML  \eqref{eq:Maxwell_PML_WaveGuide_Discrete_11}--\eqref{eq:Maxwell_PML_WaveGuide_Discrete_41}. Therefore, we expect  the PML models \eqref{eq:Maxwell_PML_Split_Discrete_11}--\eqref{eq:Maxwell_PML_Split_Discrete_21} and \eqref{eq:Maxwell_PML_WaveGuide_Discrete_11}--\eqref{eq:Maxwell_PML_WaveGuide_Discrete_41} to have identical stability properties.

Note that when the damping vanishes  $\sigma = 0$ we can eliminate the auxiliary  variables in \eqref{eq:Maxwell_PML_WaveGuide_Discrete_11}--\eqref{eq:Maxwell_PML_WaveGuide_Discrete_41},  \eqref{eq:Maxwell_PML_Ababarnel_Discrete_11}--\eqref{eq:Maxwell_PML_Ababarnel_Discrete_41}  and  the split variable in
\eqref{eq:Maxwell_PML_Split_Discrete_11}--\eqref{eq:Maxwell_PML_Split_Discrete_21} to obtain \eqref{eq:Maxwell_WaveGuide_Discrete_1}--\eqref{eq:Maxwell_WaveGuide_Discrete_3}.  
Since the discrete approximation \eqref{eq:Maxwell_WaveGuide_Discrete_1}--\eqref{eq:Maxwell_WaveGuide_Discrete_3} of the interior, when $\sigma = 0$,  is stable, we hope that the discrete approximations  \eqref{eq:Maxwell_PML_WaveGuide_Discrete_11}--\eqref{eq:Maxwell_PML_WaveGuide_Discrete_41},  \eqref{eq:Maxwell_PML_Ababarnel_Discrete_11}--\eqref{eq:Maxwell_PML_Ababarnel_Discrete_41}  and  \eqref{eq:Maxwell_PML_Split_Discrete_11}--\eqref{eq:Maxwell_PML_Split_Discrete_21}, with $\sigma \ge 0$, corresponding to the PML  are also stable. 

We set the  smooth initial data  
\begin{equation}\label{eq:Initial_Data}
 E_z(x, y, 0) = \exp\left(-\left(x^2 + y^2\right)/9\right),
\end{equation}
for the electric field and use zero initial data for all other variables.  High order accurate  SBP operators of interior order of accuracy 2, 4, 6 are used to approximate spatial derivatives. We advance the solutions in time for the three PML models using the classical fourth order accurate Runge--Kutta scheme. The final time is  $t = 5000$ and  the time step is $\mathrm{dt} = 0.4h$. Here, the modeling error is $\mathrm{tol} = 10^{-4}$. 
The time histories of the discrete $l_2$-norm of the electric field are recorded in figures \ref{fig:Model_Standard_OurPML},
\ref{fig:Model_Standard_Ababarnel} and \ref{fig:Model_Split_Field} for the discrete PML models,  \eqref{eq:Maxwell_PML_WaveGuide_Discrete_11}--\eqref{eq:Maxwell_PML_WaveGuide_Discrete_41} and  \eqref{eq:Maxwell_PML_Ababarnel_Discrete_11}--\eqref{eq:Maxwell_PML_Ababarnel_Discrete_41}, repsectively.

Clearly, from figures \ref{fig:Model_Standard_OurPML}, \ref{fig:Model_Standard_Ababarnel} and \ref{fig:Model_Split_Field},  the three discrete PML models  \eqref{eq:Maxwell_PML_WaveGuide_Discrete_11}--\eqref{eq:Maxwell_PML_WaveGuide_Discrete_41}, \eqref{eq:Maxwell_PML_Ababarnel_Discrete_11}--\eqref{eq:Maxwell_PML_Ababarnel_Discrete_41} and \eqref{eq:Maxwell_PML_Split_Discrete_11}--\eqref{eq:Maxwell_PML_Split_Discrete_21}  are unstable for SBP operators of accuracy 4 and 6. Note  that the numerical solution  for higher order  accurate schemes are more sensitive to numerical instability.   For the chosen simulation duration, $t = 5000$ (12 500 time steps), we did not observe any growth for the second order accurate case, see also figures \ref{fig:Model_Standard_OurPML}, \ref{fig:Model_Standard_Ababarnel} and 
\ref{fig:Model_Split_Field}.  It is also important to see that the stability behavior of discrete modal PML \eqref{eq:Maxwell_PML_WaveGuide_Discrete_11}--\eqref{eq:Maxwell_PML_WaveGuide_Discrete_21}  and discrete split--field PML \eqref{eq:Maxwell_PML_Split_Discrete_11}--\eqref{eq:Maxwell_PML_Split_Discrete_21} are identical. However, when  the absorption function is absent ($\sigma = 0$) the solutions, see figure \ref{fig:Model_Standard}, decay throughout the simulation for all orders of accuracy. This is consistent with the theory, since the numerical scheme, \eqref{eq:Maxwell_WaveGuide_Discrete_1}--\eqref{eq:Maxwell_WaveGuide_Discrete_3}, of the interior is provably stable. Note also when $\sigma = 0$, the corresponding ODE \eqref{eq:Maxwell_WaveGuide_Discrete_1}--\eqref{eq:Maxwell_WaveGuide_Discrete_3} is not stiff.

\begin{figure} [h!]
 \centering
\subfigure[Discrete modal PML  \eqref{eq:Maxwell_PML_WaveGuide_Discrete_11}--\eqref{eq:Maxwell_PML_WaveGuide_Discrete_41}. ]{\includegraphics[width=0.49\linewidth]{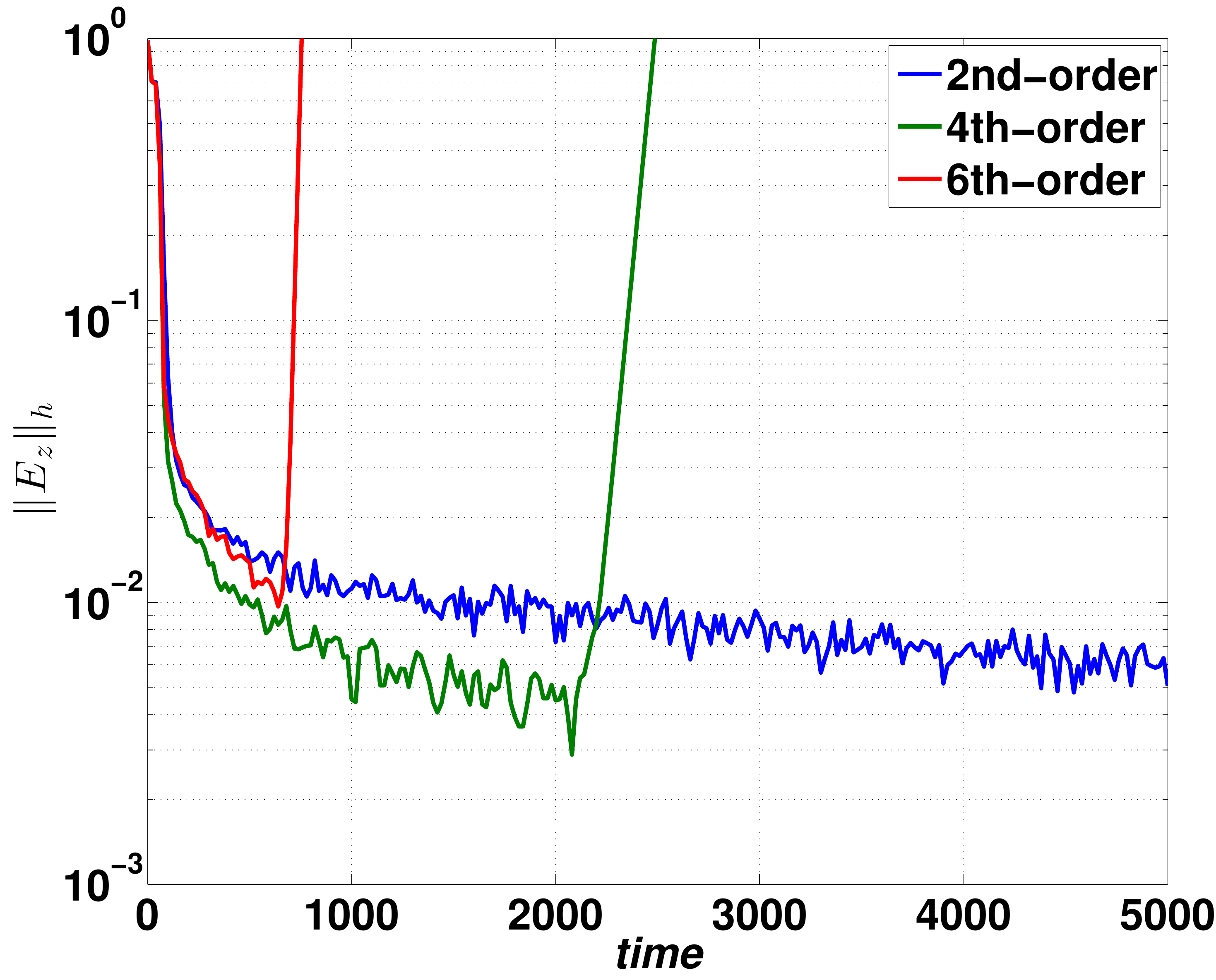}\label{fig:Model_Standard_OurPML}}
\subfigure[Discrete physically motivated PML \eqref{eq:Maxwell_PML_Ababarnel_Discrete_11}--\eqref{eq:Maxwell_PML_Ababarnel_Discrete_41}. ]{\includegraphics[width=0.49\linewidth]{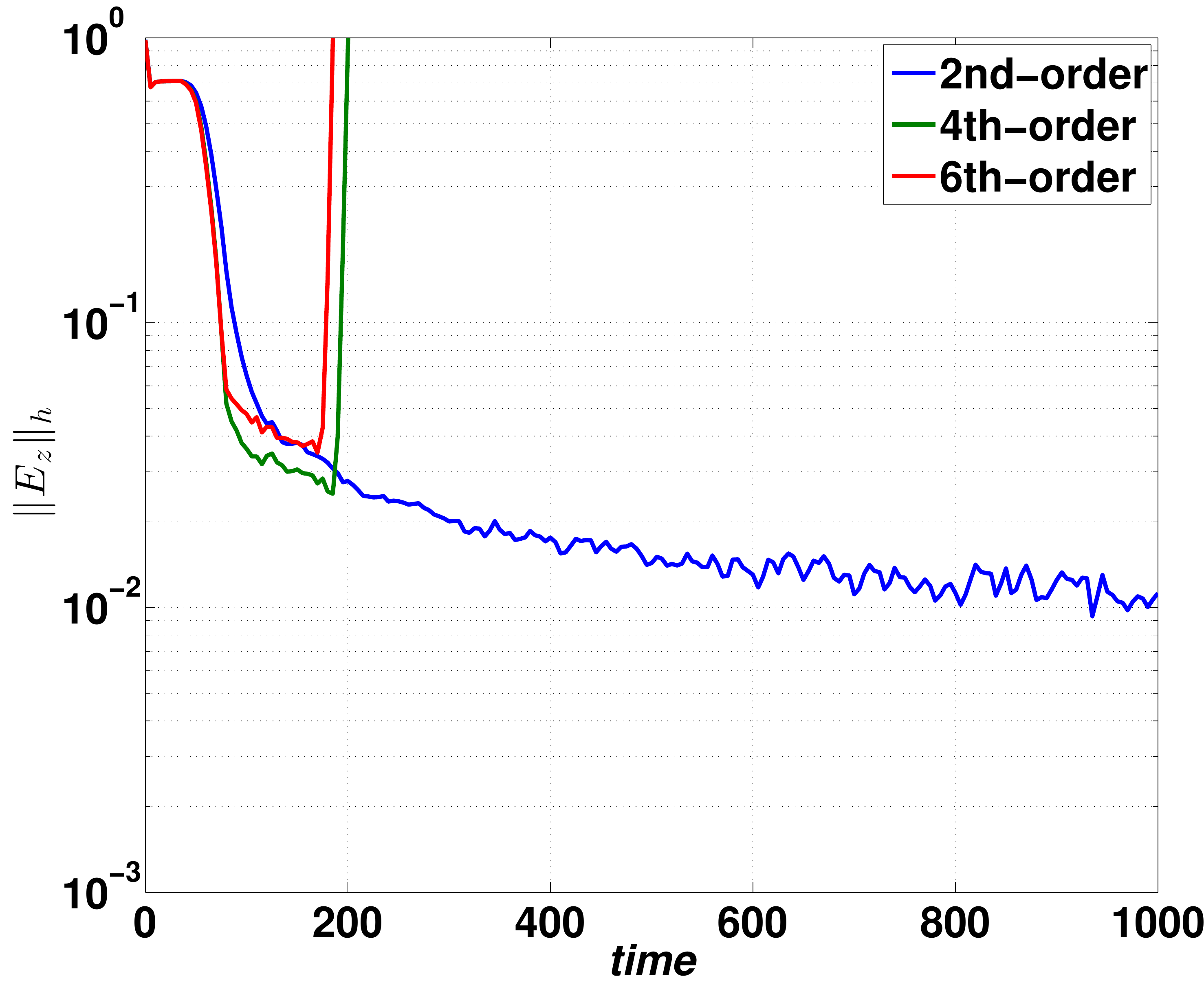}\label{fig:Model_Standard_Ababarnel}}
\subfigure[Discrete split--field PML \eqref{eq:Maxwell_PML_Split_Discrete_11}--\eqref{eq:Maxwell_PML_Split_Discrete_21}. ]{\includegraphics[width=0.5\linewidth, height=0.41\linewidth]{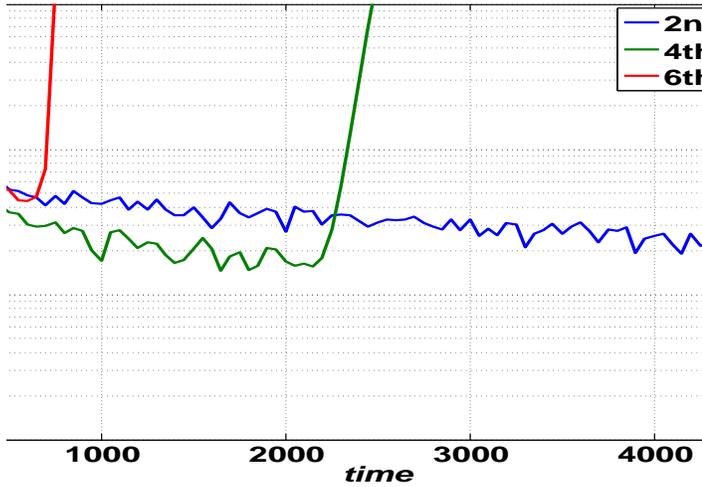}\label{fig:Model_Split_Field}}
\subfigure[No PML $(\sigma = 0)$.]{\includegraphics[width=0.49\linewidth]{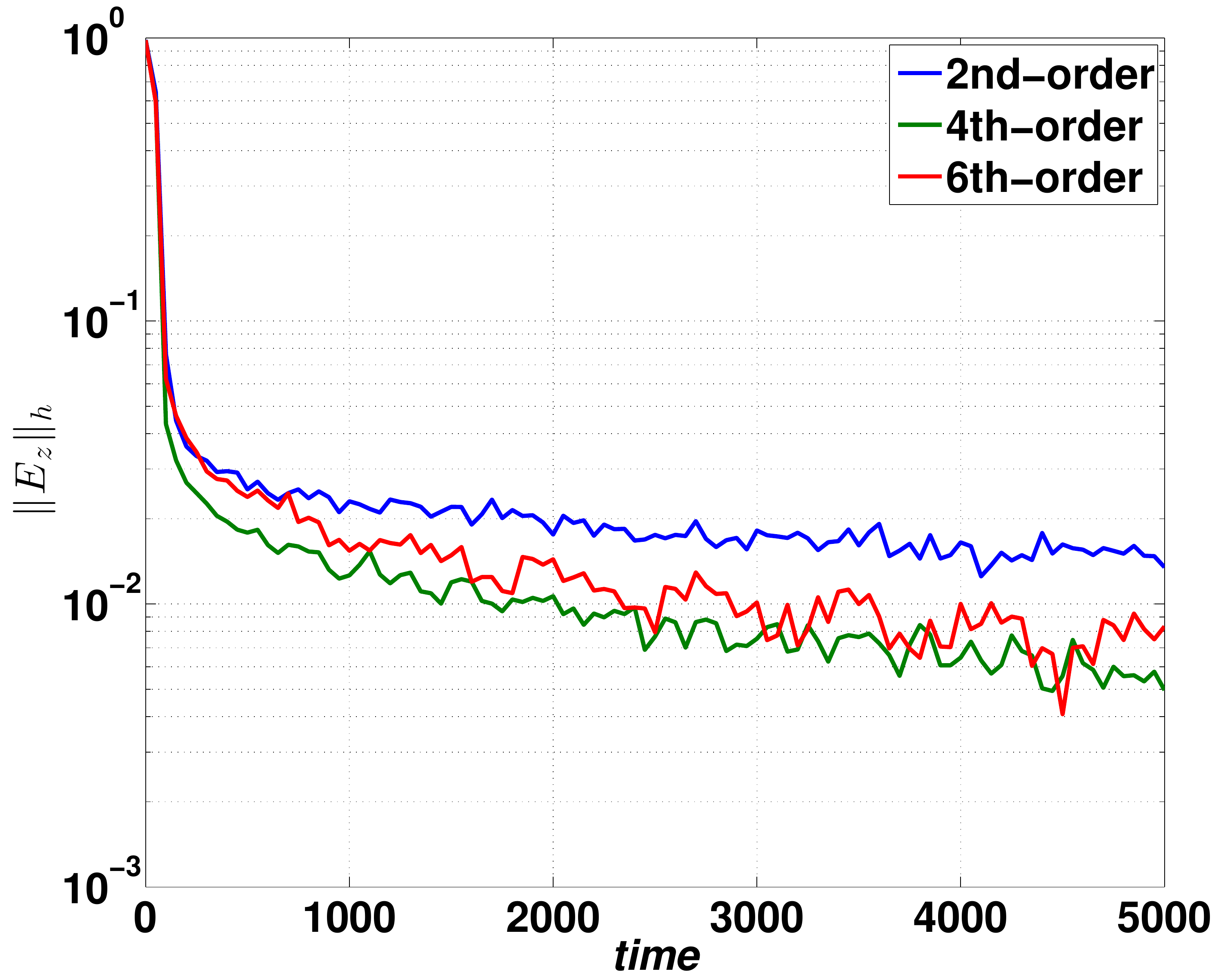}\label{fig:Model_Standard}}
 \caption{\textit{The $l_2$-norm of the electric field as a function of time  with the time-step $dt = 0.4h$}}
 \label{fig:L2Norm_waveguide_2}
\end{figure}

As in \cite{Abarbanel1998, Abarbanel2002, Abarbanel2009}, it is possible to argue here that the growth seen  for the 4th and 6th order accurate schemes are due to instabilities in the continuous PML models.   However, in \cite{DuKrApnum2013}, stability  analysis of  an IVP corresponding  to \eqref{eq:Maxwell_PML_WaveGuide_1}--\eqref{eq:Maxwell_PML_WaveGuide_4} was performed in detail. The result is that the IVP \eqref{eq:Maxwell_PML_WaveGuide_1}--\eqref{eq:Maxwell_PML_WaveGuide_4} is asymptotically stable.   We believe that the instability seen here  and,  probably, some of the growth observed in \cite{Abarbanel1998, Abarbanel2002, Abarbanel2009} are caused by inappropriate numerical boundary treatments in the PML. 

To investigate further, we have also performed numerical experiments with smaller time steps or/and spatial steps. In some cases mesh refinement or smaller time step eliminates the growth, but in other cases the instabilities persist, but  it occurs  at much later times. For instance, using half of the time step, $dt = 0.2h$, postpones the initiation of growth until about $ t= 4000$ for the  PML models \eqref{eq:Maxwell_PML_WaveGuide_1}--\eqref{eq:Maxwell_PML_WaveGuide_4} and \eqref{eq:Maxwell_PML_Split_Discrete_11}--\eqref{eq:Maxwell_PML_Split_Discrete_21} with a 4th order accurate  SBP operator.  While in all other cases growth was not seen for the entire simulation duration.  
These simulations  are  a strong  indication that the growth observed here is a numerical artifact. In particular, for the discrete physically motivated PML \eqref{eq:Maxwell_PML_Ababarnel_Discrete_11}--\eqref{eq:Maxwell_PML_Ababarnel_Discrete_41}, figure \ref{fig:Model_Standard_Ababarnel_2dt02h} indicates that the growth seen in figure  \ref{fig:Model_Standard_Ababarnel} is due to numerical stiffness. Thus, the numerical approximation \eqref{eq:Maxwell_PML_Ababarnel_Discrete_11}--\eqref{eq:Maxwell_PML_Ababarnel_Discrete_41} of the PML may not be suitable for efficient explicit time integration. These results are also  consistent with  some of the conclusions drawn in \cite{Abarbanel2009}, which show that growth for PMLs depends on both the discretization used and the PML model.  
 
 \begin{figure} [h!]
 \centering
\subfigure[Discrete  modal PML  \eqref{eq:Maxwell_PML_WaveGuide_Discrete_11}--\eqref{eq:Maxwell_PML_WaveGuide_Discrete_41}. ]{\includegraphics[width=0.49\linewidth]{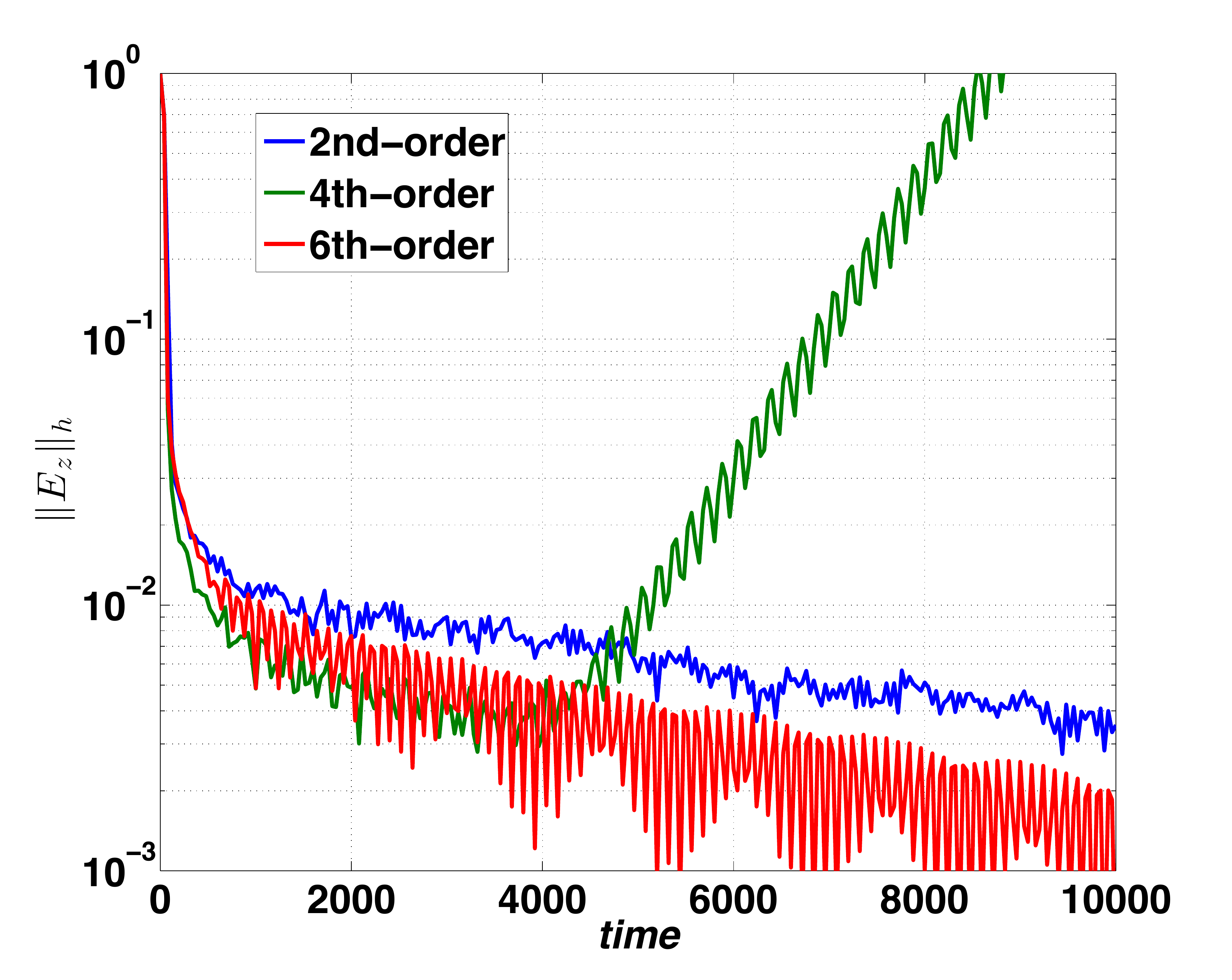}\label{fig:Model_Standard_OurPML_2dt02h}}
\subfigure[Discrete physically motivated PML \eqref{eq:Maxwell_PML_Ababarnel_Discrete_11}--\eqref{eq:Maxwell_PML_Ababarnel_Discrete_41}. ]{\includegraphics[width=0.49\linewidth]{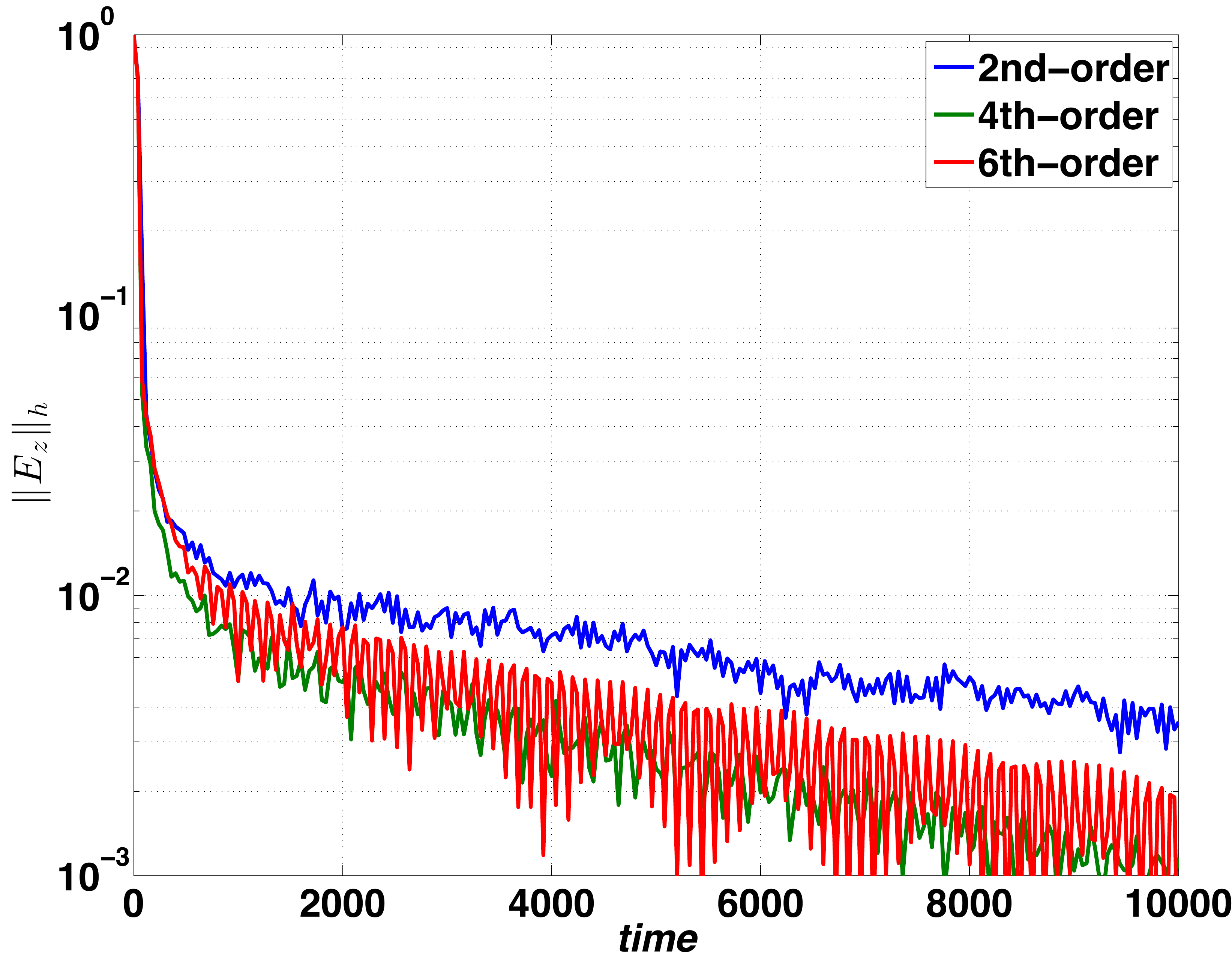}\label{fig:Model_Standard_Ababarnel_2dt02h}}
\subfigure[Discrete split--field PML \eqref{eq:Maxwell_PML_Split_Discrete_11}--\eqref{eq:Maxwell_PML_Split_Discrete_21}. ]{\includegraphics[width=0.49\linewidth]{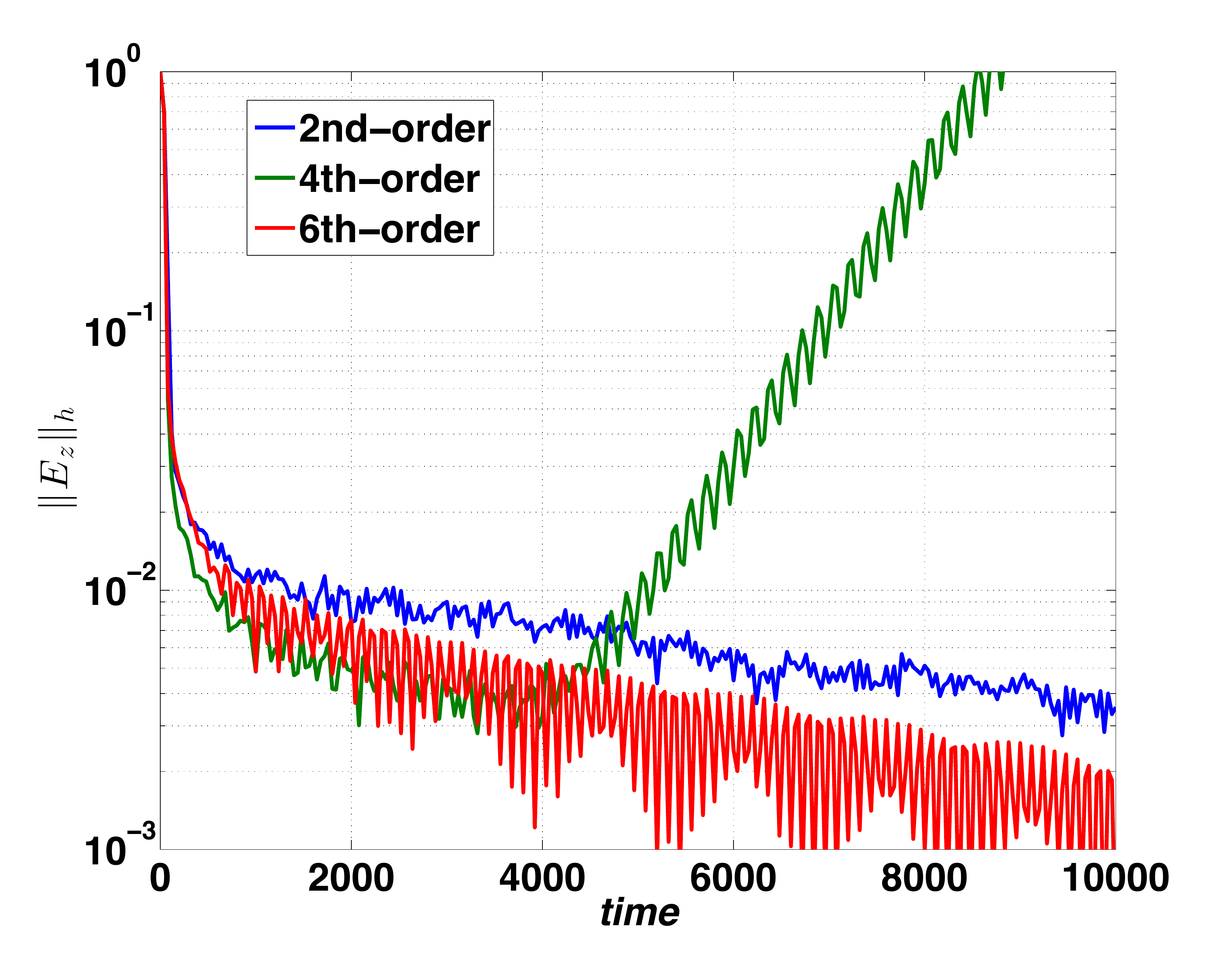}\label{fig:Model_Split_Field_2dt02h}}
\subfigure[No PML $(\sigma = 0)$.]{\includegraphics[width=0.49\linewidth]{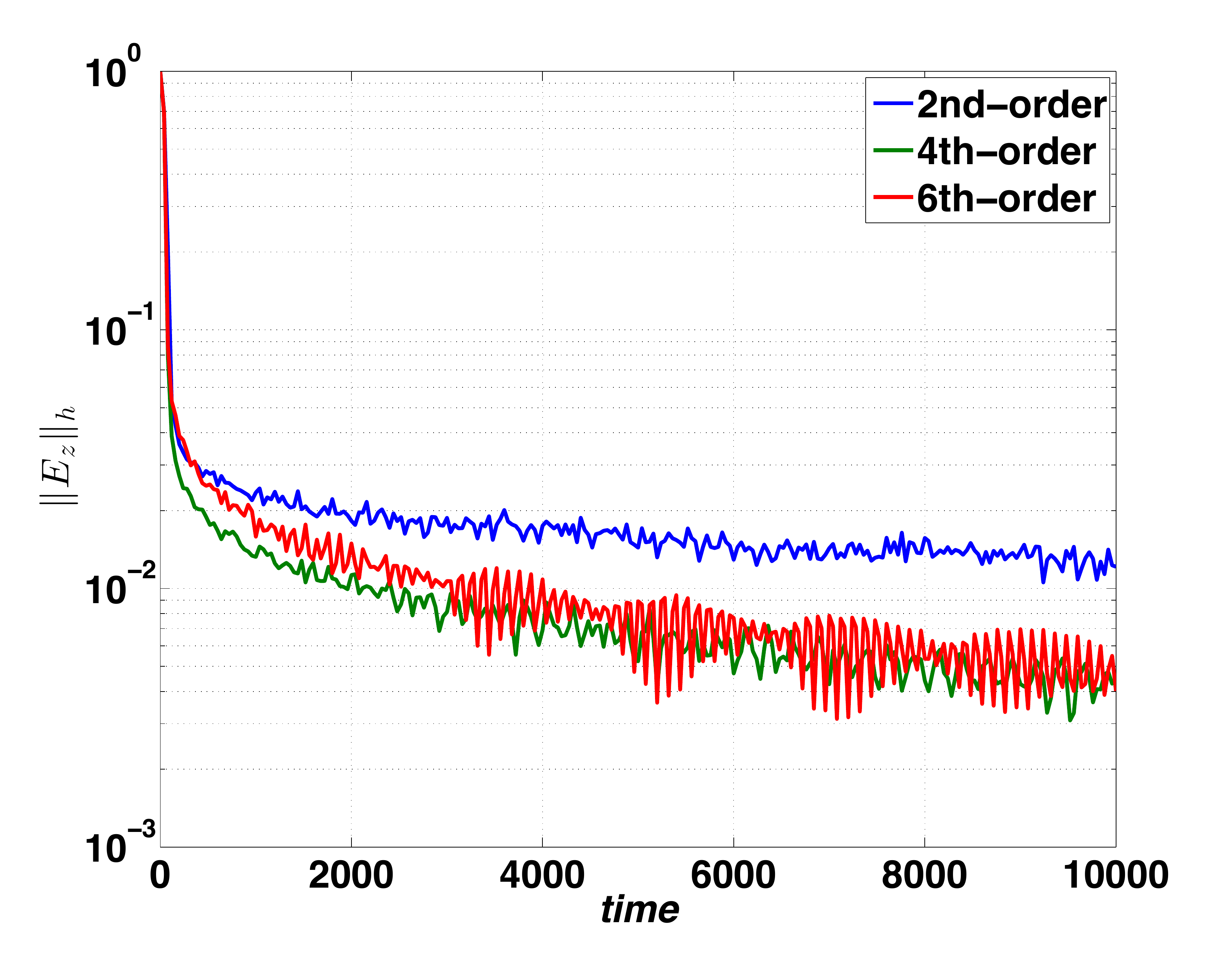}\label{fig:Model_Standard_2dt02h}}
 \caption{\textit{The $l_2$-norm of the electric field as a function of time with the time-step $dt = 0.2h$}}
 \label{fig:L2Norm_waveguide_2dt02h}
\end{figure}

  In the next section, we will show that the IBVPs, \eqref{eq:Maxwell_PML_Split_11}--\eqref{eq:Maxwell_PML_Split_14}, \eqref{eq:Maxwell_PML_WaveGuide_1}--\eqref{eq:Maxwell_PML_WaveGuide_4}, \eqref{eq:Maxwell_PML_Ababarnel_1}--\eqref{eq:Maxwell_PML_Ababarnel_4} and \eqref{eq:Maxwell_PML_Split_21}--\eqref{eq:Maxwell_PML_Split_24} with  \eqref{eq:wall_bc_y_0}, \eqref{eq:wall_bc_x}  are also stable. To enable a systematic development of stable  numerical boundary procedures we will derive continuous  energy estimates. In section 5, by mimicking the continuous  energy estimates we will design accurate and stable numerical boundary treatments for the PMLs,  \eqref{eq:Maxwell_PML_Split_11}--\eqref{eq:Maxwell_PML_Split_14}, \eqref{eq:Maxwell_PML_WaveGuide_1}--\eqref{eq:Maxwell_PML_WaveGuide_4}, \eqref{eq:Maxwell_PML_Ababarnel_1}--\eqref{eq:Maxwell_PML_Ababarnel_4} and \eqref{eq:Maxwell_PML_Split_21}--\eqref{eq:Maxwell_PML_Split_24}, subject to the boundary conditions \eqref{eq:wall_bc_y}, \eqref{eq:wall_bc_x}.

\section{Continuous analysis}
The analysis in \cite{DuKrApnum2013} shows that the constant coefficient Cauchy PML problem \eqref{eq:Maxwell_PML_WaveGuide_1}--\eqref{eq:Maxwell_PML_WaveGuide_4} is asymptotically stable. Since the continuous PML models  \eqref{eq:Maxwell_PML_Split_11}--\eqref{eq:Maxwell_PML_Split_14}, \eqref{eq:Maxwell_PML_WaveGuide_1}--\eqref{eq:Maxwell_PML_WaveGuide_4}, \eqref{eq:Maxwell_PML_Ababarnel_1}--\eqref{eq:Maxwell_PML_Ababarnel_4} and \eqref{eq:Maxwell_PML_Split_21}--\eqref{eq:Maxwell_PML_Split_24} are equivalent, the results in \cite{DuKrApnum2013} for the Cauchy problem hold also for the PML  models
\eqref{eq:Maxwell_PML_Split_11}--\eqref{eq:Maxwell_PML_Split_14},  \eqref{eq:Maxwell_PML_Ababarnel_1}--\eqref{eq:Maxwell_PML_Ababarnel_4} and \eqref{eq:Maxwell_PML_Split_21}--\eqref{eq:Maxwell_PML_Split_24}.
 Note that the boundary condition \eqref{eq:wall_bc_y_0} extends into the PML. In addition the PML is terminated in the $x$--direction by the boundary condition \eqref{eq:wall_bc_x}. 
The stability analysis  of  \eqref{eq:Maxwell_PML_Split_11}--\eqref{eq:Maxwell_PML_Split_14}, \eqref{eq:Maxwell_PML_WaveGuide_1}--\eqref{eq:Maxwell_PML_WaveGuide_4}, \eqref{eq:Maxwell_PML_Ababarnel_1}--\eqref{eq:Maxwell_PML_Ababarnel_4} or
\eqref{eq:Maxwell_PML_Split_21}--\eqref{eq:Maxwell_PML_Split_24} subject to the boundary conditions \eqref{eq:wall_bc_y_0} and 
\eqref{eq:wall_bc_x}  is not as straightforward. Here,  we  prove that  the PML  \eqref{eq:Maxwell_PML_Split_11}--\eqref{eq:Maxwell_PML_Split_14}, \eqref{eq:Maxwell_PML_WaveGuide_1}--\eqref{eq:Maxwell_PML_WaveGuide_4}, \eqref{eq:Maxwell_PML_Ababarnel_1}--\eqref{eq:Maxwell_PML_Ababarnel_4} or \eqref{eq:Maxwell_PML_Split_21}--\eqref{eq:Maxwell_PML_Split_24} with the boundary conditions \eqref{eq:wall_bc_y_0}, \eqref{eq:wall_bc_x}  is stable. First, we will use a modal analysis for  a constant coefficient PML. Second, we will derive energy estimates for the variable coefficients case.
\subsection{Normal mode analysis}\label{section:modal_analysis}
To perform a modal analysis, we assume constant coefficients. We split the PML problem into three: 1) a Cauchy PML problem with no boundary conditions, 2)  a lower half--plane PML problem, $-\infty < x < \infty, -\infty < y \le y_0$, with the boundary condition  \eqref{eq:wall_bc_y_0} at $y = y_0$, and 3) a left half-plane PML problem, $-\infty < x \le x_0+\delta, -\infty < y <\infty$, with the  boundary condition  \eqref{eq:wall_bc_x} at $x = x_0+\delta$. Each of the three PML problems above can be analyzed separately. We note that the lower half--plane PML problem and the left half--plane PML problem are distinct and must be analyzed severally.

The Cauchy problem can be analyzed, as in  \cite{DuKrApnum2013}, using standard Fourier methods.
The analysis in \cite{DuKrApnum2013} shows that the constant coefficient Cauchy PML problem  \eqref{eq:Maxwell_PML_Split_11}--\eqref{eq:Maxwell_PML_Split_14}, \eqref{eq:Maxwell_PML_WaveGuide_1}--\eqref{eq:Maxwell_PML_WaveGuide_4}, \eqref{eq:Maxwell_PML_Ababarnel_1}--\eqref{eq:Maxwell_PML_Ababarnel_4} or \eqref{eq:Maxwell_PML_Split_21}--\eqref{eq:Maxwell_PML_Split_24}  is asymptotically stable. As before we have 
\begin{theorem}\label{theo:cauchy_problem}
  The constant coefficient  Cauchy PML problem   \eqref{eq:Maxwell_PML_Split_11}--\eqref{eq:Maxwell_PML_Split_14}, \eqref{eq:Maxwell_PML_WaveGuide_1}--\eqref{eq:Maxwell_PML_WaveGuide_4}, \eqref{eq:Maxwell_PML_Ababarnel_1}--\eqref{eq:Maxwell_PML_Ababarnel_4} or \eqref{eq:Maxwell_PML_Split_21}--\eqref{eq:Maxwell_PML_Split_24}  with  $\sigma\ge 0$ is asymptotically stable.
\end{theorem}
\subsubsection{The lower half--plane PML problem: $-\infty < x < \infty, -\infty < y \le y_0$}
Consider now  the constant coefficient PML  \eqref{eq:Maxwell_PML_Split_11}--\eqref{eq:Maxwell_PML_Split_14}, \eqref{eq:Maxwell_PML_WaveGuide_1}--\eqref{eq:Maxwell_PML_WaveGuide_4}, \eqref{eq:Maxwell_PML_Ababarnel_1}--\eqref{eq:Maxwell_PML_Ababarnel_4} or \eqref{eq:Maxwell_PML_Split_21}--\eqref{eq:Maxwell_PML_Split_24} in a lower half--plane, $-\infty < x < \infty, -\infty < y \le y_0$, with the boundary condition \eqref{eq:wall_bc_y_0} at $y = y_0$. 
  Note that  the relevant ansatz is no longer a pure Fourier mode, but a simple wave solution on the form
\begin{equation}\label{eg:Ansatz2} 
\textbf{V} = \widehat{\textbf{V}}(y)e^{s t -ik_xx },  \quad |\widehat{\textbf{V}}(y)|< \infty.
\end{equation}
In order to prove the stability of the PML \eqref{eq:Maxwell_PML_Split_11}--\eqref{eq:Maxwell_PML_Split_14}, \eqref{eq:Maxwell_PML_WaveGuide_1}--\eqref{eq:Maxwell_PML_WaveGuide_4}, \eqref{eq:Maxwell_PML_Ababarnel_1}--\eqref{eq:Maxwell_PML_Ababarnel_4} or \eqref{eq:Maxwell_PML_Split_21}--\eqref{eq:Maxwell_PML_Split_24}  subject to the boundary condition \eqref{eq:wall_bc_y_0}, we will show that the only solution of the form       \eqref{eg:Ansatz2}  with $\Re{s} \ge   0$ for any $\sigma > 0$ is the trivial  solution, $\textbf{V}  = 0$.

 To begin with, we introduce the complex number $z = a + ib$ and define  the branch cut of $\sqrt{z}$ by
\begin{align}
-\pi < \arg{\left(a + ib\right)} \le \pi, \quad \arg{\sqrt{a + ib}} = \frac{1}{2}\arg{{\left(a + ib\right)}}.
\end{align}
Next, we introduce \eqref{eg:Ansatz2}  in  the PML  \eqref{eq:Maxwell_PML_Split_11}--\eqref{eq:Maxwell_PML_Split_14}, \eqref{eq:Maxwell_PML_WaveGuide_1}--\eqref{eq:Maxwell_PML_WaveGuide_4}, \eqref{eq:Maxwell_PML_Ababarnel_1}--\eqref{eq:Maxwell_PML_Ababarnel_4} or \eqref{eq:Maxwell_PML_Split_21}--\eqref{eq:Maxwell_PML_Split_24},   and the boundary condition \eqref{eq:wall_bc_y_0}.
We can eliminate the magnetic fields and the auxiliary variable,  then  we obtain a second order ODE,
 \begin{equation}\label{eq:wave_laplace_t_fourier_x}
\begin{split}
\left({s^2} +\left(\frac{k_x}{1+\frac{\sigma}{s}}\right)^2\right) \widehat{ E}_z = &\frac{d^2 \widehat{ E}_z}{d y^2},
\end{split}
\end{equation}
subject to the boundary condition
\begin{equation}\label{eq:wall_bc_y_Laplace}
 \widehat{H}_x + \gamma_y  \widehat{ E}_z= 0 \iff \frac{1}{s}\left(\frac{d \widehat{ E}_z}{d y}  + \gamma_y s \widehat{ E}_z\right)= 0,\quad \text{with} \quad\gamma_y \ge 0, \quad \text{at} \quad y =  y_0.
\end{equation}
or the PEC boundary condition
\begin{equation}\label{eq:pec_bc_y_Laplace}
  \widehat{ E}_z= 0 , \quad \text{at} \quad y =  y_0.
\end{equation}
Consider now $\Re{s} > 0 $, we can construct modal solutions, for \eqref{eq:wave_laplace_t_fourier_x}, of the form 
\begin{equation}\label{modal_solution_y}
\widehat{ E}_z = a_0e^{\kappa y} + b_0e^{-\kappa y}, 
\end{equation} 
where 
\[
\kappa = \sqrt{s^2 + \left(\frac{1}{1+\frac{\sigma}{s}}k_x\right)^2}, \quad \Re{s} > 0,
\]
and $a_0, b_0$ are parameters which must not be zero at the same time.
Note that it can be shown that $\Re{\kappa} > 0$ for all $\sigma, \Re{s} > 0$, see lemma \ref{Lem:Duru1} in the appendix. Thus, boundedness of the solutions at $y \to -\infty$ implies $b_0 = 0$. The parameter $a_0$ is determined by the boundary condition \eqref{eq:wall_bc_y_Laplace}. Introducing \eqref{modal_solution_y} in  \eqref{eq:wall_bc_y_Laplace} we obtain 
\begin{equation}\label{eq:wall_bc_y_Laplace_condition}
   \frac{a_0}{s+\sigma}\left( \sqrt{\left(s+\sigma\right)^2 + k_x^2}  + \gamma_y \left(s+\sigma\right)\right)= 0.
\end{equation}
Since $a_0 \ne 0$, we have nontrivial solutions only if 
\begin{equation}\label{eq:wall_bc_y_condition}
   \mathcal{F}_1\left(s+\sigma, k_x\right)\equiv\frac{1}{s+\sigma}\left(\sqrt{\left(s+\sigma\right)^2 + k_x^2}  + \gamma_y \left(s+\sigma\right)\right)= 0.
\end{equation}
The following lemma characterizes the roots of the dispersion relation \eqref{eq:wall_bc_y_condition}.
\begin{lemma}\label{lem:lower_half_plane}
Consider the dispersion relation $\mathcal{F}_1\left(s+\sigma, k_x\right)$ defined in \eqref{eq:wall_bc_y_condition} with $\gamma_y \ge 0.$  The equation $\mathcal{F}_1\left(s+\sigma, k_x\right) = 0$ has no solution for all $\Re{s} \ge 0$, $k_x \in \mathbb{R}$ and $\sigma > 0.$
\end{lemma}
\begin{proof}
\newline
\newline
Consider first $\sigma = 0$,  we have $\mathcal{F}_1\left(s,  k_x\right) = 0$, with the solution $s = s_0 \in \mathbb{C}$. Note that since  the Maxwell's equation \eqref{eq:Maxwell} subject to the boundary conditions  \eqref{eq:wall_bc_y} satisfy the energy estimate \eqref{eq:strict} we must have $\Re{s_0}\le 0$, for all $\gamma_y \ge 0$. Otherwise if $\Re{s_0}> 0$, then the energy will grow for some solutions, which is in contradiction with the energy estimate \eqref{eq:strict}.  By inspection, the zero root $s_0 = 0$ is also not a solution of $\mathcal{F}_1\left(s,  k_x\right) = 0.$ 

 If $\sigma > 0$, then the equation $\mathcal{F}_1\left(s+\sigma, k_x\right) = 0$  has the solution $s  = -\sigma + s_0.$ Therefore, any $\sigma > 0$ will move all roots $s$ further into the stable complex plane.
\end{proof}

By lemma \ref{lem:lower_half_plane} we must have $\mathcal{F}_1\left(s+\sigma,  k_x\right) \ne 0$ for all $\sigma > 0$, $\Re{s}  \ge 0$. This also  implies that  $a_0 \equiv 0$ in \eqref{modal_solution_y} for all $\sigma, \Re{s}  > 0$. Therefore,   the trivial solution is the only possible solution of \eqref{eq:Maxwell_PML_WaveGuide_1}--\eqref{eq:Maxwell_PML_WaveGuide_4} with \eqref{eq:wall_bc_y} or \eqref{eq:Maxwell_PML_Ababarnel_1}--\eqref{eq:Maxwell_PML_Ababarnel_4} with \eqref{eq:wall_bc_y}, of  the form       \eqref{eg:Ansatz2}  with $\Re{s} \ge  0$.

Consider now the PEC boundary condition \eqref{eq:pec_bc_y_Laplace} and substitute \eqref{modal_solution_y}, with $b_0 \equiv 0$, we have
\[
a_0 = 0.
\]
Thus the PEC boundary condition \eqref{eq:pec_bc_y_Laplace}  supports only trivial solutions, $\textbf{V} = 0$.
 
\subsubsection{The left half--plane problem: $-\infty < x \le x_0+\delta, -\infty < y < \infty $}
Consider now  the constant coefficient PML \eqref{eq:Maxwell_PML_Split_11}--\eqref{eq:Maxwell_PML_Split_14}, \eqref{eq:Maxwell_PML_WaveGuide_1}--\eqref{eq:Maxwell_PML_WaveGuide_4}, \eqref{eq:Maxwell_PML_Ababarnel_1}--\eqref{eq:Maxwell_PML_Ababarnel_4} or \eqref{eq:Maxwell_PML_Split_21}--\eqref{eq:Maxwell_PML_Split_24} in the left half--plane, $-\infty < x \le x_0+\delta,$ $ -\infty < y < \infty $, with the boundary condition \eqref{eq:wall_bc_x} at $x = x_0 +\delta$. 
  We make the ansatz
\begin{equation}\label{eg:Ansatz_x} 
\textbf{V} = \widehat{\textbf{V}}(x)e^{s t -ik_yy },  \quad |\widehat{\textbf{V}}(x)|< \infty.
\end{equation}
In order to prove the stability of the PML, \eqref{eq:Maxwell_PML_Split_11}--\eqref{eq:Maxwell_PML_Split_14}, \eqref{eq:Maxwell_PML_WaveGuide_1}--\eqref{eq:Maxwell_PML_WaveGuide_4}, \eqref{eq:Maxwell_PML_Ababarnel_1}--\eqref{eq:Maxwell_PML_Ababarnel_4} or \eqref{eq:Maxwell_PML_Split_21}--\eqref{eq:Maxwell_PML_Split_24},   subject to the boundary condition \eqref{eq:wall_bc_x}, we will show that there are no nontrivial  solutions on the form           
 \eqref{eg:Ansatz_x}  with $\Re{s} >  0$ for any $\sigma \ge 0$. 
 We insert  \eqref{eg:Ansatz_x} in  \eqref{eq:Maxwell_PML_Split_11}--\eqref{eq:Maxwell_PML_Split_14}, \eqref{eq:Maxwell_PML_WaveGuide_1}--\eqref{eq:Maxwell_PML_WaveGuide_4}, \eqref{eq:Maxwell_PML_Ababarnel_1}--\eqref{eq:Maxwell_PML_Ababarnel_4} or \eqref{eq:Maxwell_PML_Split_21}--\eqref{eq:Maxwell_PML_Split_24} and in  the boundary condition \eqref{eq:wall_bc_x}.
As before, we  eliminate the magnetic fields and the auxiliary variable to  obtain a second order ordinary differential equation,
 \begin{equation}\label{eq:wave_laplace_t_fourier_y}
\begin{split}
\left({s^2} +k_y^2\right) \widehat{ E}_z = &\left(\frac{1}{1+\frac{\sigma}{s}}\right)^2\frac{d^2 \widehat{ E}_z}{d x^2},
\end{split}
\end{equation}
subject to the boundary conditions
\begin{equation}\label{eq:wall_bc_x_Laplace}
 \widehat{H}_y -   \gamma_x\widehat{ E}_z= 0 \iff \frac{1}{s}\left(\left(\frac{1}{1+\frac{\sigma}{s}}\right)\frac{d \widehat{ E}_z}{d x}  +  s \gamma_x\widehat{ E}_z\right)= 0, \quad  \text{with} \quad\gamma_x \ge 0, \quad \text{at} \quad x =  x_0+\delta,
\end{equation}
or the PEC boundary condition
\begin{equation}\label{eq:pec_bc_x_Laplace}
  \widehat{ E}_z= 0 , \quad \text{at} \quad x =  x_0+\delta.
\end{equation}
Consider now $\Re{s} > 0 $, we can construct modal solutions, for \eqref{eq:wave_laplace_t_fourier_y}, on the form 
\begin{equation}\label{modal_solution_x}
\widehat{ E}_z = a_0e^{\kappa x} + b_0e^{-\kappa x}, 
\end{equation} 
where 
\[
\kappa = \left({1 + \frac{\sigma}{s}}\right)\sqrt{ s^2+ k_y^2}, \quad \Re{s} > 0,
\]
and $a_0, b_0$ are parameters which must not be zero at the same time.
Note that it can be shown that $\Re{\kappa} > 0$ for all $\sigma, \Re{s} > 0$, see lemma \ref{Lem:Duru2} in the appendix. Thus, boundedness of the solution \eqref{modal_solution_x} at $x \to -\infty$ implies $b_0 = 0$. The parameter $a_0$ is determined by the boundary condition \eqref{eq:wall_bc_x_Laplace}. Introducing \eqref{modal_solution_x} in  \eqref{eq:wall_bc_x_Laplace} we obtain 
\begin{equation}\label{eq:wall_bc_x_Laplace_condition}
   \frac{a_0}{s}\left( \sqrt{s^2 + k_y^2}  +  s \right)= 0.
\end{equation}
Since $a_0 \ne 0$, we have nontrivial solutions only if 
\begin{equation}\label{eq:wall_bc_x_condition}
   \mathcal{F}_2\left(s, k_y\right)\equiv \frac{1}{s}\left(\sqrt{s^2 + k_y^2}  +  \gamma_x s\right) = 0.
\end{equation}
It is particularly important to note that the dispersion relation $\mathcal{F}_2\left(s, k_y\right)$,  defined by \eqref{eq:wall_bc_x_condition}, is independent of the PML parameter $\sigma\ge 0$.  The dispersion relation \eqref{eq:wall_bc_x_condition} corresponds to that of  the undamped problem, and it is not perturbed by the PML. This is  unlike the dispersion relation, $\mathcal{F}_1\left(s+\sigma, k_y\right)$ for the lower half--plane  PML problem,  which depends  on the PML parameter $\sigma\ge 0$.

Since, the Maxwell's equation \eqref{eq:Maxwell} subject to the boundary conditions  \eqref{eq:wall_bc_x} and \eqref{eq:wall_bc_y} satisfy the energy estimate \eqref{eq:strict}, the following lemma characterizing the roots of the dispersion relation \eqref{eq:wall_bc_x_condition}  holds
\begin{lemma}\label{lem:left_half_plane}
Consider the dispersion relation $\mathcal{F}_2\left(s, k_y\right)$ defined in \eqref{eq:wall_bc_x_condition} with $\gamma_x > 0$.  The equation $\mathcal{F}_2\left(s, k_y\right) = 0$ has no solution $s$  for all  $\Re{s} \ge 0$ and $ k_y \in \mathbb{R}$.
\end{lemma}
\begin{proof}
 Since  the Maxwell's equation \eqref{eq:Maxwell} with the boundary condition  \eqref{eq:wall_bc_x} at $x = x_0+\delta$ satisfy the energy estimate \eqref{eq:strict} we must have $\Re{s}\le 0$, for all $ k_y \in \mathbb{R}$. Otherwise if $\Re{s}> 0$, then the energy will grow for some solutions, which is also in contradiction with the energy estimate \eqref{eq:strict}.  Again by inspection, the zero root $s = 0$ is not a solution of $\mathcal{F}_2\left(s,  k_y\right) = 0.$ Note that if $k_y = 0$ we have $\mathcal{F}_2\left(s, k_y\right) = 1+\gamma_x \ne 0$.   Next, we consider the case $s = i\xi$, $\xi \in \mathbb{R}/\{0\}$, and introduce
$\bar{\xi} = \xi/k_y$
 \begin{equation}\label{eq:wall_bc_x_condition2}
   \mathcal{F}_2\left(s, k_y\right)\equiv \frac{1}{s}\left(\sqrt{s^2 + k_y^2}  +  \gamma_xs\right) = 0 \iff \mathcal{F}_0\left(\bar{\xi}\right)\equiv \sqrt{ 1 - \frac{1}{\bar{\xi^2}} } + \gamma_x = 0 .
\end{equation}
If $\bar{\xi^2} > 1$, then we have $\mathcal{F}_0\left(\bar{\xi}\right) \ne 0$ since the square root can never be negative. If $\bar{\xi^2} \le 1$, then we also have $\mathcal{F}_0\left(\bar{\xi}\right) \ne 0$ since the square root is zero or  purely imaginary.  Therefore, we must have $\Re{s}< 0$, for all $ k_y \in \mathbb{R}$. This completes the proof of the lemma.
\end{proof}

As before,  for $\Re{s} \ge  0$ the trivial solution is the only possible solution of \eqref{eq:Maxwell_PML_WaveGuide_1}--\eqref{eq:Maxwell_PML_WaveGuide_4} with \eqref{eq:wall_bc_x} or \eqref{eq:Maxwell_PML_Ababarnel_1}--\eqref{eq:Maxwell_PML_Ababarnel_4} with \eqref{eq:wall_bc_x},  on  the form       \eqref{eg:Ansatz_x}.

Consider now the PEC boundary condition \eqref{eq:pec_bc_y_Laplace} and substitute \eqref{modal_solution_x}, with $b_0 \equiv 0$, we have
\[
a_0 = 0.
\]
The PEC boundary condition \eqref{eq:pec_bc_x_Laplace}  also supports only trivial solutions, $\textbf{V} = 0$.

By theorem \ref{theo:cauchy_problem}, and lemmas \ref{lem:lower_half_plane} and \ref{lem:left_half_plane}, it then follows that all non--trivial modes in the PML decay. The difficulty lies in constructing  accurate and stable numerical approximations for the PML,  \eqref{eq:Maxwell_PML_Split_11}--\eqref{eq:Maxwell_PML_Split_14}, \eqref{eq:Maxwell_PML_WaveGuide_1}--\eqref{eq:Maxwell_PML_WaveGuide_4}, \eqref{eq:Maxwell_PML_Ababarnel_1}--\eqref{eq:Maxwell_PML_Ababarnel_4} or \eqref{eq:Maxwell_PML_Split_21}--\eqref{eq:Maxwell_PML_Split_24} subject to the boundary conditions  \eqref{eq:wall_bc_x} and \eqref{eq:wall_bc_y}, and ensuring numerical stability.

\subsection{Energy equation in the Laplace space}
Consider now  the constant coefficient PML  \eqref{eq:Maxwell_PML_Split_11}--\eqref{eq:Maxwell_PML_Split_14}, \eqref{eq:Maxwell_PML_WaveGuide_1}--\eqref{eq:Maxwell_PML_WaveGuide_4}, \eqref{eq:Maxwell_PML_Ababarnel_1}--\eqref{eq:Maxwell_PML_Ababarnel_4} or \eqref{eq:Maxwell_PML_Split_21}--\eqref{eq:Maxwell_PML_Split_24} in the rectangular domain, $-(x_0 +\delta) \le x \le (x_0 +\delta), -y_0 \le y \le y_0$, and the boundary conditions \eqref{eq:wall_bc_y_0}, \eqref{eq:wall_bc_x}.  We will prove that the corresponding constant coefficient IBVP can not support growing modes. 

To begin, we assume homogeneous initial data and take the Laplace transform of the PML equations and boundary conditions in time. We can eliminate the magnetic fields and all the auxiliary variables,
we have
 \begin{equation}\label{eq:pml_wave_laplace_t}
\begin{split}
{s^2} \widehat{ E}_z = &\frac{1}{S_x}\frac{\partial}{\partial x}\left(\frac{1}{S_x}\frac{\partial \widehat{ E}_z}{\partial x}\right) + \frac{\partial^2 \widehat{ E}_z}{\partial y^2}, \quad \Re{s} > 0,
\end{split}
\end{equation}
subject to the boundary conditions
\begin{equation}\label{eq:wall_bc_x_Laplace_t}
 \widehat{H}_y \mp   \gamma_x\widehat{ E}_z= 0 \iff \frac{1}{s}\left(\frac{1}{S_x}\frac{d \widehat{ E}_z}{d x}  \pm  \gamma_x s \widehat{ E}_z\right)= 0, \quad \text{at} \quad x = \pm( x_0+\delta),
\end{equation}
\begin{equation}\label{eq:wall_bc_y_Laplace_t}
 \widehat{H}_x \pm \gamma_x \widehat{ E}_z= 0 \iff \frac{1}{s}\left(\frac{d \widehat{ E}_z}{d y}  \pm  \gamma_x s \widehat{ E}_z\right)= 0, \quad \text{at} \quad y =  \pm y_0.
\end{equation}
or
\begin{equation}\label{eq:wall_bc_x_Laplace_pec}
\widehat{ E}_z= 0, \quad \text{at} \quad x = \pm( x_0+\delta),
\end{equation}
\begin{equation}\label{eq:wall_bc_y_Laplace_pec}
 \widehat{ E}_z= 0 , \quad \text{at} \quad y =  \pm y_0.
\end{equation}
We define the quantity
 \begin{align}\label{eq:pml_norm_laplace}
\widehat{\mathcal{E}}_{\sigma}(s) = & \Re{s}\|s\widehat{E}_z\|^2 + \Re{\left(\frac{(sS_x)^*}{S_x}\right)}\Big\|\frac{1}{S_x}\frac{\partial { \widehat{E}_z}}{\partial x} \Big\|^2  + \Re{s}\Big\|\frac{\partial { \widehat{E}_z}}{\partial y} \Big\|^2  
 +   \gamma_y\left(\|s\widehat{E}_z\left(y_0\right)\|^2_{\Gamma_x} + \|s\widehat{E}_z\left(-y_0\right)\|^2_{\Gamma_x}\right)  \notag \\
&+ \gamma_x\Re{\left(\frac{1}{S_x}\right)}  \left(\|s\widehat{E}_z\left(x_0+\delta\right)\|^2_{\Gamma_y} + \|s\widehat{E}_z\left(-x_0-\delta\right)\|^2_{\Gamma_y}\right).
 \end{align}
For the PEC boundary conditions \eqref{eq:wall_bc_x_Laplace_pec}, \eqref{eq:wall_bc_y_Laplace_pec}, the boundary terms vanish and we have
\begin{align}\label{eq:pml_norm_laplace_pec}
\widehat{\mathcal{E}}^0_{\sigma}(s) = & \Re{s}\|s\widehat{E}_z\|^2 + \Re{\left(\frac{(sS_x)^*}{S_x}\right)}\Big\|\frac{1}{S_x}\frac{\partial { \widehat{E}_z}}{\partial x} \Big\|^2  + \Re{s}\Big\|\frac{\partial { \widehat{E}_z}}{\partial y} \Big\|^2 .
 \end{align}
Note that  $\Re{s} > 0 \implies \Re{\left(\frac{(sS_x)^*}{S_x}\right)}, \Re{\left(\frac{1}{S_x}\right)}  > 0 $, see lemma \ref{Lem:Duru3} in the Appendix.
 Therefore, $\widehat{\mathcal{E}}_{\sigma}(s) > 0$ and  $\widehat{\mathcal{E}}^0_{\sigma}(s) > 0$ are  energies.  Note also that for $ \gamma_x \ge 0$, $ \gamma_y \ge 0$ we have
\[
\widehat{\mathcal{E}}_{\sigma}(s) = 0 \iff \widehat{E}_z = \frac{\partial { \widehat{E}_z}}{\partial x} = \frac{\partial { \widehat{E}_z}}{\partial y} = 0,
\]
and
\[
\widehat{\mathcal{E}}^0_{\sigma}(s) = 0 \iff \widehat{E}_z = \frac{\partial { \widehat{E}_z}}{\partial x} = \frac{\partial { \widehat{E}_z}}{\partial y} = 0.
\]
We can now prove the theorem:
\begin{theorem}\label{Theorem:Wellposedness_Laplace}
For any $s$, with $ \Re{s} > 0$,  and $\sigma > 0,$ all  solutions $\widehat{E}_z \ne 0$ of the constant coefficient PML  \eqref{eq:pml_wave_laplace_t} in the rectangular domain, with the boundary conditions \eqref{eq:wall_bc_x_Laplace_t}, \eqref{eq:wall_bc_y_Laplace_t}  or \eqref{eq:wall_bc_x_Laplace_pec}, \eqref{eq:wall_bc_y_Laplace_pec} must satisfy
\begin{equation}\label{eq:EnergyEstimate_continuous_Laplace}
\begin{split}
\widehat{\mathcal{E}}_{\sigma}(s) = 0, \quad \text{or} \quad \widehat{\mathcal{E}}^0_{\sigma}(s) = 0.
\end{split}
\end{equation}
\end{theorem}
\begin{proof}
Multiply equation \eqref{eq:pml_wave_laplace_t} from the left by $ \left(s\widehat{{E}}_z\right)^*$,  and integrate over the whole domain, we have 
\begin{align}\label{eq:enrgery_laplace_1}
\left(s^2\widehat{E}_z, s\widehat{E}_z\right) = \left(\frac{1}{S_x}\frac{\partial}{\partial x}\left(\frac{1}{S_x}\frac{\partial \widehat{ E}_z}{\partial x}\right), s\widehat{E}_z \right) + \left(\frac{d^2 \widehat{ E}_z}{d y^2}, s\widehat{E}_z\right) 
\end{align}
Integrating by parts and using the boundary conditions \eqref{eq:wall_bc_x_Laplace_t}, \eqref{eq:wall_bc_y_Laplace_t} obtaining
\begin{align}\label{eq:enrgery_laplace_2}
s\left(s\widehat{E}_z, s\widehat{E}_z\right)& = -\frac{\left(sS_x\right)^*}{S_x}\left(\frac{1}{S_x}\frac{\partial \widehat{ E}_z}{\partial x}, \frac{1}{S_x}\frac{\partial \widehat{ E}_z}{\partial x} \right) - s^*\left(\frac{\partial \widehat{ E}_z}{\partial  y}, \frac{\partial \widehat{ E}_z}{\partial  y}\right)   -  \gamma_y\left( s\widehat{ E}_z\left(-y_0\right), s\widehat{ E}_z\left(-y_0\right)\right)_{\Gamma_x}  -  \gamma_y\left( s\widehat{ E}_z\left(y_0\right), s\widehat{ E}_z\left(y_0\right)\right)_{\Gamma_x}
 \notag  \\ 
 &
 - \frac{ \gamma_x}{S_x}\left( s\widehat{ E}_z\left(-x_0-\delta\right), s\widehat{ E}_z\left(-x_0-\delta\right)\right)_{\Gamma_y}  - \frac{ \gamma_x}{S_x}\left( s\widehat{ E}_z\left(x_0+\delta\right), s\widehat{ E}_z\left(x_0+\delta\right)\right)_{\Gamma_y}.
\end{align}
We add the complex  conjugates of the products in \eqref{eq:enrgery_laplace_2} to obtain 
\begin{align}\label{eq:enrgery_laplace_3}
& \Re{s}\|s\widehat{E}_z\|^2 + \Re{\left(\frac{(sS_x)^*}{S_x}\right)}\Big\|\frac{1}{S_x}\frac{\partial { \widehat{E}_z}}{\partial x} \Big\|^2  + \Re{s}\Big\|\frac{\partial { \widehat{E}_z}}{\partial y} \Big\|^2  
 +   \gamma_y\left(\|s\widehat{E}_z\left(y_0\right)\|^2_{\Gamma_x} + \|s\widehat{E}_z\left(-y_0\right)\|^2_{\Gamma_x}\right) \notag  \\
&+  \gamma_x\Re{\left(\frac{1}{S_x}\right)}  \left(\|s\widehat{E}_z\left(x_0+\delta\right)\|^2_{\Gamma_y} + \|s\widehat{E}_z\left(-x_0-\delta\right)\|^2_{\Gamma_y}\right) = 0.
\end{align}
 Note that in particular if the PEC boundary conditions \eqref{eq:wall_bc_x_Laplace_pec}, \eqref{eq:wall_bc_y_Laplace_pec} are imposed,  the boundary terms vanish and 
 \begin{align}\label{eq:enrgery_laplace_4}
& \Re{s}\|s\widehat{E}_z\|^2 + \Re{\left(\frac{(sS_x)^*}{S_x}\right)}\Big\|\frac{1}{S_x}\frac{\partial { \widehat{E}_z}}{\partial x} \Big\|^2  + \Re{s}\Big\|\frac{\partial { \widehat{E}_z}}{\partial y} \Big\|^2   = 0.
\end{align}
 Thus, for $\Re{s} > 0$ no nontrivial solutions $\widehat{E}_z \ne 0$  can be supported by the PML.
\end{proof}

\subsection{Energy estimates in the time domain}
To enable the development of a  systematic  and rigorous  stable numerical boundary procedure for the time--dependent PML  we will  derive energy estimates for the variable coefficient PML. The importance of the energy estimates are twofold: The first is that the energy estimates establish the well--posedness of the variable coefficients PMLs. The second is that by mimicking the energy estimates we can construct high order accurate, stable and convergent discrete approximations of the PML in a bounded domain. 
\subsubsection{Energy estimates for the modal PML}
The first step is to rewrite the PML  \eqref{eq:Maxwell_PML_WaveGuide_1}--\eqref{eq:Maxwell_PML_WaveGuide_4} and the boundary conditions \eqref{eq:wall_bc_y_0},  \eqref{eq:wall_bc_x} as a second order system. Differentiating equation  \eqref{eq:Maxwell_PML_WaveGuide_1} with respect to time and using  equations  \eqref{eq:Maxwell_PML_WaveGuide_2}--\eqref{eq:Maxwell_PML_WaveGuide_4} to eliminate the time--derivatives on the right hand side, we obtain 
\begin{subequations}\label{eq:Maxwell_PML_WaveGuide_2nd2}
    \begin{alignat}{2}
     \frac{\partial^2{ E_z}}{\partial t^2} + \sigma \frac{\partial{ E_z}}{\partial t}&= \frac{\partial}{\partial x}\left(\frac{\partial{ E_z}}{\partial x} + \sigma H_y\right) + \frac{\partial}{\partial y}\left(\frac{\partial{ E_z}}{\partial y} + \sigma H_x\right)  , \label{eq:Maxwell_PML_WaveGuide_2nd}
 \\
     \frac{\partial{ H_y}}{\partial t}  &= -\frac{\partial{ E_z}}{\partial x} - \sigma H_y, \label{eq:Maxwell_PML_WaveGuide_x}\\
     \frac{\partial{ H_x}}{\partial t} &= \frac{\partial{ E_z}}{\partial y}.
      \label{eq:Maxwell_PML_WaveGuide_y}
    \end{alignat}
  \end{subequations}

The corresponding boundary conditions are
\begin{equation}\label{eq:wall_bc_y1}
   \gamma_x E_z \pm H_x = 0\implies \gamma_x\frac{\partial{ E_z}}{\partial t} + \sigma\gamma_x E_z \pm \left(\frac{\partial{ E_z}}{\partial y}+ \sigma H_x\right) = 0, \quad \text{at} \quad y = \pm y_0.
\end{equation}
\begin{equation}\label{eq:wall_bc_x1}
 \gamma_yE_z \mp   H_y = 0, \implies\gamma_y\frac{\partial{ E_z}}{\partial t} \pm \left(\frac{\partial{ E_z}}{\partial x} + \sigma H_y\right) = 0, \quad \text{at} \quad x = \pm (x_0 + \delta).
\end{equation}
or 
\begin{equation}\label{eq:wall_bc_y1_pec}
  E_z  = 0, \implies \frac{\partial{ E_z}}{\partial t} = 0, \quad \text{at} \quad y = \pm y_0.
\end{equation}
\begin{equation}\label{eq:wall_bc_x1_pec}
E_z  = 0, \implies \frac{\partial{ E_z}}{\partial t} = 0, \quad \text{at} \quad x = \pm (x_0 + \delta).
\end{equation}
Here, the length and width of the computational domain are $\Gamma_x := \{ x: -x_0-\delta \le x\le x_0+\delta\} $ and  $\Gamma_y := \{ y: -y_0 \le y\le y_0\}$, respectively.
To derive an  energy estimate we  introduce the relevant  energy norms 
\begin{equation}\label{eq:EnergyNorm_continuous}
\begin{split}
\mathcal{E}_{\sigma}^{\left(1\right)}\left(t\right) &= \Big\|\frac{\partial { E_z}}{\partial t} \Big\|^2  +  \Big\|\frac{\partial{ E_z}}{\partial x} + \sigma H_y\Big\|^2 + \Big\|\frac{\partial{ E_z}}{\partial y}+ \sigma H_x\Big\|^2 + \Big\|\sigma H_y\Big\|^2 + \Big\| \sigma H_x\Big\|^2 +\gamma \sigma\left\| E_z\left(-y_0, \mathrm{t}'\right)\right\|^2_{\Gamma_x} +\gamma\sigma \left\|E_z\left(y_0, \mathrm{t}'\right)\right\|^2_{\Gamma_x} \\
&
+\int_0^{t}\left(\gamma _x\left\|\frac{\partial { E_z}}{\partial t}\left(-y_0, \mathrm{t}'\right)\right\|^2_{\Gamma_x} +\gamma_x \left\|\frac{\partial { E_z}}{\partial t}\left(y_0, \mathrm{t}'\right)\right\|^2_{\Gamma_x}  + \gamma _y\left\|\frac{\partial { E_z}}{\partial t}\left(x_0+\delta, \mathrm{t}'\right)\right\|^2_{\Gamma_y}+\gamma _y \left\|\frac{\partial { E_z}}{\partial t}\left(-x_0-\delta, \mathrm{t}'\right)\right\|^2_{\Gamma_y}\right)\mathrm{dt}',
\end{split}
\end{equation}
and 
\begin{equation}\label{eq:EnergyNorm_continuous_pec}
\begin{split}
\mathcal{E}_{\sigma}^{\left(0\right)}\left(t\right) &= \Big\|\frac{\partial { E_z}}{\partial t} \Big\|^2  +  \Big\|\frac{\partial{ E_z}}{\partial x} + \sigma H_y\Big\|^2 + \Big\|\frac{\partial{ E_z}}{\partial y}+ \sigma H_x\Big\|^2 + \Big\|\sigma H_y\Big\|^2 + \Big\| \sigma H_x\Big\|^2 .
\end{split}
\end{equation}
Note that away from the boundaries, the energy norm \eqref{eq:EnergyNorm_continuous} is analogous to the norms derived in  \cite{DuAcoustic, Duru2013}. The difference  between \eqref{eq:EnergyNorm_continuous} and \cite{DuAcoustic, Duru2013}  is that the energy norm  introduced here contains also boundary norms, $\|\cdot\|_{\Gamma_x}$, $\|\cdot\|_{\Gamma_y}$.
Note also that  $\mathcal{E}_{\sigma}^{\left(1\right)}\left(t\right)  = 0$, $\mathcal{E}_{\sigma}^{\left(0\right)}\left(t\right)  = 0$, if and only if 
\[
\frac{\partial { E_z}}{\partial t} = \frac{\partial { E_z}}{\partial x} = \frac{\partial { E_z}}{\partial y} =  \sigma H_x =  \sigma H_y = 0.
\]
When the damping vanishes, $\sigma = 0$, we have the physical energy norms
\begin{equation}\label{eq:EnergyNorm_continuous_physical}
\begin{split}
\mathcal{E}^{\left(1\right)}\left(t\right) &= \Big\|\frac{\partial { E_z}}{\partial t} \Big\|^2  +  \Big\|\frac{\partial{ E_z}}{\partial x} \Big\|^2 + \Big\|\frac{\partial{ E_z}}{\partial y}\Big\|^2  \\
&
+2\int_0^{t}\left(\gamma_x \left\|\frac{\partial { E_z}}{\partial t}\left(-y_0, \mathrm{t}'\right)\right\|^2_{\Gamma_x} +\gamma_x \left\|\frac{\partial { E_z}}{\partial t}\left(y_0, \mathrm{t}'\right)\right\|^2_{\Gamma_x}  + \gamma _y \left\|\frac{\partial { E_z}}{\partial t}\left(x_0+\delta, \mathrm{t}'\right)\right\|^2_{\Gamma_y}+ \gamma _y\left\|\frac{\partial { E_z}}{\partial t}\left(-x_0-\delta, \mathrm{t}'\right)\right\|^2_{\Gamma_y}\right)\mathrm{dt}',
\end{split}
\end{equation}
and 
\begin{equation}\label{eq:EnergyNorm_continuous_physical_pec}
\begin{split}
\mathcal{E}^{\left(0\right)}\left(t\right) &= \Big\|\frac{\partial { E_z}}{\partial t} \Big\|^2  +  \Big\|\frac{\partial{ E_z}}{\partial x} \Big\|^2 + \Big\|\frac{\partial{ E_z}}{\partial y}\Big\|^2  .
\end{split}
\end{equation}
Let us define the spaces of real functions
\begin{equation}
\mathbf{\mathcal{V}}_{\mathcal{E}^{\left(1\right)}}: = \{E_z:  \mathcal{E}^{\left(1\right)} < \infty\}, \quad \mathbf{\mathcal{V}}_{\mathcal{E}_{\sigma}^{\left(1\right)}}: = \{\left(E_z,  H_x,  H_y \right):  \mathcal{E}_{\sigma}^{\left(1\right)} < \infty\},
\end{equation}
\begin{equation}
\mathbf{\mathcal{V}}_{\mathcal{E}^{\left(0\right)}}: = \{E_z:  \mathcal{E}^{\left(0\right)} < \infty\}, \quad \mathbf{\mathcal{V}}_{\mathcal{E}_{\sigma}^{\left(0\right)}}: = \{\left(E_z,  H_x,  H_y \right):  \mathcal{E}_{\sigma}^{\left(0\right)} < \infty\},
\end{equation}
and $L^2\left(\Omega\right)$  the space of square integrable real functions. It can be  shown that 
\begin{equation}
\left(E_z,  H_x,  H_y \right) \in \mathbf{\mathcal{V}}_{\mathcal{E}_{\sigma}^{\left(1\right)}} \iff \{E_z \in \mathbf{\mathcal{V}}_{\mathcal{E}^{\left(1\right)}}, \quad \left(H_x,  H_y \right) \in L^2\left(\Omega\right)\},
\end{equation}
and
\begin{equation}
\left(E_z,  H_x,  H_y \right) \in \mathbf{\mathcal{V}}_{\mathcal{E}_{\sigma}^{\left(0\right)}} \iff \{E_z \in \mathbf{\mathcal{V}}_{\mathcal{E}^{\left(0\right)}}, \quad \left(H_x,  H_y \right) \in L^2\left(\Omega\right)\}.
\end{equation}
We can now prove the theorem, see  Appendix \ref{proof:Wellposedness}:
\begin{theorem}\label{Theorem:Wellposedness}
All solutions of the  perfectly matched layer \eqref{eq:Maxwell_PML_WaveGuide_1}--\eqref{eq:Maxwell_PML_WaveGuide_4} or  (\ref{eq:Maxwell_PML_WaveGuide_2nd})--(\ref{eq:Maxwell_PML_WaveGuide_y}) subject to the  boundary conditions  (\ref{eq:wall_bc_y1})  at $y = \pm y_0$ and  (\ref{eq:wall_bc_x1}) at $x = \pm x_0 $, or (\ref{eq:wall_bc_y1_pec})  at $y = \pm y_0$ and  (\ref{eq:wall_bc_x1_pec}) at $x = \pm x_0 $ satisfies the energy estimate
\begin{equation}\label{eq:EnergyEstimate_continuous}
\begin{split}
\frac{\mathrm{d}}{\mathrm{dt}}\sqrt{\mathcal{E}_{\sigma}^{\left(1\right)}\left(\mathrm{t}\right)} \le { \sigma_{\infty}}\sqrt{\mathcal{E}_{\sigma}^{\left(1\right)}\left(t\right)}, \quad  \text{or}  \quad  \frac{\mathrm{d}}{\mathrm{dt}}\sqrt{\mathcal{E}_{\sigma}^{\left(0\right)}\left(\mathrm{t}\right)} \le { \sigma_{\infty}}\sqrt{\mathcal{E}_{\sigma}^{\left(0\right)}\left(t\right)},
\end{split}
\end{equation}
where $\sigma_{\infty} = \max_x \sigma(x). $
\end{theorem}

\begin{remark}
It is particularly noteworthy that the PML models  \eqref{eq:Maxwell_PML_Split_11}--\eqref{eq:Maxwell_PML_Split_14} and \eqref{eq:Maxwell_PML_Split_21}--\eqref{eq:Maxwell_PML_Split_24} can be rewritten directly in the time domain as \eqref{eq:Maxwell_PML_WaveGuide_1}--\eqref{eq:Maxwell_PML_WaveGuide_4}. Therefore, theorem \ref{Theorem:Wellposedness} holds also for the solutions of \eqref{eq:Maxwell_PML_Split_11}--\eqref{eq:Maxwell_PML_Split_14} and \eqref{eq:Maxwell_PML_Split_21}--\eqref{eq:Maxwell_PML_Split_24}.
\end{remark}

In the next section, by mimicking the continuous energy estimate \eqref{eq:EnergyEstimate_continuous}, we will develop stable numerical boundary procedures for the PML models \eqref{eq:Maxwell_PML_Split_11}--\eqref{eq:Maxwell_PML_Split_14}, \eqref{eq:Maxwell_PML_WaveGuide_1}--\eqref{eq:Maxwell_PML_WaveGuide_4} and \eqref{eq:Maxwell_PML_Split_21}--\eqref{eq:Maxwell_PML_Split_24}, subject to the  boundary conditions  (\ref{eq:wall_bc_y1})  at $y = \pm y_0$ and  (\ref{eq:wall_bc_x1}) at $x = \pm (x_0+\delta) $. 
\subsubsection{Energy estimates for the physically motivated PML}
The physically motivated PML model \eqref{eq:Maxwell_PML_Ababarnel_1}--\eqref{eq:Maxwell_PML_Ababarnel_4}, satisfies the definition of a strongly hyperbolic system. 
By applying the energy method directly to  the PML, \eqref{eq:Maxwell_PML_Ababarnel_1}--\eqref{eq:Maxwell_PML_Ababarnel_4},  subject to the boundary conditions \eqref{eq:wall_bc_x_0}, \eqref{eq:wall_bc_y},  with $\gamma_j \ge 0$,  we have
\begin{equation}\label{eq:estimate_phymotpml_1}
\begin{split}
&\frac{d}{dt}\left(\|E_z\left(t\right)\|^2  + \|H_y\left(t\right)\|^2  + \|H_x\left(t\right)\|^2 + \|P\left(t\right)\|^2\right)  =  -2\sigma \|E_z\left(t\right)\|^2  -2\sigma \|H_y\left(t\right)\|^2  + 2\sigma \|H_x\left(t\right)\|^2 -2\sigma \|P\left(t\right)\|^2
\\
&- 2\gamma_x  \left(\|E_z\left(t, x_0+\delta\right)\|^2_{\Gamma_y} + \|E_z\left(t, -x_0-\delta\right)\|^2_{\Gamma_y}\right) - 2\gamma_y \ \left(\|E_z\left(t, y_0\right)\|^2_{\Gamma_x} + \|E_z\left(t, -y_0\right)\|^2_{\Gamma_x}\right).
\end{split}
\end{equation}
In particular if  $\gamma_j = 0$ in \eqref{eq:wall_bc_x_0}, \eqref{eq:wall_bc_y} or $R_j = -1$ (we consider the PEC condition) in \eqref{eq:wall_bc_x},  \eqref{eq:wall_bc_y_0} then
\begin{equation}\label{eq:estimate_phymotpml_pec}
\begin{split}
&\frac{d}{dt}\left(\|E_z\left(t\right)\|^2  + \|H_y\left(t\right)\|^2  + \|H_x\left(t\right)\|^2 + \|P\left(t\right)\|^2\right)  =  -2\sigma \|E_z\left(t\right)\|^2  -2\sigma \|H_y\left(t\right)\|^2  + 2\sigma \|H_x\left(t\right)\|^2 -2\sigma \|P\left(t\right)\|^2.
\end{split}
\end{equation}
Note that  all terms, excepting $+ 2\sigma \|H_x\left(t\right)\|^2$, in the right hand sides of  \eqref{eq:estimate_phymotpml_1},  \eqref{eq:estimate_phymotpml_pec}   are dissipative.
We denote the  energy by
\begin{equation}\label{eq:energy_phymotpml}
\begin{split}
\mathcal{E}^{(2)}\left(\mathrm{t}\right) &= \|E_z\left(t\right)\|^2  + \|H_y\left(t\right)\|^2  + \|H_x\left(t\right)\|^2 + \|P\left(t\right)\|^2 \\
&
+2\int_0^{t}\left(\gamma_x  \left(\|E_z\left(t', x_0+\delta\right)\|^2_{\Gamma_y} + \|E_z\left(t', -x_0-\delta\right)\|^2_{\Gamma_y}\right)  + \gamma_y \ \left(\|E_z\left(t', y_0\right)\|^2_{\Gamma_x} + \|E_z\left(t', -y_0\right)\|^2_{\Gamma_x}\right)\right)\mathrm{dt}',
\end{split}
\end{equation}
and 
\begin{equation}\label{eq:energy_phymotpml_pec}
\begin{split}
\mathcal{E}^{(2)}_0\left(\mathrm{t}\right) &= \|E_z\left(t\right)\|^2  + \|H_y\left(t\right)\|^2  + \|H_x\left(t\right)\|^2 + \|P\left(t\right)\|^2 ,
\end{split}
\end{equation}
if  $\gamma_j = 0$ in \eqref{eq:wall_bc_x_0}, \eqref{eq:wall_bc_y} or $R_j = -1$ (we consider the PEC condition) in \eqref{eq:wall_bc_x}.
We have
\begin{theorem}\label{Theorem:Wellposedness_physical}
The solutions of the perfectly matched layer \eqref{eq:Maxwell_PML_Ababarnel_1}--\eqref{eq:Maxwell_PML_Ababarnel_4} subject to the  boundary conditions \eqref{eq:wall_bc_x_0}, \eqref{eq:wall_bc_y} with $|R_j| < 1$ or $|R_j| = 1$ satisfy the energy estimate
\begin{equation}\label{eq:EnergyEstimate_continuous_physical}
\begin{split}
\frac{\mathrm{d}}{\mathrm{dt}}\sqrt{\mathcal{E}^{\left(2\right)}\left(\mathrm{t}\right)} \le  \sigma_{\infty} \sqrt{\mathcal{E}^{\left(2\right)}\left(t\right)}, \quad \text{or} \quad \frac{\mathrm{d}}{\mathrm{dt}}\sqrt{\mathcal{E}_0^{\left(2\right)}\left(\mathrm{t}\right)} \le  \sigma_{\infty} \sqrt{\mathcal{E}_0^{\left(2\right)}\left(t\right)},
\end{split}
\end{equation}
where $\sigma_{\infty} = \max_x \sigma(x). $
\end{theorem}
\section{Discrete approximations and discrete stability analysis of the PMLs}
The main focus of this section is to design accurate, stable and efficient  numerical  boundary procedures for the PMLs, \eqref{eq:Maxwell_PML_Split_11}--\eqref{eq:Maxwell_PML_Split_14}, \eqref{eq:Maxwell_PML_WaveGuide_1}--\eqref{eq:Maxwell_PML_WaveGuide_4},\eqref{eq:Maxwell_PML_Split_21}--\eqref{eq:Maxwell_PML_Split_24} and \eqref{eq:Maxwell_PML_Ababarnel_1}--\eqref{eq:Maxwell_PML_Ababarnel_4}, subject to the boundary conditions \eqref{eq:wall_bc_y_0}, \eqref{eq:wall_bc_x}. We will use SBP finite difference operators, see \cite{Bert_Gust, Strand94}, to approximate the spatial derivatives and impose boundary conditions weakly using penalties \cite{CarpenterGottliebAbarbanel1994, Matt2003}. To ensure numerical stability we will impose boundary conditions in a manner which allows a derivation of   discrete energy estimates analogous to the continuous energy estimate (\ref{eq:EnergyEstimate_continuous}). For the physically motivated PML model \eqref{eq:Maxwell_PML_Ababarnel_1}--\eqref{eq:Maxwell_PML_Ababarnel_4}, an efficient numerical  boundary treatment is achieved  by  using characteristics and reducing the strength of the penalty parameters. Finally, we will show that the corresponding discrete models do not support growing modes.
\subsection{Discrete approximation  and discrete energy estimates for the  modal PML}
In standard numerical methods \cite{Abarbanel2002, Abarbanel2009,  Becache2003}, once the interior scheme is proven stable the discrete PML is obtained by discretizing the auxiliary differential equations and appending the auxiliary functions accordingly. Usually no more attention is placed on how the boundary conditions are implemented in the PML. Note that  numerical experiments  presented in  section 3 and the next section show that a straightforward enforcement of the boundary conditions, \eqref{eq:wall_bc_y_0} and \eqref{eq:wall_bc_x}, can lead to unstable solutions in the PML. A  stable discrete approximation for the modal PML can  be obtained by  extending the enforcement of the boundary condition to the auxiliary differential equation, having 
\begin{subequations}\label{eq:Maxwell_PML1_WaveGuide_Stable_0}
    \begin{alignat}{2}
       \frac{\mathrm{d}{\mathbf{E}_z}}{\mathrm{d t}} &= -\left(D_x \otimes I_y\right)\mathbf {H}_y + \left(I_x \otimes D_y\right)\mathbf {H}_x +\mathbf {H}_x^* - \mathbf{\sigma}\mathbf{E}_z \notag \\
    &   \underbrace{-\alpha_x\left(\frac{1-R_x}{2}\left(\mathrm{P}_x^{-1}\left(E_{Rx}+E_{Lx}\right)\otimes I_y \right)\mathbf {E}_z  
 - \frac{1+R_x}{2}\left(\mathrm{P}_x^{-1}\left(E_{Rx}-E_{Lx}\right)\otimes I_y \right)\mathbf {H}_y\right)}_{\mathrm{SAT}_x} \notag \\
 & \underbrace{-\alpha_y \left( \frac{1-R_y}{2}\left(I_x \otimes \mathrm{P}_y^{-1}\left(E_{Ry}+E_{Ly}\right)\right)\mathbf {E}_z +\frac{1+R_y}{2}\left(I_x \otimes \mathrm{P}_y^{-1}\left(E_{Ry}-E_{Ly}\right)\right)\mathbf {H}_x\right)}_{\mathrm{SAT}_y},
\label{eq:Maxwell_PML_WaveGuide_Discrete_11_Stable_0} \\
      \frac{\mathrm{d}{\mathbf{H}_y}}{\mathrm{d t}}  &= -\left(D_x \otimes I_y\right)\mathbf {E}_z- \mathbf{\sigma}\mathbf {H}_y \underbrace{+\theta_x\left(\frac{1-R_x}{2}\left(\mathrm{P}_x^{-1}\left(E_{Rx}-E_{Lx}\right)\otimes I_y \right)\mathbf {E}_z  
 - \frac{1+R_x}{2}\left(\mathrm{P}_x^{-1}\left(E_{Rx}+E_{Lx}\right)\otimes I_y \right)\mathbf {H}_y\right)}_{\mathrm{SAT}_x},\label{eq:Maxwell_PML_WaveGuide_Discrete_21_Stable_0} \\
     \frac{\mathrm{d}{\mathbf{H}_x}}{\mathrm{d t}}  &= \left(I_x \otimes D_y\right)\mathbf {E}_z \underbrace{-\theta_y \left( \frac{1-R_y}{2}\left(I_x \otimes \mathrm{P}_y^{-1}\left(E_{Ry}-E_{Ly}\right)\right)\mathbf {E}_z +\frac{1+R_y}{2}\left(I_x \otimes \mathrm{P}_y^{-1}\left(E_{Ry}+E_{Ly}\right)\right)\mathbf {H}_x\right)}_{\mathrm{SAT}_y}, \label{eq:Maxwell_PML_WaveGuide_Discrete_31_Stable_0}\\ 
   \frac{\mathrm{d}{\mathbf{H}_x^*}}{\mathrm{d t}}   &= \sigma \left(\left(I_x \otimes D_y\right)  \mathbf {H}_x  \underbrace{-\theta \alpha_y \left( \frac{1-R_y}{2}\left(I_x \otimes \mathrm{P}_y^{-1}\left(E_{Ry}+E_{Ly}\right)\right)\mathbf {E}_z +\frac{1+R_y}{2}\left(I_x \otimes \mathrm{P}_y^{-1}\left(E_{Ry}-E_{Ly}\right)\right)\mathbf {H}_x\right)}_{\text{Stabilizing term}} \right) .\label{eq:Maxwell_PML_WaveGuide_Discrete_41_Stable_0}
    \end{alignat}
  \end{subequations}
Note that we have included the term  involving $\theta$.  The standard choice $\theta = 0$,  which corresponds to \eqref{eq:Maxwell_PML_WaveGuide_Discrete_11}--\eqref{eq:Maxwell_PML_WaveGuide_Discrete_41}, is obtained  by simply replacing spatial derivatives with difference operators in the PML. The other choice, $\theta \ne 0$, extends weak enforcements of the boundary conditions, \eqref{eq:wall_bc_y_0} in the $y$-direction,  to the  auxiliary differential equation \eqref{eq:Maxwell_PML_WaveGuide_Discrete_4}. It is particularly important to note that this extension, with $\theta \ne 0$, does not destroy  the accuracy  of the scheme \eqref{eq:Maxwell_PML_WaveGuide_Discrete_1}--\eqref{eq:Maxwell_PML_WaveGuide_Discrete_4}. When $\sigma = 0$, the auxiliary variable $\mathbf{H}_x^*$ also vanishes, and the discrete approximation \eqref{eq:Maxwell_PML_WaveGuide_Discrete_1}--\eqref{eq:Maxwell_PML_WaveGuide_Discrete_4},   for any value of $\theta$, correspond to the stable discrete approximation  \eqref{eq:Maxwell_WaveGuide_Discrete_1}--\eqref{eq:Maxwell_WaveGuide_Discrete_3}. However, as we will  see later, when $\sigma > 0$ the particular choice $\theta = 1$ ensures numerical stability.

To be able to simplify the analysis and also make comparisons with results published in the literature, we will focus on  the setup considered in section \ref{sect:num_example} and  \cite{Abarbanel1998, Abarbanel2002, Abarbanel2009}. Therefore, we set $R_x = R_y = 0$ and the penalty parameters
\[
\alpha_x = 2, \quad \alpha_y = 2, \quad \theta_x = \theta_y = 0,
\]
having
\begin{subequations}\label{eq:Maxwell_PML1_WaveGuide_Stable}
    \begin{alignat}{2}
        \frac{\mathrm{d}{\mathbf{E}_z}}{\mathrm{d t}} &= -\left(D_x \otimes I_y\right)\mathbf {H}_y + \left(I_x \otimes D_y\right)\mathbf {H}_x +\mathbf {H}_x^* - \mathbf{\sigma}\mathbf{E}_z 
      \underbrace{-\left(\left(\mathrm{P}_x^{-1}\left(E_{Rx}+E_{Lx}\right)\otimes I_y \right)\mathbf {E}_z  
 - \left(\mathrm{P}_x^{-1}\left(E_{Rx}-E_{Lx}\right)\otimes I_y \right)\mathbf {H}_y\right)}_{\mathrm{SAT}_x} \notag \\
 & \underbrace{- \left(\left(I_x \otimes \mathrm{P}_y^{-1}\left(E_{Ry}+E_{Ly}\right)\right)\mathbf {E}_z +\left(I_x \otimes \mathrm{P}_y^{-1}\left(E_{Ry}-E_{Ly}\right)\right)\mathbf {H}_x\right)}_{\mathrm{SAT}_y},
 \label{eq:Maxwell_PML_WaveGuide_Discrete_1} \\
      \frac{\mathrm{d}{\mathbf{H}_y}}{\mathrm{d t}}  &= -\left(D_x \otimes I_y\right)\mathbf {E}_z- \mathbf{\sigma}\mathbf {H}_y, \label{eq:Maxwell_PML_WaveGuide_Discrete_2} \\
     \frac{\mathrm{d}{\mathbf{H}_x}}{\mathrm{d t}}  &= \left(I_x \otimes D_y\right)\mathbf {E}_z , \label{eq:Maxwell_PML_WaveGuide_Discrete_3}\\ 
   \frac{\mathrm{d}{\mathbf{H}_x^*}}{\mathrm{d t}}   &= \sigma \left(\left(I_x \otimes D_y\right)  \mathbf {H}_x \underbrace{- \theta \left(\left(I_x \otimes \mathrm{P}_y^{-1}\left(E_{Ry}+E_{Ly}\right)\right)\mathbf {E}_z +\left(I_x \otimes \mathrm{P}_y^{-1}\left(E_{Ry}-E_{Ly}\right)\right)\mathbf {H}_x\right)}_{\text{Stabilizing term}}\right) .\label{eq:Maxwell_PML_WaveGuide_Discrete_4}
    \end{alignat}
  \end{subequations}

\subsubsection{Discrete energy estimates in the time domain}
We will derive a discrete energy estimate  analogous to the continuous energy estimate  (\ref{eq:EnergyEstimate_continuous}). First and foremost, we rewrite \eqref{eq:Maxwell_PML_WaveGuide_Discrete_1}--\eqref{eq:Maxwell_PML_WaveGuide_Discrete_4} as a second order system. Differentiate equation \eqref{eq:Maxwell_PML_WaveGuide_Discrete_1} in time, and use equations \eqref{eq:Maxwell_PML_WaveGuide_Discrete_2}--\eqref{eq:Maxwell_PML_WaveGuide_Discrete_4}  to eliminate the time derivatives in the right hand side. We have
\begin{subequations}\label{eq:Maxwell_PML1_WaveGuide_Stable_step1}
    \begin{alignat}{2}
        \frac{\mathrm{d^2}{\mathbf{E}_z}}{\mathrm{d t^2}} &= \left(D_x \otimes I_y\right)\left(\left(D_x \otimes I_y\right)\mathbf {E}_z+ \mathbf{\sigma}\mathbf {H}_y\right) + \left(I_x \otimes D_y\right)\left(I_x \otimes D_y\right)\mathbf {E}_z  \\
        \notag
        &+ \sigma \left(\left(I_x \otimes D_y\right)  \mathbf {H}_x \underbrace{- \theta \left( \left(I_x \otimes \mathrm{P}_y^{-1}\left(E_{Ry}+E_{Ly}\right)\right)\mathbf {E}_z +\left(I_x \otimes \mathrm{P}_y^{-1}\left(E_{Ry}-E_{Ly}\right)\right)\mathbf {H}_x\right)}_{\text{Stabilizing term}}\right) - \mathbf{\sigma}\frac{\mathrm{d}{\mathbf{E}_z}}{\mathrm{d t}} 
       \\
       \notag
     & \underbrace{-\left(\left(\mathrm{P}_x^{-1}\left(E_{Rx}+E_{Lx}\right)\otimes I_y \right)\frac{\mathrm{d}{\mathbf{E}_z}}{\mathrm{d t}}   
 + \left(\mathrm{P}_x^{-1}\left(E_{Rx}-E_{Lx}\right)\otimes I_y \right)\left(\left(D_x \otimes I_y\right)\mathbf {E}_z+ \mathbf{\sigma}\mathbf {H}_y\right)\right)}_{\mathrm{SAT}_x} \notag \\
 & \underbrace{- \left(\left(I_x \otimes \mathrm{P}_y^{-1}\left(E_{Ry}+E_{Ly}\right)\right)\frac{\mathrm{d}{\mathbf{E}_z}}{\mathrm{d t}}  +\left(I_x \otimes \mathrm{P}_y^{-1}\left(E_{Ry}-E_{Ly}\right)\right)\left(I_x \otimes D_y\right)\mathbf {E}_z\right)}_{\mathrm{SAT}_y},
 \label{eq:Maxwell_PML_WaveGuide_Discrete_1_step1} \\
      \frac{\mathrm{d}{\mathbf{H}_y}}{\mathrm{d t}}  &= -\left(D_x \otimes I_y\right)\mathbf {E}_z- \mathbf{\sigma}\mathbf {H}_y, \label{eq:Maxwell_PML_WaveGuide_Discrete_2_step1} \\
     \frac{\mathrm{d}{\mathbf{H}_x}}{\mathrm{d t}}  &= \left(I_x \otimes D_y\right)\mathbf {E}_z , \label{eq:Maxwell_PML_WaveGuide_Discrete_3_step1}
    \end{alignat}
  \end{subequations}
 Subsequently, by using the SBP property \eqref{eq:SBPDx} we eliminate some of the SAT terms and obtain

\begin{subequations}\label{eq:Maxwell_PML_WaveGuide_Discrete_first}
    \begin{alignat}{2}
      \frac{\mathrm{d^2}{\mathbf{E}_z}}{\mathrm{d t^2}}  &= -\left(\mathrm{P}_x^{-1}D_x^T \otimes I_y\right)\left(\mathrm{P}_x \otimes I_y\right)\left(\left(D_x \otimes I_y\right)\mathbf {E}_z +  \mathbf{\sigma}\mathbf {H}_y\right) - \left( I_x\otimes\mathrm{P}_y^{-1}D_y^T \right)\left(I_x \otimes \mathrm{P}_y\right)\left(\left(I_x \otimes D_y\right)\mathbf {E}_z +  \mathbf{\sigma}\mathbf {H}_x\right) 
  - \sigma\frac{\mathrm{d}{\mathbf{E}_z}}{\mathrm{d t}} \notag \\ 
  &+ \sigma(1-\theta) \left(I_x \otimes \mathrm{P}_y^{-1}\left(E_{Ry}-E_{Ly}\right)\right)\mathbf {H}_x  - \left(\mathrm{P}_x^{-1}\left(E_{Rx}+E_{Lx}\right)\otimes I_y + \left(I_x \otimes \mathrm{P}_y^{-1}\left(E_{Ry}+E_{Ly}\right)\right)\right)\frac{\mathrm{d}{\mathbf{E}_z}}{\mathrm{d t}} \notag\\
  & - \sigma\theta\left(I_x \otimes \mathrm{P}_y^{-1}\left(E_{Ry}+E_{Ly}\right)\right)\mathbf{E}_z,\label{eq:Maxwell_PML_WaveGuide_Discrete_first_1} \\
      \frac{\mathrm{d}{\mathbf{H}_y}}{\mathrm{d t}}  &= -\left(D_x \otimes I_y\right)\mathbf {E}_z- \mathbf{\sigma}\mathbf {H}_y , \label{eq:Maxwell_PML_WaveGuide_Discrete_first_2} \\
     \frac{\mathrm{d}{\mathbf{H}_x}}{\mathrm{d t}}  &= \left(I_x \otimes D_y\right)\mathbf {E}_z.\label{eq:Maxwell_PML_WaveGuide_Discrete_first_3}
    \end{alignat}
  \end{subequations}
Introduce the discrete energy
\begin{equation}\label{eq:Energy_Discrete_first}
\begin{split}
 \mathcal{E}_{h\sigma}^{(1)}\left(t\right) &= \Big\|\frac{\mathrm{d} { \mathbf{E}_z}}{\mathrm {d t}} \Big\|^2_{\mathbf{P}_{xy}}  +  \Big\|\left(D_x \otimes I_y\right)\mathbf {E}_z+ \mathbf{\sigma}\mathbf {H}_y\Big\|^2_{\mathbf{P}_{xy}} + \Big\| \left(I_x \otimes D_y\right)\mathbf {E}_z + \sigma\mathbf {H}_x\Big\|^2_{\mathbf{P}_{xy}} + \Big\|\sigma\mathbf {H}_y\Big\|^2_{\mathbf{P}_{xy}} + \Big\| \sigma\mathbf {H}_x\Big\|^2_{\mathbf{P}_{xy}}\\
 & + \mathbf{E}_z^T\left(\mathrm{P}_x \otimes \sigma\theta\left(E_{Ry}+E_{Ly}\right)\right)\mathbf{E}_z
 + 2\int_0^t\left(\frac{\mathrm{d}{\mathbf{E}_z}}{\mathrm{d t}}^T \left(\left(E_{Rx}+E_{Lx}\right)\otimes \mathrm{P}_y  +  \mathrm{P}_x\otimes \left(E_{Ry}+E_{Ly}\right)\right)\frac{\mathrm{d}{\mathbf{E}_z}}{\mathrm{d t}}\right) dt' .
  \end{split}
  \end{equation}
When $\sigma = 0$ we have the standard discrete energy
\begin{equation}\label{eq:Energy_Discrete_standard}
\begin{split}
 \mathcal{E}_{h}^{(1)}\left(t\right) &= \Big\|\frac{\mathrm{d} { \mathbf{E}_z}}{\mathrm {d t}} \Big\|^2_{\mathbf{P}_{xy}}  +  \Big\|\left(D_x \otimes I_y\right)\mathbf {E}_z\Big\|^2_{\mathbf{P}_{xy}} + \Big\| \left(I_x \otimes D_y\right)\mathbf {E}_z \Big\|^2_{\mathbf{P}_{xy}} \\
 &+ 2\int_0^t\left(\frac{\mathrm{d}{\mathbf{E}_z}}{\mathrm{d t}}^T \left(\left(E_{Rx}+E_{Lx}\right)\otimes \mathrm{P}_y  +  \mathrm{P}_x\otimes \left(E_{Ry}+E_{Ly}\right)\right)\frac{\mathrm{d}{\mathbf{E}_z}}{\mathrm{d t}}\right) dt' .
  \end{split}
  \end{equation}

We can also prove the following, see  Appendix \ref{proof:Discrete_Stability}.
\begin{theorem}\label{Theorem:Discrete_Stability}
Consider the discrete approximation \eqref{eq:Maxwell_PML_WaveGuide_Discrete_1}--\eqref{eq:Maxwell_PML_WaveGuide_Discrete_4} or \eqref{eq:Maxwell_PML_WaveGuide_Discrete_first_1}--\eqref{eq:Maxwell_PML_WaveGuide_Discrete_first_3} of the the perfectly matched layer \eqref{eq:Maxwell_PML_WaveGuide_1}--\eqref{eq:Maxwell_PML_WaveGuide_4}  subject to the  boundary conditions  (\ref{eq:wall_bc_y1})  at $y = \pm y_0$ and  (\ref{eq:wall_bc_x1}) at $x = \pm x_0 $. If $\theta = 1$, then the quantity $ \mathcal{E}_{h\sigma}\left(t\right)$ defined in \eqref{eq:Energy_Discrete_first} is an energy and    the solutions of the discrete system \eqref{eq:Maxwell_PML_WaveGuide_Discrete_1}--\eqref{eq:Maxwell_PML_WaveGuide_Discrete_4} or \eqref{eq:Maxwell_PML_WaveGuide_Discrete_first_1}--\eqref{eq:Maxwell_PML_WaveGuide_Discrete_first_3} satisfy the energy estimate
\begin{equation}\label{eq:EnergyEstimate_Discrete}
\begin{split}
\frac{\mathrm{d}}{\mathrm{dt}}\sqrt{\mathcal{E}_{h\sigma}^{(1)}\left(t\right)}  \le &\sigma_{\infty}\sqrt{\mathcal{E}_{h\sigma}^{(1)}\left(t\right)} .
\end{split}
\end{equation}
\end{theorem}
Note that when $\sigma = 0$ we have a strict energy estimate
\begin{equation}\label{eq:EnergyEstimate_Discrete_Standard}
\begin{split}
\frac{\mathrm{d}}{\mathrm{dt}}\sqrt{\mathcal{E}_{h\sigma}^{(1)}\left(t\right)}  \le &0,
\end{split}
\end{equation}
 for any values of $\theta$. However,  for  $\sigma \ge 0$, if  $\theta \ne  1$ we are unable to derive a discrete energy estimate for the discrete PML. In standard numerical codes, as shown in section 3, once the interior scheme is proven stable, the PML  is then added as lower order modification, leading to the standard choice $\theta = 0$. If  $\theta \ne 1$, then there is a non-vanishing boundary term,
 \begin{align}\label{eq:NonVanishing}
  \sigma(1-\theta) \left(I_x \otimes \mathrm{P}_y^{-1}\left(E_{Ry}-E_{Ly}\right)\right)\mathbf {H}_x. 
 \end{align}
  The boundary term \eqref{eq:NonVanishing} will probably diminish with mesh refinement, $h \to 0$. However, on a realistic mesh the boundary term \eqref{eq:NonVanishing} may be nontrivial.  

The energy estimate \eqref{eq:EnergyEstimate_Discrete} above does not exclude the possibility of exponentially growing solutions for $\sigma > 0$. However, the estimate can be extended to the error equations to proof convergence of the solutions of the discrete approximation \eqref{eq:Maxwell_PML_WaveGuide_Discrete_1}--\eqref{eq:Maxwell_PML_WaveGuide_Discrete_4} or \eqref{eq:Maxwell_PML_WaveGuide_Discrete_first_1}--\eqref{eq:Maxwell_PML_WaveGuide_Discrete_first_3} for any finite time interval $[0, T]$. Below, we will prove that at constant coefficients the semi--discrete PML problem  \eqref{eq:Maxwell_PML_WaveGuide_Discrete_1}--\eqref{eq:Maxwell_PML_WaveGuide_Discrete_4} or \eqref{eq:Maxwell_PML_WaveGuide_Discrete_first_1}--\eqref{eq:Maxwell_PML_WaveGuide_Discrete_first_3} can not support exponentially growing solutions.
%
\subsubsection{Discrete energy estimates in the Laplace space}
Consider now the semi--discrete system \eqref{eq:Maxwell_PML_WaveGuide_Discrete_first_1}--\eqref{eq:Maxwell_PML_WaveGuide_Discrete_first_3}  and take the Laplace transform in time, we have
\begin{subequations}\label{eq:Maxwell_PML_WaveGuide_Discrete_first_Laplace}
    \begin{alignat}{2}
      s^2\widehat{\mathbf{E}}_z  &= -\left(\mathrm{P}_x^{-1}D_x^T \otimes I_y\right)\left(\mathrm{P}_x \otimes I_y\right)\left(\left(D_x \otimes I_y\right) \widehat{\mathbf{E}}_z +  \mathbf{\sigma} \widehat{\mathbf{H}}_y\right) - \left( I_x\otimes\mathrm{P}_y^{-1}D_y^T \right)\left(I_x \otimes \mathrm{P}_y\right)\left(\left(I_x \otimes D_y\right) \widehat{\mathbf{E}}_z +  \mathbf{\sigma} \widehat{\mathbf{H}}_x\right) 
  - \sigma s \widehat{\mathbf{E}}_z \notag \\ 
  &+ \sigma(1-\theta) \left(I_x \otimes \mathrm{P}_y^{-1}\left(E_{Ry}-E_{Ly}\right)\right)\mathbf {H}_x  - \left(\mathrm{P}_x^{-1}\left(E_{Rx}+E_{Lx}\right)\otimes I_y + \left(I_x \otimes \mathrm{P}_y^{-1}\left(E_{Ry}+E_{Ly}\right)\right)\right)s \widehat{\mathbf{E}}_z\notag\\
  & - \sigma\theta\left(I_x \otimes \mathrm{P}_y^{-1}\left(E_{Ry}+E_{Ly}\right)\right) \widehat{\mathbf{E}}_z,\label{eq:Maxwell_PML_WaveGuide_Discrete_first_1_Laplace} \\
       s\widehat{\mathbf{H}}_y &= -\left(D_x \otimes I_y\right) \widehat{\mathbf{E}}_z- \mathbf{\sigma} \widehat{\mathbf{H}}_y , \label{eq:Maxwell_PML_WaveGuide_Discrete_first_2_Laplace} \\
      s\widehat{\mathbf{H}}_x  &= \left(I_x \otimes D_y\right) \widehat{\mathbf{E}}_z.\label{eq:Maxwell_PML_WaveGuide_Discrete_first_3_Laplace}
    \end{alignat}
  \end{subequations}
Here, $s$ with $\Re{s} > 0$, is the dual variable to time. Note that in \eqref{eq:Maxwell_PML_WaveGuide_Discrete_first_Laplace} we have also used the SBP properties \eqref{eq:SBPDx} and \eqref{eq:SBPD2x} to eliminate some of the SAT terms.  We can also  eliminate the magnetic fields in \eqref{eq:Maxwell_PML_WaveGuide_Discrete_first_1_Laplace} using \eqref{eq:Maxwell_PML_WaveGuide_Discrete_first_2_Laplace}, \eqref{eq:Maxwell_PML_WaveGuide_Discrete_first_3_Laplace}  and obtain
    \begin{align}\label{eq:Maxwell_PML_WaveGuide_Discrete_second_Laplace}
      s^2\widehat{\mathbf{E}}_z  &= -\frac{1}{S_x^2}\mathbf{P}_{xy}^{-1}\left(D_x^T \otimes I_y\right)\mathbf{P}_{xy}\left(\left(D_x \otimes I_y\right) \widehat{\mathbf{E}}_z\right) - \mathbf{P}_{xy}^{-1}\left( I_x\otimes D_y^T \right)\mathbf{P}_{xy}\left(\left(I_x \otimes D_y\right) \widehat{\mathbf{E}}_z\right)  \notag \\ 
  &+ \frac{\sigma}{sS_x}(1-\theta) \mathbf{P}_{xy}^{-1}\left(\mathrm{P}_x \otimes \left(E_{Ry}-E_{Ly}\right)\right)\left(I_x \otimes D_y\right) \widehat{\mathbf{E}}_z  -  \frac{s}{S_x} \mathbf{P}_{xy}^{-1}\left(\left(E_{Rx}+E_{Lx}\right)\otimes \mathrm{P}_y + \left(\mathrm{P}_x \otimes\left(E_{Ry}+E_{Ly}\right)\right)\right) \widehat{\mathbf{E}}_z\notag\\
  & -  \frac{\sigma}{S_x}\theta \mathbf{P}_{xy}^{-1}\left(\mathrm{P}_x \otimes \left(E_{Ry}+E_{Ly}\right)\right) \widehat{\mathbf{E}}_z,
    \end{align}
 where $S_x = 1 + \sigma/s$. 
 We define the discrete quantity
 \begin{align}
\widehat{\mathcal{E}}^{(1)}_{h\sigma}(s) &= \Re{s}|s|^2\widehat{\mathbf{E}}_z^*\mathbf{P}_{xy}\widehat{\mathbf{E}}_z + \frac{1}{|S_x|^2}\Re{\left(\frac{(sS_x)^*}{S_x}\right)}\left(\left(D_x \otimes I_y\right)\widehat{\mathbf{E}}_z\right)^*\mathbf{P}_{xy} \left(\left(D_x \otimes I_y\right)\widehat{\mathbf{E}}_z\right) \notag\\
&+ \Re{s}\left(\left(I_x \otimes D_y\right)\widehat{\mathbf{E}}_z\right)^*\mathbf{P}_{xy} \left(\left(I_x \otimes D_y\right)\widehat{\mathbf{E}}_z\right)
  +|s|^2\Re{\left(\frac{1}{S_x}\right)} \widehat{\mathbf{E}}_z^*\left(\left(E_{Rx}+E_{Lx}\right)\otimes \mathrm{P}_y + \left(\mathrm{P}_x \otimes \left(E_{Ry}+E_{Ly}\right)\right)\right) \widehat{\mathbf{E}}_z \notag \\
  &+ \theta\sigma \Re\left(\frac{s^*}{S_x}\right)\widehat{\mathbf{E}}_z^*\left(\mathrm{P}_x \otimes \left(E_{Ry}+E_{Ly}\right)\right) \widehat{\mathbf{E}}_z.
 \end{align}
Note that  it can be shown that 
\[
\Re{s} > 0 \implies \Re{\left(\frac{(sS_x)^*}{S_x}\right)},\Re{\left(\frac{1}{S_x}\right)}, \Re\left(\frac{s^*}{S_x}\right) > 0,
\] 
see lemma \ref{Lem:Duru3} in the Appendix. Therefore, if $\theta \ge 0$ then $\widehat{\mathcal{E}}^{(1)}_{h\sigma}(s)  > 0$ is an energy.
 
We can now prove the following result establishing the stability of the discretization \eqref{eq:Maxwell_PML1_WaveGuide_Stable}.
 \begin{theorem}\label{theo:asymptotic_stability}
 Consider the constant coefficient semi--discrete PML \eqref{eq:Maxwell_PML_WaveGuide_Discrete_first_Laplace} with $\Re{s} > 0$, $\sigma \ge 0$ and $\theta = 1$. All solutions of \eqref{eq:Maxwell_PML_WaveGuide_Discrete_first_Laplace} satisfy the energy equation
 \begin{align}\label{eq:asymptotic_stability}
 \widehat{\mathcal{E}}^{(1)}_{h\sigma}(s) = 0.
 \end{align}
 \end{theorem}
 \begin{proof}
 Multiply equation \eqref{eq:Maxwell_PML_WaveGuide_Discrete_second_Laplace} from the left by $ \left(s\widehat{\mathbf{E}}_z\right)^*\mathbf{P}_{xy}$ and add the complex conjugate of the product we have 
 \begin{align}\label{eq:asymp_1}
& \Re{s}|s|^2\widehat{\mathbf{E}}_z^*\mathbf{P}_{xy}\widehat{\mathbf{E}}_z + \frac{1}{|S_x|^2}\Re{\left(\frac{(sS_x)^*}{S_x}\right)}\left(\left(D_x \otimes I_y\right)\widehat{\mathbf{E}}_z\right)^*\mathbf{P}_{xy} \left(\left(D_x \otimes I_y\right)\widehat{\mathbf{E}}_z\right) + \Re{s}\left(\left(I_x \otimes D_y\right)\widehat{\mathbf{E}}_z\right)^*\mathbf{P}_{xy} \left(\left(I_x \otimes D_y\right)\widehat{\mathbf{E}}_z\right)\notag\\
 & |s|^2\Re{\left(\frac{1}{S_x}\right)} \widehat{\mathbf{E}}_z^*\left(\left(E_{Rx}+E_{Lx}\right)\otimes \mathrm{P}_y + \left(\mathrm{P}_x \otimes \left(E_{Ry}+E_{Ly}\right)\right)\right) \widehat{\mathbf{E}}_z + \theta\sigma \Re\left(\frac{s^*}{S_x}\right)\widehat{\mathbf{E}}_z^*\left(\mathrm{P}_x \otimes \left(E_{Ry}+E_{Ly}\right)\right) \widehat{\mathbf{E}}_z\notag\\
 & = (1-\theta)\left(\frac{\sigma s^*}{sS_x} \widehat{\mathbf{E}}_z ^*\left(\mathrm{P}_x \otimes \left(E_{Ry}-E_{Ly}\right)\right)\left(I_x \otimes D_y\right) \widehat{\mathbf{E}}_z 
 + \frac{\sigma s}{s^*S_x^*} \widehat{\mathbf{E}}_z ^*\left(I_x \otimes D_y^T\right)\left(\mathrm{P}_x \otimes \left(E_{Ry}-E_{Ly}\right)\right) \widehat{\mathbf{E}}_z \right).
 \end{align}
On the left hand side of \eqref{eq:asymp_1} we recognize the energy $\widehat{\mathcal{E}}^{(1)}_{h\sigma}(s)$. On the right hand side, if $\theta = 1$, then all the expressions vanish, yielding
\[
\widehat{\mathcal{E}}^{(1)}_{h\sigma}(s) = 0.
\]
The proof of the theorem is complete.
\end{proof}

By theorem \eqref{theo:asymptotic_stability} no nontrivial solution of the discrete problem  \eqref{eq:Maxwell_PML1_WaveGuide_Stable} can grow exponentially for all $\sigma \ge 0$ and $\theta = 1.$ However, for  $\sigma \ge 0$ and $\theta \ne 1$ there are indefinite  non--vanishing boundary terms.
As we saw in the previous experiments these terms can ruin the accuracy of numerical solutions.
\subsection{Discrete approximation  and discrete energy estimates for the  split--field PML}
Here, we develop a stable numerical method for the split--field PML \eqref{eq:Maxwell_PML_Split_1}. Since the split--field PML \eqref{eq:Maxwell_PML_Split_1} can be rewritten as the modal unsplit \eqref{eq:Maxwell_PML_WaveGuide_1}--\eqref{eq:Maxwell_PML_WaveGuide_4}, our task is made simpler. We will design a numerical method for \eqref{eq:Maxwell_PML_Split_1} such that the corresponding discrete problem can be rewritten as the stable discrete modal PML \eqref{eq:Maxwell_PML1_WaveGuide_Stable} with $\theta = 1$. We can then apply theorem \ref{Theorem:Discrete_Stability} and theorem \ref{theo:asymptotic_stability} to prove the stability of the discretization.

The corresponding stable semi-discretization for the split--field PML   \eqref{eq:Maxwell_PML_Split_11}--\eqref{eq:Maxwell_PML_Split_14} is 
\begin{subequations}\label{eq:Maxwell_StableSplitFieldPML}
    \begin{alignat}{2}
       \frac{\mathrm{d}{\mathbf{E}_z^{(x)}}}{\mathrm{d t}} &= -\left(D_x \otimes I_y\right)\mathbf {H}_y  - \mathbf{\sigma}\mathbf{E}_z^{(y)} 
       \underbrace{-\left(\left(\mathrm{P}_x^{-1}\left(E_{Rx}+E_{Lx}\right)\otimes I_y \right)\left(\mathbf {E}_z^{(x)} + \mathbf {E}_z^{(y)}  \right)
 - \left(\mathrm{P}_x^{-1}\left(E_{Rx}-E_{Lx}\right)\otimes I_y \right)\mathbf {H}_y\right)}_{\mathrm{SAT}_x} ,
  \label{eq:Maxwell_StableSplitFieldPML1}\\
      \frac{\mathrm{d}{\mathbf{H}_y}}{\mathrm{d t}}  &= -\left(D_x \otimes I_y\right)\left(\mathbf {E}_z^{(x)} + \mathbf {E}_z^{(y)}  \right)- \mathbf{\sigma}\mathbf {H}_y, \label{eq:Maxwell_StableSplitFieldPML3} \\
     \frac{\mathrm{d}{\mathbf{H}_x}}{\mathrm{d t}}  &= \left(I_x \otimes D_y\right)\left(\mathbf {E}_z^{(x)} + \mathbf {E}_z^{(y)}  \right),\label{eq:Maxwell_StableSplitFieldPML4} \\
     \frac{\mathrm{d}{\mathbf{E}_z^{(y)}}}{\mathrm{d t}} &=  \left(I_x \otimes D_y\right)\mathbf {H}_x \underbrace{- \left( \left(I_x \otimes \mathrm{P}_y^{-1}\left(E_{Ry}+E_{Ly}\right)\right)\left(\mathbf {E}_z^{(x)} + \mathbf {E}_z^{(y)}  \right) +\left(I_x \otimes \mathrm{P}_y^{-1}\left(E_{Ry}-E_{Ly}\right)\right)\mathbf {H}_x\right)}_{\mathrm{SAT}_y}. \label{eq:Maxwell_StableSplitFieldPML2}
    \end{alignat}
  \end{subequations}
Note that we have moved the SAT-term $\mathrm{SAT}_y$ to the last equation.
This is important since we can rewrite the discrete approximation \eqref{eq:Maxwell_StableSplitFieldPML} as the stable discrete unsplit PML \eqref{eq:Maxwell_PML1_WaveGuide_Stable} with $\theta = 1$.

Similarly, as in the continuous case, using  the identity $\mathbf{E}_z  = \mathbf{E}_z^{(x)} + \mathbf{E}_z^{(y)} \implies \mathbf{E}_z^{(x)}  = \mathbf{E}_z - \mathbf{E}_z^{(y)}$ we can eliminate the split variable $\mathbf{E}_z^{(x)}$, obtaining
\begin{subequations}\label{eq:Maxwell_StableSplitFieldPML_modal}
    \begin{alignat}{2}
       \frac{\mathrm{d}{\mathbf{E}_z}}{\mathrm{d t}} &= -\left(D_x \otimes I_y\right)\mathbf {H}_y + \left(I_x \otimes D_y\right)\mathbf {H}_x + \mathbf{\sigma}\mathbf{E}_z ^{(y)}  - \mathbf{\sigma}\mathbf{E}_z  
       \underbrace{-\left(\left(\mathrm{P}_x^{-1}\left(E_{Rx}+E_{Lx}\right)\otimes I_y \right)\mathbf {E}_z
 - \left(\mathrm{P}_x^{-1}\left(E_{Rx}-E_{Lx}\right)\otimes I_y \right)\mathbf {H}_y\right)}_{\mathrm{SAT}_x}  \notag \\ 
& \underbrace{- \left( \left(I_x \otimes \mathrm{P}_y^{-1}\left(E_{Ry}+E_{Ly}\right)\right)\mathbf {E}_z +\left(I_x \otimes \mathrm{P}_y^{-1}\left(E_{Ry}-E_{Ly}\right)\right)\mathbf {H}_x\right)}_{\mathrm{SAT}_y},
  \label{eq:Maxwell_StableSplitFieldPML1_modal}\\
      \frac{\mathrm{d}{\mathbf{H}_y}}{\mathrm{d t}}  &= -\left(D_x \otimes I_y\right)\mathbf {E}_z- \mathbf{\sigma}\mathbf {H}_y, \label{eq:Maxwell_StableSplitFieldPML3_modal} \\
     \frac{\mathrm{d}{\mathbf{H}_x}}{\mathrm{d t}}  &= \left(I_x \otimes D_y\right)\mathbf {E}_z,
     \label{eq:Maxwell_StableSplitFieldPML4_modal} \\
      \frac{\mathrm{d}{\mathbf{E}_z^{(y)}}}{\mathrm{d t}} &=  \left(I_x \otimes D_y\right)\mathbf {H}_x \underbrace{- \left( \left(I_x \otimes \mathrm{P}_y^{-1}\left(E_{Ry}+E_{Ly}\right)\right)\mathbf {E}_z +\left(I_x \otimes \mathrm{P}_y^{-1}\left(E_{Ry}-E_{Ly}\right)\right)\mathbf {H}_x\right)}_{\text{Stabilizing term}}, \label{eq:Maxwell_StableSplitFieldPML2_modal}
    \end{alignat}
  \end{subequations}
Note that equation  \eqref{eq:Maxwell_StableSplitFieldPML_modal} above corresponds to discretizing the unsplit PML \eqref{eq:Maxwell_PML_Split_2}. Now multiply equation \eqref{eq:Maxwell_StableSplitFieldPML2_modal} by $\sigma$ and introduce  $\mathbf{H}_y^* = \sigma \mathbf{E}_z^{(y)} $, we obtain exactly 
the stable discrete unsplit PML \eqref{eq:Maxwell_PML1_WaveGuide_Stable} with $\theta = 1$. We can now apply theorem \ref{Theorem:Discrete_Stability} and theorem \ref{theo:asymptotic_stability} to prove the temporal stability of the discretizations \eqref{eq:Maxwell_StableSplitFieldPML} and \eqref{eq:Maxwell_StableSplitFieldPML_modal}. These results are clearly demonstrated by the numerical experiments performed in the next section.

\subsection{Discrete approximation  and discrete energy estimates for the  physically motivated PML}
 Now, we turn our attention to  the physically motivated PML model \eqref{eq:Maxwell_PML_Ababarnel_1}--\eqref{eq:Maxwell_PML_Ababarnel_4}. We know that the principal part of this PML is symmetric hyperbolic. The lower order terms appear in the PML in a nontrivial way. As  was shown in  \cite{Abarbanel1998, Abarbanel2002, Abarbanel2009} (and we have already seen in section \ref{sect:num_example}), the PML terms can pose a numerical challenge.  However, we will show that using the SBP--SAT scheme proposed in section \ref{sec:SBP_SAT},  we can choose penalties such that any stable numerical boundary procedure for the undamped problem (with $\sigma \equiv 0$) satisfying an energy estimate also  leads to  a discrete energy estimate,  analogous to continuous estimate \eqref{eq:EnergyEstimate_continuous_physical}, for the discrete PML  \eqref{eq:Maxwell_PML_Ababarnel_Discrete_11}--\eqref{eq:Maxwell_PML_Ababarnel_Discrete_41}. 
 However, the continuous energy \eqref{eq:EnergyEstimate_continuous_physical} for the PML  and the corresponding discrete energy for the discrete PML can grow exponentially in time. By theorem \eqref{Theorem:Wellposedness_Laplace}, we also know the constant coefficient PML model \eqref{eq:Maxwell_PML_Ababarnel_1}--\eqref{eq:Maxwell_PML_Ababarnel_4} subject to the  boundary conditions  (\ref{eq:wall_bc_y_0}),  (\ref{eq:wall_bc_x}) can not support exponentially growing solutions.  Therefore, at constant coefficients, for a particular set of penalty parameters, we will show that the solutions for the discrete PML can not grow exponentially in time.
 
%

To begin, consider the discrete PML \eqref{eq:Maxwell_PML_Ababarnel_Discrete_11}--\eqref{eq:Maxwell_PML_Ababarnel_Discrete_41}. 
By applying the energy method directly to \eqref{eq:Maxwell_PML_Ababarnel_Discrete_11}--\eqref{eq:Maxwell_PML_Ababarnel_Discrete_41}, we have
\begin{equation}\label{eq:Discrete_Dissipation}
\begin{split}
& \frac{d}{dt}\left(\left\|  \mathbf{E}_z\left(t\right) \right\|^2_{\mathbf{P}_{xy}} +  \left\|  \mathbf{H}_y\left(t\right) \right\|^2_{\mathbf{P}_{xy}}  +  \left\|  \mathbf{H}_x\left(t\right) \right\|^2_{\mathbf{P}_{xy}} + \left\|  \mathbf{P}\left(t\right) \right\|^2_{\mathbf{P}_{xy}}\right) =  -2\sigma \left\|  \mathbf{E}_z\left(t\right) \right\|^2_{\mathbf{P}_{xy}}  -2\sigma \left\|  \mathbf{H}_y\left(t\right) \right\|^2_{\mathbf{P}_{xy}}  +2 \sigma \left\|  \mathbf{H}_x\left(t\right) \right\|^2_{\mathbf{P}_{xy}} \\
&-2\sigma \left\|  \mathbf{P}\left(t\right) \right\|^2_{\mathbf{P}_{xy}}
-\textbf{BT}_s(t).
\end{split}
\end{equation}
The boundary terms $\textbf{BT}_s$ are exactly the same as the boundary terms derived for the undamped problem in section \ref{sec:SBP_SAT}.
As before, if  $R_x \ne -1$ and $R_y \ne -1$, then we can choose 
\[
\alpha_x = \frac{2}{1+R_x}, \quad \alpha_y = \frac{2}{1+R_y}, \quad \theta_x = \frac{2\bar{\theta}_x}{1+R_x}, \quad \theta_y =  \frac{2\bar{\theta}_y}{1+R_y},
\]
and
\[
0\le \bar{\theta}_x\le \frac{4}{\gamma_x}, \quad 0\le \bar{\theta}_y\le \frac{4}{\gamma_y},
\]
we have $\textbf{BT}_s(t) \ge 0$.
Also, for any $|R_x| \le 1$ and $|R_y| \le 1$, then we can choose 
\[
\alpha_x =  \alpha_y =  \theta_x = \theta_y = 1,
\]
obtaining $\textbf{BT}_s(t) \ge 0$.
Note that all terms in the right hand side of \eqref{eq:Discrete_Dissipation}, excepting $2\sigma \left\|  \mathbf{H}_x\left(t\right) \right\|^2_{\mathbf{P}_{xy}}$, are dissipative.
Introducing the  discrete energy 
\begin{equation}\label{eq:Energy_Discrete_first_Abarbanel}
\begin{split}
 \mathcal{E}_{h}^{(2)}\left(t\right) &= \left\|  \mathbf{E}_z\left(t\right) \right\|^2_{\mathbf{P}_{xy}} +  \left\|  \mathbf{H}_y\left(t\right) \right\|^2_{\mathbf{P}_{xy}}  +  \left\|  \mathbf{H}_x\left(t\right) \right\|^2_{\mathbf{P}_{xy}} + \left\|  \mathbf{P}\left(t\right) \right\|^2_{\mathbf{P}_{xy}}  + \int_0^t\mathbf{BT}_s\left(\mathrm{t}'\right)\mathrm{dt}',
  \end{split}
  \end{equation}
we have
\begin{theorem}\label{Theorem:Discrete_Stability2}
Consider the discrete approximation \eqref{eq:Maxwell_PML2_WaveGuide} for the physically motivated PML \eqref{eq:Maxwell_PML_Ababarnel_1}--\eqref{eq:Maxwell_PML_Ababarnel_4}   subject to the  boundary conditions  (\ref{eq:wall_bc_y_0})  at $y = \pm y_0$ and  (\ref{eq:wall_bc_x}) at $x = \pm \left(x_0+\delta\right) $. 
For 
\begin{itemize}
\item
 $R_x \ne -1$ and $R_y \ne -1$, and  
\[
\alpha_x = \frac{2}{1+R_x}, \quad \alpha_y = \frac{2}{1+R_y}, \quad \theta_x = \frac{2\bar{\theta}_x}{1+R_x}, \quad \theta_y =  \frac{2\bar{\theta}_y}{1+R_y},
\]
with 
\[
0\le \bar{\theta}_x\le \frac{4}{\gamma_x}, \quad 0\le \bar{\theta}_y\le \frac{4}{\gamma_y},
\]
or 
\item
  $|R_x| \le 1$ and $|R_y| \le 1$, and
\[
\alpha_x =  \alpha_y =  \theta_x =  \theta_y =  1,
\]
\end{itemize}
 then $\textbf{BT}_s(t) \ge 0$ and the quantity $  \mathcal{E}_{h}^{(2)}\left(t\right)$ defined in \eqref{eq:Energy_Discrete_first_Abarbanel} is an energy. For these sets of penalty parameters   the solutions of  the discrete PML problem \eqref{eq:Maxwell_PML2_WaveGuide}  satisfy the energy estimate
\begin{equation}\label{eq:EnergyEstimate_Discrete_Abarbanel}
\begin{split}
\frac{\mathrm{d}}{\mathrm{dt}}\sqrt{\mathcal{E}_{h}^{(2)}\left(t\right)}  \le & \sigma_{\infty}\sqrt{\mathcal{E}_{h}^{(2)}\left(t\right)} .
\end{split}
\end{equation}
\end{theorem}

By \eqref{eq:EnergyEstimate_Discrete_Abarbanel},  the discrete energy can grow exponentially in time. However, the modal analysis presented in section \ref{section:modal_analysis} shows  that  the continuous PML problem  \eqref{eq:Maxwell_PML_Ababarnel_1}--\eqref{eq:Maxwell_PML_Ababarnel_4}   subject to the  boundary conditions  (\ref{eq:wall_bc_y})  at $y = \pm y_0$ and  (\ref{eq:wall_bc_x}) at $x = \pm \left(x_0+\delta\right) $ does not support temporally growing modes. Therefore, in the discrete setting, it will be more desirable to  design  numerical methods that ensure that all eigenvalues of the discrete spatial operator satisfy  $\Re{\lambda} \le 0$ for any $h > 0$. This  has the potential to  eliminate any nonphysical  temporal growth. At constant coefficients we will show that the  corresponding discretization used in  section \ref{sect:num_example} can not allow growing modes
\subsubsection{Energy estimates in the Laplace space}
Consider the discrete approximation \eqref{eq:Maxwell_PML2_WaveGuide} with $R_x = R_y = 0$ and 
\[
\alpha_x = 2, \quad \alpha_y = 2, \quad \theta_x = 0, \quad \theta_y =  0,
\]
we have
\begin{subequations}\label{eq:Maxwell_PML2_WaveGuide_Stable}
    \begin{alignat}{2}
       \frac{\mathrm{d}{\mathbf{E}_z}}{\mathrm{d t}} &= -\left(D_x \otimes I_y\right)\mathbf {H}_y + \left(I_x \otimes D_y\right)\mathbf {H}_x  - \sigma \mathbf {E}_z 
      \underbrace{-\left(\left(\mathrm{P}_x^{-1}\left(E_{Rx}+E_{Lx}\right)\otimes I_y \right)\mathbf {E}_z  
 - \left(\mathrm{P}_x^{-1}\left(E_{Rx}-E_{Lx}\right)\otimes I_y \right)\mathbf {H}_y\right)}_{\mathrm{SAT}_x} \notag \\
 & \underbrace{- \left( \left(I_x \otimes \mathrm{P}_y^{-1}\left(E_{Ry}+E_{Ly}\right)\right)\mathbf {E}_z +\left(I_x \otimes \mathrm{P}_y^{-1}\left(E_{Ry}-E_{Ly}\right)\right)\mathbf {H}_x\right)}_{\mathrm{SAT}_y},
 \label{eq:Maxwell_PML_Ababarnel_Discrete_11_Stable}\\
      \frac{\mathrm{d}{\mathbf{H}_y}}{\mathrm{d t}}  &= -\left(D_x \otimes I_y\right)\mathbf {E}_z- \mathbf{\sigma}\mathbf {H}_y  ,
  \label{eq:Maxwell_PML_Ababarnel_Discrete_21_Stable} \\
     \frac{\mathrm{d}{\mathbf{H}_x}}{\mathrm{d t}}  &= \left(I_x \otimes D_y\right)\mathbf {E}_z +\sigma \left(\mathbf {H}_x - \mathbf {P}\right) ,
      \notag \\ \label{eq:Maxwell_PML_Ababarnel_Discrete_31_Stable} \\ 
   \frac{\mathrm{d}{\mathbf{P}}}{\mathrm{d t}}   &= \sigma \left(\mathbf {H}_x - \mathbf {P}\right).\label{eq:Maxwell_PML_Ababarnel_Discrete_41_Stable}
    \end{alignat}
  \end{subequations}
  
 After Laplace transformation of \eqref{eq:Maxwell_PML2_WaveGuide_Stable}  in time, we can   eliminate the auxiliary variable and the magnetic fields  obtaining
    \begin{align}\label{eq:Maxwell_PML_WaveGuide_Discrete_Laplace_Abarbanel}
      s^2\widehat{\mathbf{E}}_z  &= -\frac{1}{S_x^2}\mathbf{P}_{xy}^{-1}\left(D_x^T \otimes I_y\right)\mathbf{P}_{xy}\left(\left(D_x \otimes I_y\right) \widehat{\mathbf{E}}_z\right) - \mathbf{P}_{xy}^{-1}\left( I_x\otimes D_y^T \right)\mathbf{P}_{xy}\left(\left(I_x \otimes D_y\right) \widehat{\mathbf{E}}_z\right)  \notag \\ 
  &  -  \frac{s}{S_x} \mathbf{P}_{xy}^{-1}\left(\left(E_{Rx}+E_{Lx}\right)\otimes \mathrm{P}_y + \left(\mathrm{P}_x \otimes\left(E_{Ry}+E_{Ly}\right)\right)\right) \widehat{\mathbf{E}}_z, 
    \end{align}
 where $S_x = 1 + \sigma/s$ and $\Re{s} > 0$. As in \eqref{eq:Maxwell_PML_WaveGuide_Discrete_second_Laplace}, note also that we have used the SBP property \eqref{eq:SBPDx} to eliminate some of the SAT boundary terms.
 We define the corresponding discrete energy
 \begin{align}
\widehat{\mathcal{E}}^{(2)}_{h\sigma}(s) &= \Re{s}|s|^2\widehat{\mathbf{E}}_z^*\mathbf{P}_{xy}\widehat{\mathbf{E}}_z + \frac{1}{|S_x|^2}\Re{\left(\frac{(sS_x)^*}{S_x}\right)}\left(\left(D_x \otimes I_y\right)\widehat{\mathbf{E}}_z\right)^*\mathbf{P}_{xy} \left(\left(D_x \otimes I_y\right)\widehat{\mathbf{E}}_z\right) \notag\\
&+ \Re{s}\left(\left(I_x \otimes D_y\right)\widehat{\mathbf{E}}_z\right)^*\mathbf{P}_{xy} \left(\left(I_x \otimes D_y\right)\widehat{\mathbf{E}}_z\right)
  +|s|^2\Re{\left(\frac{1}{S_x}\right)} \widehat{\mathbf{E}}_z^*\left(\left(E_{Rx}+E_{Lx}\right)\otimes \mathrm{P}_y + \left(\mathrm{P}_x \otimes \left(E_{Ry}+E_{Ly}\right)\right)\right) \widehat{\mathbf{E}}_z .
 \end{align}
Note that  by lemma \ref{Lem:Duru3} in the Appendix, we have
\[
\Re{s} > 0 \implies \Re{\left(\frac{(sS_x)^*}{S_x}\right)},\Re{\left(\frac{1}{S_x}\right)},  > 0.
\] 
 Therefore, $\widehat{\mathcal{E}}^{(2)}_{h\sigma}(s)  > 0$ is an energy.
We can now prove the following result establishing the stability of the discretization \eqref{eq:Maxwell_PML2_WaveGuide_Stable}.
 \begin{theorem}\label{theo:asymptotic_stability_Abarbanel}
 Consider the constant coefficient semi--discrete PML \eqref{eq:Maxwell_PML_WaveGuide_Discrete_Laplace_Abarbanel} with $\Re{s} > 0$, $\sigma \ge 0$. All solutions of \eqref{eq:Maxwell_PML_WaveGuide_Discrete_Laplace_Abarbanel} satisfy the energy equation
 \begin{align}\label{eq:asymptotic_stability_Abarbanel}
 \widehat{\mathcal{E}}^{(2)}_{h\sigma}(s) = 0.
 \end{align}
 \end{theorem}
 The proof of theorem \ref{theo:asymptotic_stability_Abarbanel} is completely analogous to the proof of theorem \ref{theo:asymptotic_stability}. We therefore omit it here.

Note that theorem \ref{theo:asymptotic_stability_Abarbanel} shows that the discrete PML problem \eqref{eq:Maxwell_PML_WaveGuide_Discrete_Laplace_Abarbanel} can not have eigenvalues with positive real parts. However, for efficient explicit time integration we do not only need the temporal eigenvalues to have non positive real parts,   the spectral radius, $\max_i{|\lambda_i|}$, must also be sufficiently small. Note that in  section \ref{sect:num_example},   the penalty parameters
\[
\alpha_x =  \alpha_y = 2, \quad \theta_x =  \theta_y =  0, \quad \text{with} \quad R_x = R_y = 0,
\]
when combined with the physically motivated PML yielded unstable solutions for the time step $dt = 0.4h$. However, when we reduced the time step  to $dt = 0.2h$ the numerical solutions are stable.
This shows that the above penalty parameters for the physically motivated PML is not adequate for efficient explicit time integration. 

It is possible though to construct stable boundary procedures which will not restrict the time step further when the PML is introduced. For the physically motivated PML we propose to use the stable penalty parameters,
\[
\alpha_x =  \alpha_y =  \theta_x =  \theta_y =  1.
\]
Note that the maximum weights of these penalty parameters are smaller. For many choices of boundary parameters $|R_x|, |R_y| \le 1$, this has the potential to reduce numerical stiffness in the discrete PML.
  Numerical experiments performed in the next section confirm this.
\section{Numerical experiments}
In this section, we present some numerical tests  to verify the analysis of  previous sections.  In particular, we will  investigate temporal stability of the discrete PMLs, \eqref{eq:Maxwell_PML_WaveGuide_Discrete_1}--\eqref{eq:Maxwell_PML_WaveGuide_Discrete_4},  \eqref{eq:Maxwell_PML_Ababarnel_Discrete_11}--\eqref{eq:Maxwell_PML_Ababarnel_Discrete_41}, \eqref{eq:Maxwell_StableSplitFieldPML1}--\eqref{eq:Maxwell_StableSplitFieldPML2},  and also evaluate the accuracy of the PML,    in  an electromagnetic waveguide.

To begin with, we will repeat the experiments of section 3.  We consider exactly the same setup as in section 3. Also,  the same discretization parameters,  time step  $\mathrm{dt} = 0.4h$, absorption function and damping coefficient are used. We run the simulations now for a very long time, until $t = 50000$ (125 000 time steps), using the discrete  PMLs, \eqref{eq:Maxwell_PML_WaveGuide_Discrete_1}--\eqref{eq:Maxwell_PML_WaveGuide_Discrete_4} with $\theta = 1$, \eqref{eq:Maxwell_PML_Ababarnel_Discrete_11}--\eqref{eq:Maxwell_PML_Ababarnel_Discrete_41} with $\alpha_x =  \alpha_y =  \theta_x =  \theta_y =  1$ and \eqref{eq:Maxwell_StableSplitFieldPML1}--\eqref{eq:Maxwell_StableSplitFieldPML2},  respectively. The time histories of the $l_2$-norm of the solutions are recorded in figures \ref{fig:Model1_Stable}--\ref{fig:Model3_Stable}   for the discrete PMLs  \eqref{eq:Maxwell_PML_WaveGuide_Discrete_1}--\eqref{eq:Maxwell_PML_WaveGuide_Discrete_4}, \eqref{eq:Maxwell_PML_Ababarnel_Discrete_11}--\eqref{eq:Maxwell_PML_Ababarnel_Discrete_41} and \eqref{eq:Maxwell_StableSplitFieldPML1}--\eqref{eq:Maxwell_StableSplitFieldPML2}, respectively. Clearly, the solutions decay with time throughout the simulation. This verifies the analysis of sections 4 and 5. The numerical growth seen in section 3 are due to inappropriate boundary treatments in the PML. We also believe that some of the growth seen in 
\cite{Abarbanel2002, Abarbanel2009}  are probably caused by  numerical enforcements of boundary conditions.

  It is noteworthy that the behaviors of the discrete PMLs, \eqref{eq:Maxwell_PML_WaveGuide_Discrete_1}--\eqref{eq:Maxwell_PML_WaveGuide_Discrete_4} and \eqref{eq:Maxwell_StableSplitFieldPML1}--\eqref{eq:Maxwell_StableSplitFieldPML2}, are completely analogous.

\begin{figure} [h!]
 \centering
\subfigure[Discrete modal PML \eqref{eq:Maxwell_PML_WaveGuide_Discrete_1}--\eqref{eq:Maxwell_PML_WaveGuide_Discrete_4} ]{\includegraphics[width=0.49\linewidth]{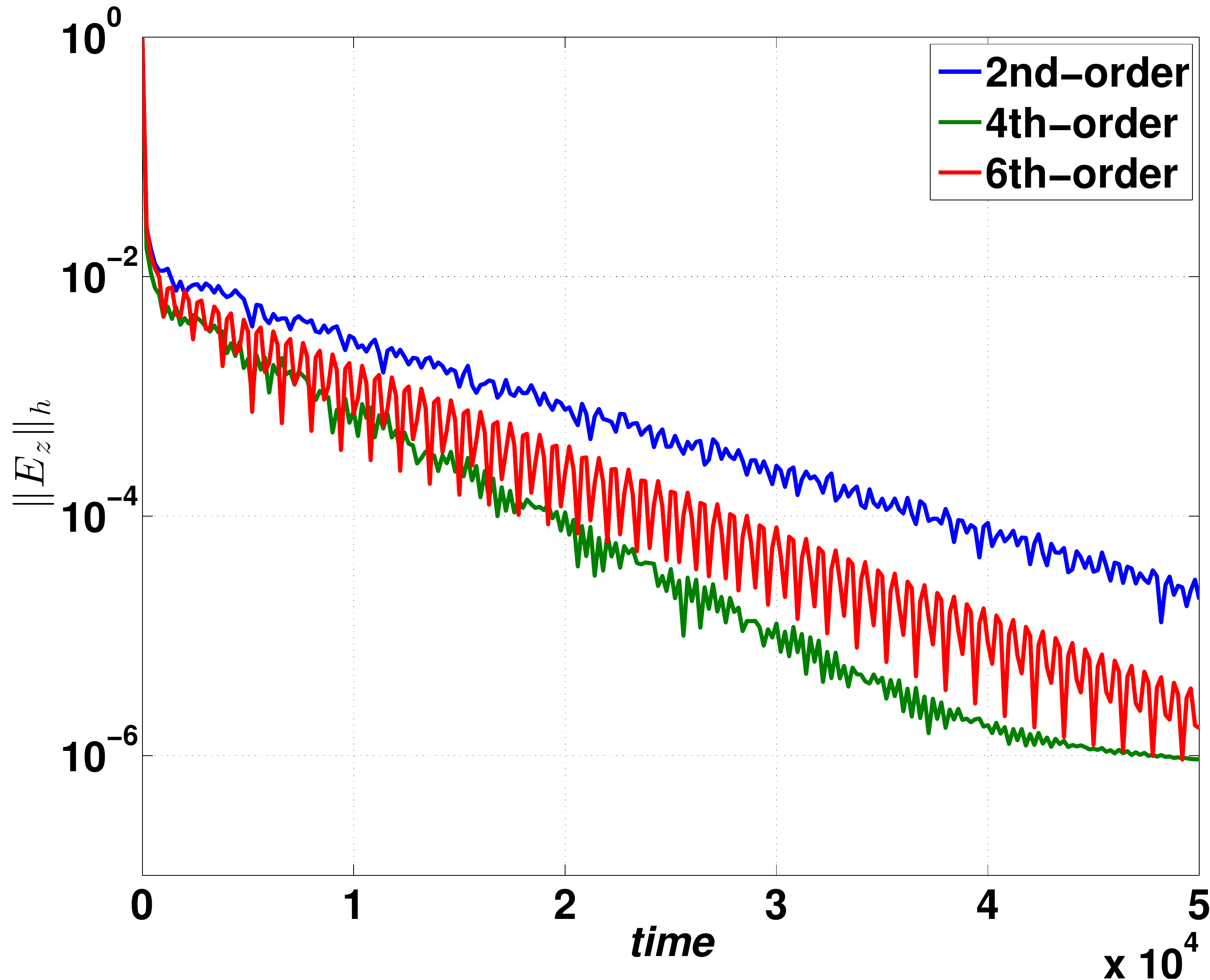}\label{fig:Model1_Stable}}
\subfigure[Discrete physically motivated PML \eqref{eq:Maxwell_PML_Ababarnel_Discrete_11}--\eqref{eq:Maxwell_PML_Ababarnel_Discrete_41}]{\includegraphics[width=0.49\linewidth]{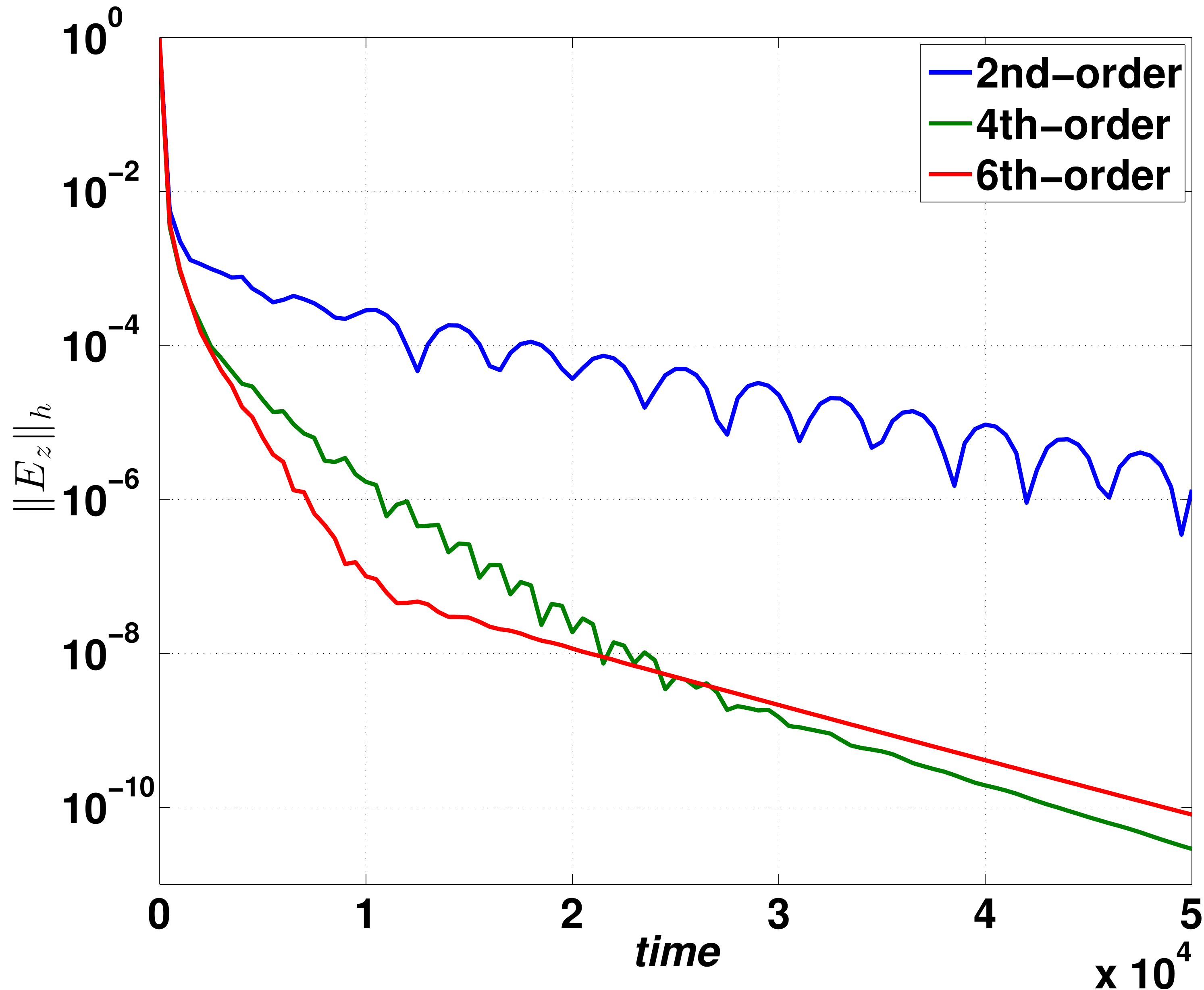}\label{fig:Model2_Stable}}
\subfigure[Discrete split--field PML \eqref{eq:Maxwell_StableSplitFieldPML1}--\eqref{eq:Maxwell_StableSplitFieldPML2} ]{\includegraphics[width=0.49\linewidth]{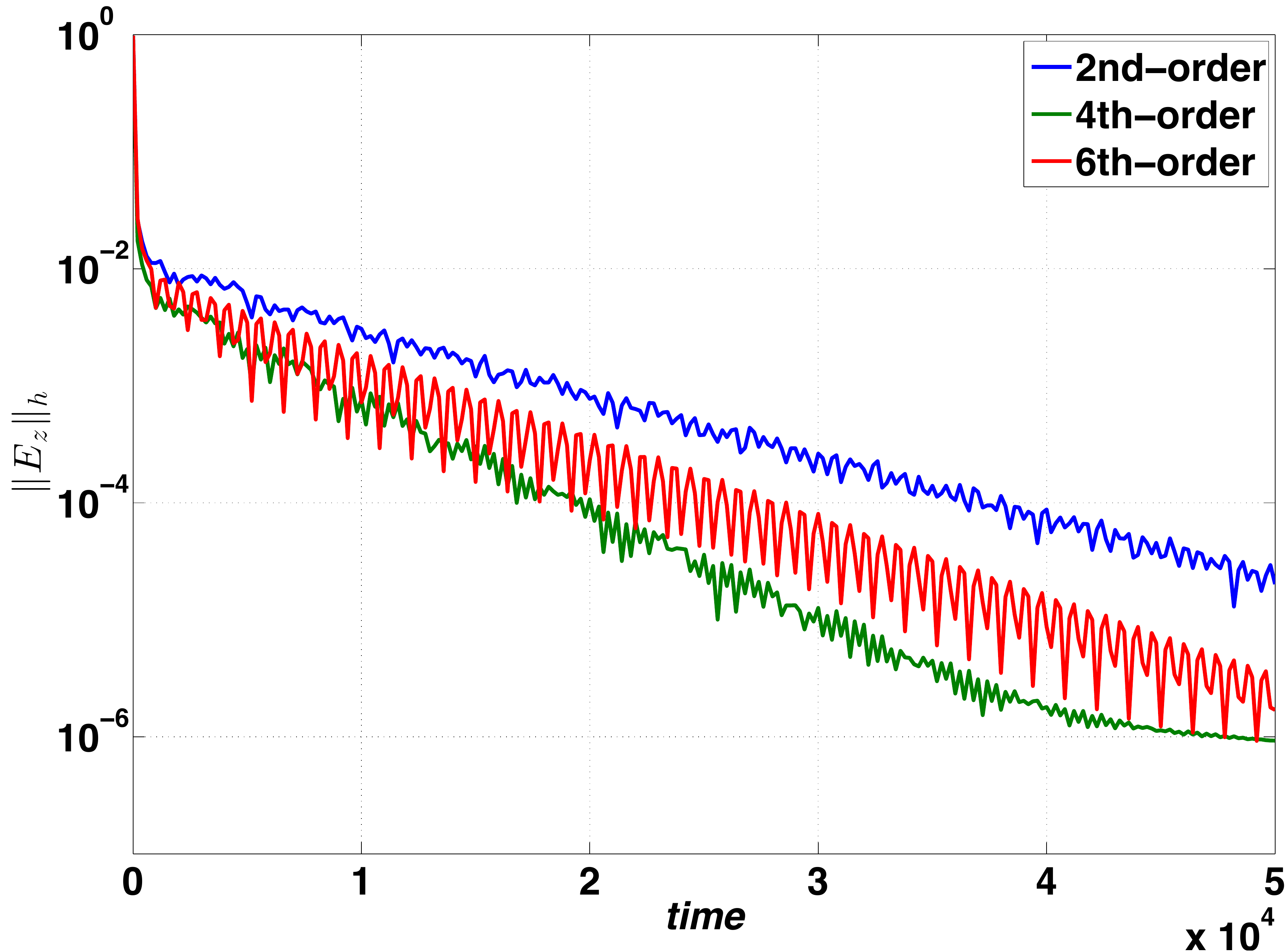}\label{fig:Model3_Stable}}
 \caption{\textit{The $l_2$-norm of the electric fields as a function of time}}
 \label{fig:L2Norm_waveguide}
\end{figure}

 We have also performed numerous numerical experiments with different boundary parameters, $R_x, R_y$ and penalty parameters $\alpha_x, \alpha_y,\theta_x, \theta_y$. Unfortunately, because of space,  we can not report all of them here.
 It is particularly  important to note that applying the penalty $\alpha_x =  \alpha_y =  \theta_x =  \theta_y =  1$ to the discrete modal PML  \eqref{eq:Maxwell_PML_WaveGuide_Discrete_11_Stable_0}--\eqref{eq:Maxwell_PML_WaveGuide_Discrete_41_Stable_0}, with $\theta = 0$, yields unstable solutions. Numerical growth was also observed when the stable boundary treatment  used in  \eqref{eq:Maxwell_PML_WaveGuide_Discrete_1}--\eqref{eq:Maxwell_PML_WaveGuide_Discrete_4} was applied to the physically motivated PML \eqref{eq:Maxwell_PML_Ababarnel_1}--\eqref{eq:Maxwell_PML_Ababarnel_4} without decreasing the time step further. This indicates that each PML model must be treated individually. A stable and efficient numerical boundary treatment for a particular PML model may not necessarily yield stable and efficient numerical scheme for another PML modeling the same physical phenomena. Of course, when the underlying physical model changes, the corresponding PML model will also change. We will also expect different  stable numerical boundary treatments.
 
  In the next section, we will verify the stability and accuracy of the PML in an electromagnetic waveguide.
\subsection{Electromagnetic waveguides}
Here, we will focus on the discrete PML \eqref{eq:Maxwell_PML_WaveGuide_Discrete_1}--\eqref{eq:Maxwell_PML_WaveGuide_Discrete_4}. The results obtained for the discrete PMLs,   \eqref{eq:Maxwell_PML_Ababarnel_Discrete_11}--\eqref{eq:Maxwell_PML_Ababarnel_Discrete_41} and \eqref{eq:Maxwell_StableSplitFieldPML1}--\eqref{eq:Maxwell_StableSplitFieldPML2}, are analogous, and will not be presented.
Consider the rectangular semi--infinite electromagnetic waveguide $-2 \le x  < \infty$,  $-1 \le y \le 1$. At the upper wall, $y = 1,$ of the waveguide, the boundary condition is a magnetic source
$H_x (x, 1, t)= F(x, 1, t)$,  where 
\begin{equation}\label{eq:Forcing} 
F(x, y, t) =  \exp(-\pi^2(f_0t- 1)^2) g(x, y),
\end{equation}
and
\[ 
 g(x, y) = \exp\left(-\left(\left(x - 1\right)^2 + \left(y-1\right)^2\right)/0.01\right), \quad f_0 = 10.
\]
 At the lower wall, $y = -1,$ the boundary condition is homogeneous $H_x (x,  -1, t)= 0$. At the left wall, $x = -2$, we impose the characteristic boundary condition $E_z\left(-2 ,y, t\right) + H_y \left(-2, y, t\right)= 0$. In order to perform numerical simulations, we truncate the right boundary at $x = 2$ with a PML of width $\delta = 0.4$. 
 We terminate the PML at $x = 2 + \delta$ with the characteristic boundary condition $E_z(2 + \delta ,y, t) - H_y (2 + \delta, y, t)= 0$. Homogeneous initial conditions 
 \begin{equation}\label{eq:initialdata} 
E_z(x, y, 0) = H_x(x, y, 0) = H_y(x, y, 0) = H_x^*(x, y, 0) = 0,
\end{equation}
   are used for all variables. We also use the cubic monomial  damping profile 
\begin{equation}\label{eq:Damping_xx}
\begin{split}
&\sigma(x) = \left \{
\begin{array}{rl}
0 \quad {}  \quad {}& \text{if} \quad |x| \le 2,\\
d_0\Big(\frac{|x|-2}{\delta}\Big)^3  & \text{if}  \quad |x| \ge 2,
\end{array} \right.
\end{split}
\end{equation}
   with   the damping coefficient   $d_0 =  4/(2\times \delta )\log(1/(10^{-4}h)^2)$, where the factor $10^{-4}$ is empirically determined and $h>0$ is the spatial step.

We approximate all spatial derivatives using SBP operators with interior order of accuracy 2, 4, 6.   Boundary conditions are imposed weakly using penalties.  As before, we discretize in time  using the classical 4th order accurate Runge--Kutta method. 
\subsection{Numerical stability}
In the first sets of experiments we investigate discrete stability. Consider  SBP finite difference operators of order of accuracy  2, 4, 6, respectively.  We use the spatial step $h = 0.02$ and the time step $dt = 0.4h$, and evolve the solutions until $t = 2000$ (250 000 time steps) with $\theta = 0, 1$, respectively, in \eqref{eq:Maxwell_PML_WaveGuide_Discrete_4}.  Note that here the domain,  spatial and time steps are much   smaller than the ones used in the previous sections. The snapshots of the solutions at $t = 300$ are shown in figure \ref{fig:slution_t300}, for the  4th order accurate SBP operator. 
\begin{figure} [h!]
 \centering
\subfigure[$\theta = 0$]{\includegraphics[width=.475\linewidth]{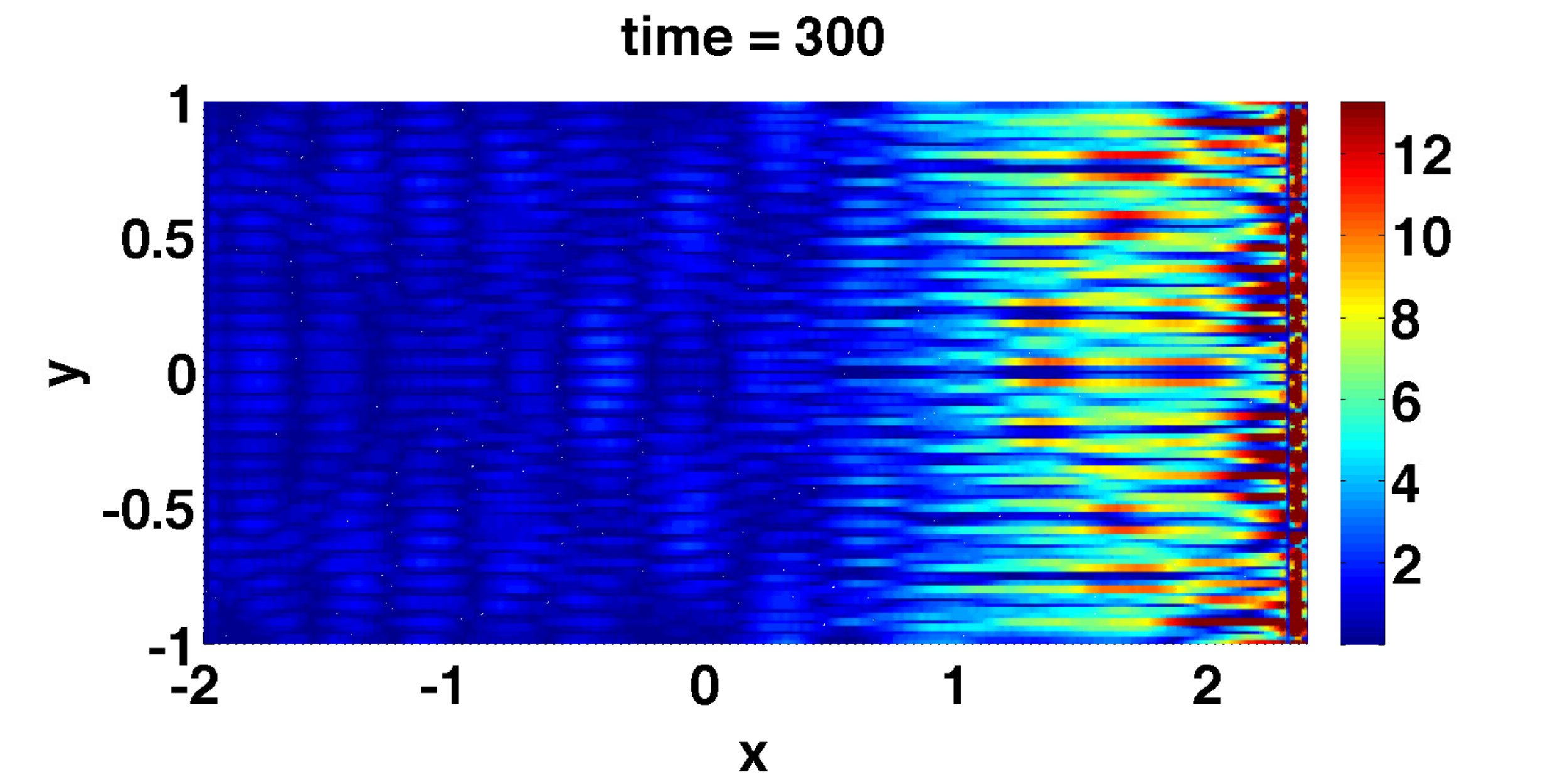} \label{fig:unstable_solution_t300_4th}} 
\subfigure[$\theta = 1$]{\includegraphics[width=.475\linewidth]{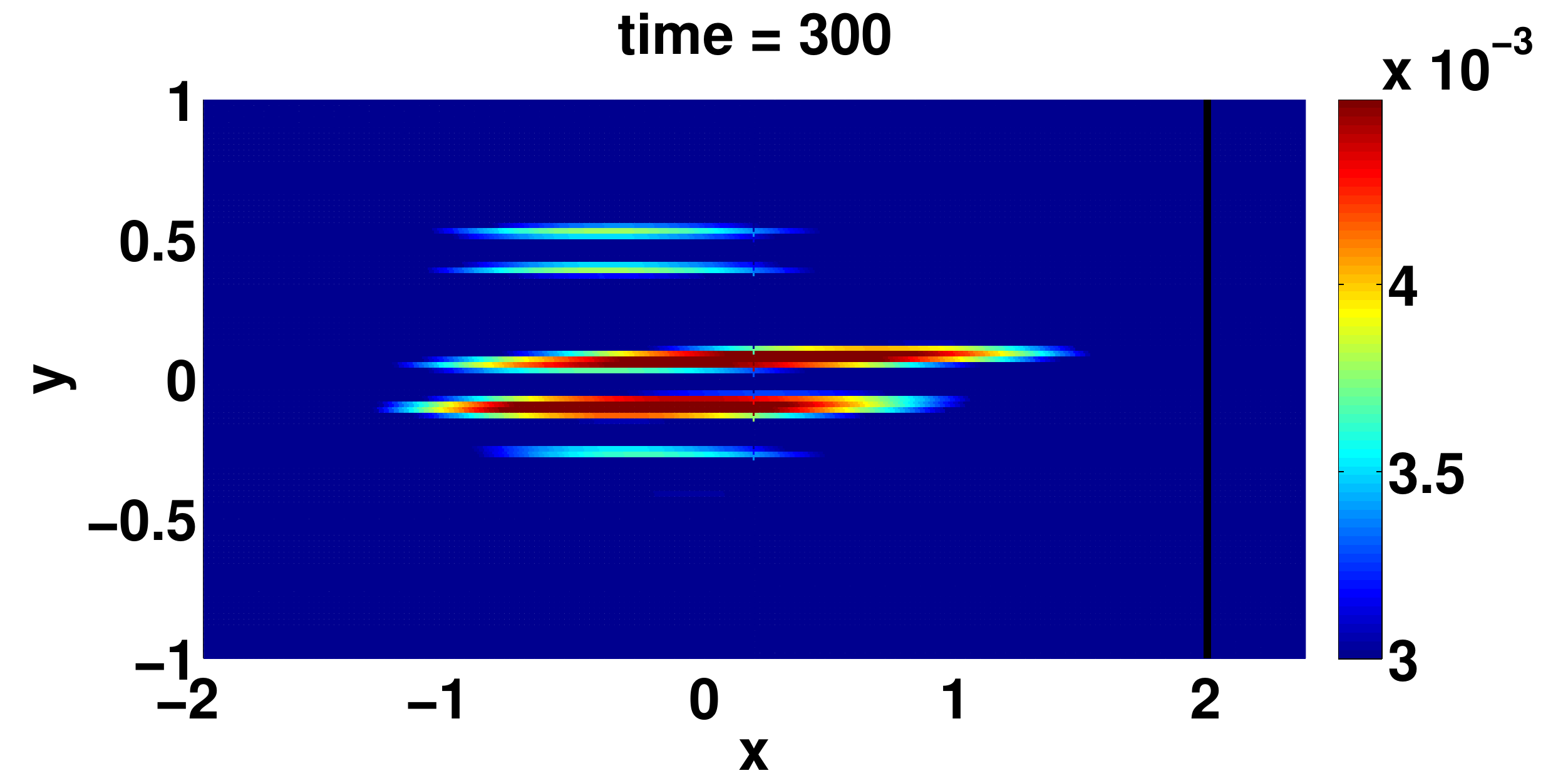} \label{fig:stable_solution_t300_4th}} 
 \caption{\textit{The snapshot of the solutions at $t = 300$.}}
 \label{fig:slution_t300}
\end{figure}

Note that initially the two solutions, using $\theta = 0, 1$, are similar.   For the standard choice, $\theta = 0$, the solution in the PML starts growing after a longtime, just before $t = 200$, and spreads into the computational domain as time passes.  At $t = 300$, growth in the PML has  already corrupted the solutions everywhere.  The 2nd and 6th order accurate operators with $\theta = 0$ are also unstable, but with slower growth rates. On the other hand, for $\theta = 1,$ the solutions decay throughout the simulation for all SBP operators. It is important to remark that  the  observed growth  is due to the unstable boundary treatments in the PML. 

The conclusion here is that a numerical method which is stable  in the absence of the PML, when $\sigma = 0$, can result in an unstable scheme when, $\sigma > 0$, the PML is included. It is possible though to design stable numerical methods which can be proven stable for all $\sigma \ge 0.$ Here,  numerical stability for the discrete PML \eqref{eq:Maxwell_PML_WaveGuide_Discrete_1}--\eqref{eq:Maxwell_PML_WaveGuide_Discrete_4} was achieved by using $\theta = 1$ in equation  \eqref{eq:Maxwell_PML_WaveGuide_Discrete_4} without any additional cost or difficulty.

\subsection{Accuracy}
Next, we evaluate the accuracy of the PML. We set $\theta = 1$ and compute the solutions for various resolutions using SBP operators of interior order of accuracy 2, 4 and 6,   until  the final time, $t = 5.$ The snapshots of the solutions at $t = 1,  2, 3, 4$ are displayed in figure \ref{fig:Maxwell_Solutions_Waveguide}, showing the evolution of the electric field on the boundary, the propagation and reflection of waves in the waveguide, and the absorption of waves by the PML.  

The errors introduced when a PML is used in computations can be divided into two different categories: numerical reflections and the modeling error. Numerical reflections are discrete effects introduced by discretizing the PML and seen inside the computational domain. The modeling error is introduced because the layer and the magnitude of damping coefficient are finite. Numerical reflections should converge to zero as the mesh is refined. Note that the modeling error is not caused by numerical approximations, and thus is independent of the numerical method used.  The modeling error  is expected to decrease as the PML width or the magnitude of damping coefficient increases. For sufficiently small mesh sizes numerical reflections are infinitesimally small and the modeling error dominates the total PML error.

In order to measure numerical errors, we also computed a reference solution in  a larger domain without the PML.   By comparing the reference solutions to the PML solutions in the interior, measured in the maximum norm, we obtain accurate measure of total PML errors. The numerical errors at $t = 5$ are displayed in table \ref{tab:Error_PML_Waveguide}, and the time history of the error are shown in figure \ref{fig:PMLError_WaveGuide}. 

Note that the errors converge to zero exponentially. Because we have used exactly the same  damping strength, $d_0$, for the 2nd, 4th and 6th order accurate schemes we expect to have the same  modeling error independent of the accuracy of the scheme used. It is also important to note that for small mesh--sizes the modeling error dominates and  the behavior of the total PML error does not depend on the order of accuracy of the scheme.
\begin{figure} [h!]
 \centering
{\includegraphics[width=.49\linewidth]{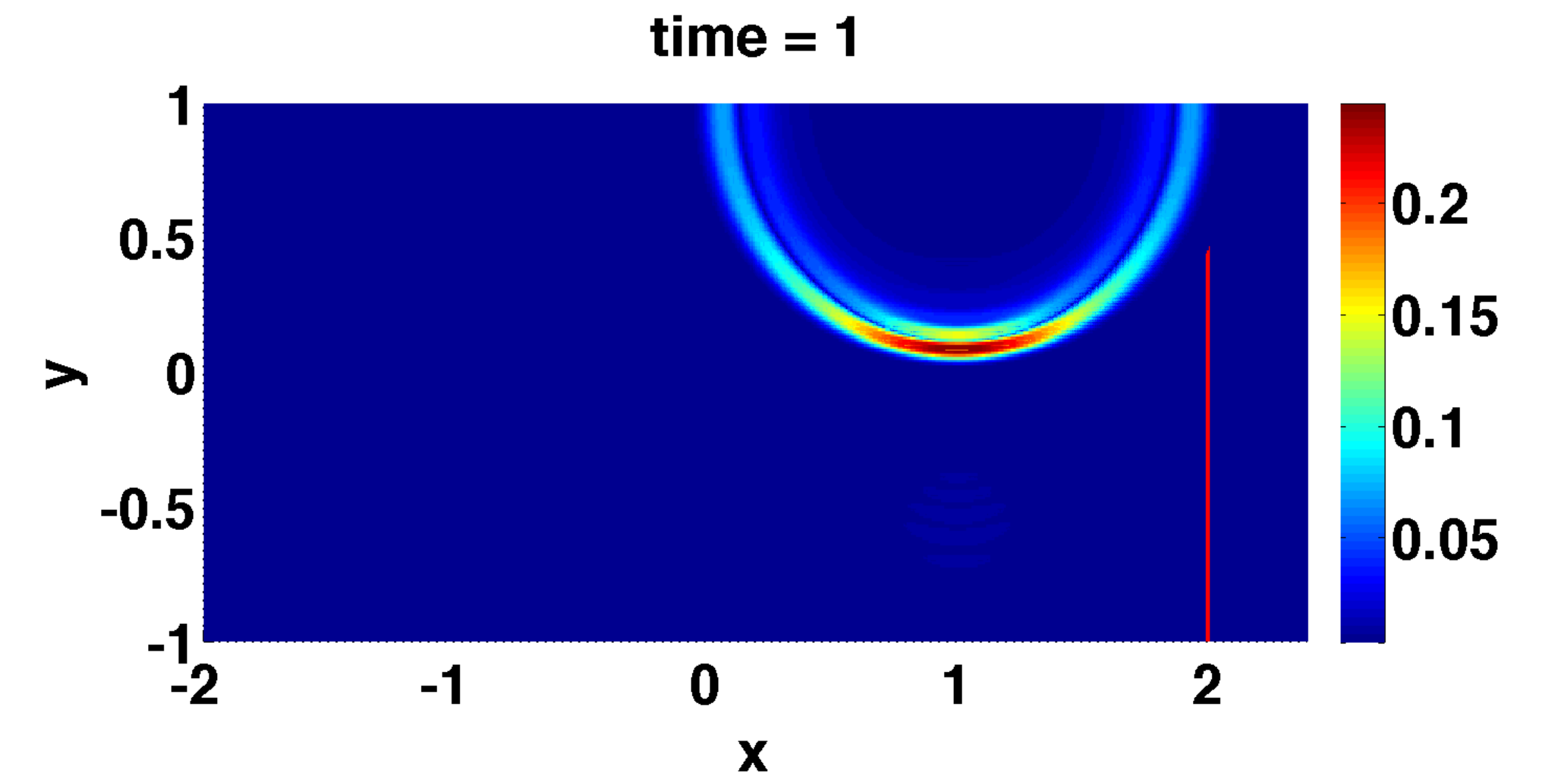}}
{\includegraphics[width=.49\linewidth]{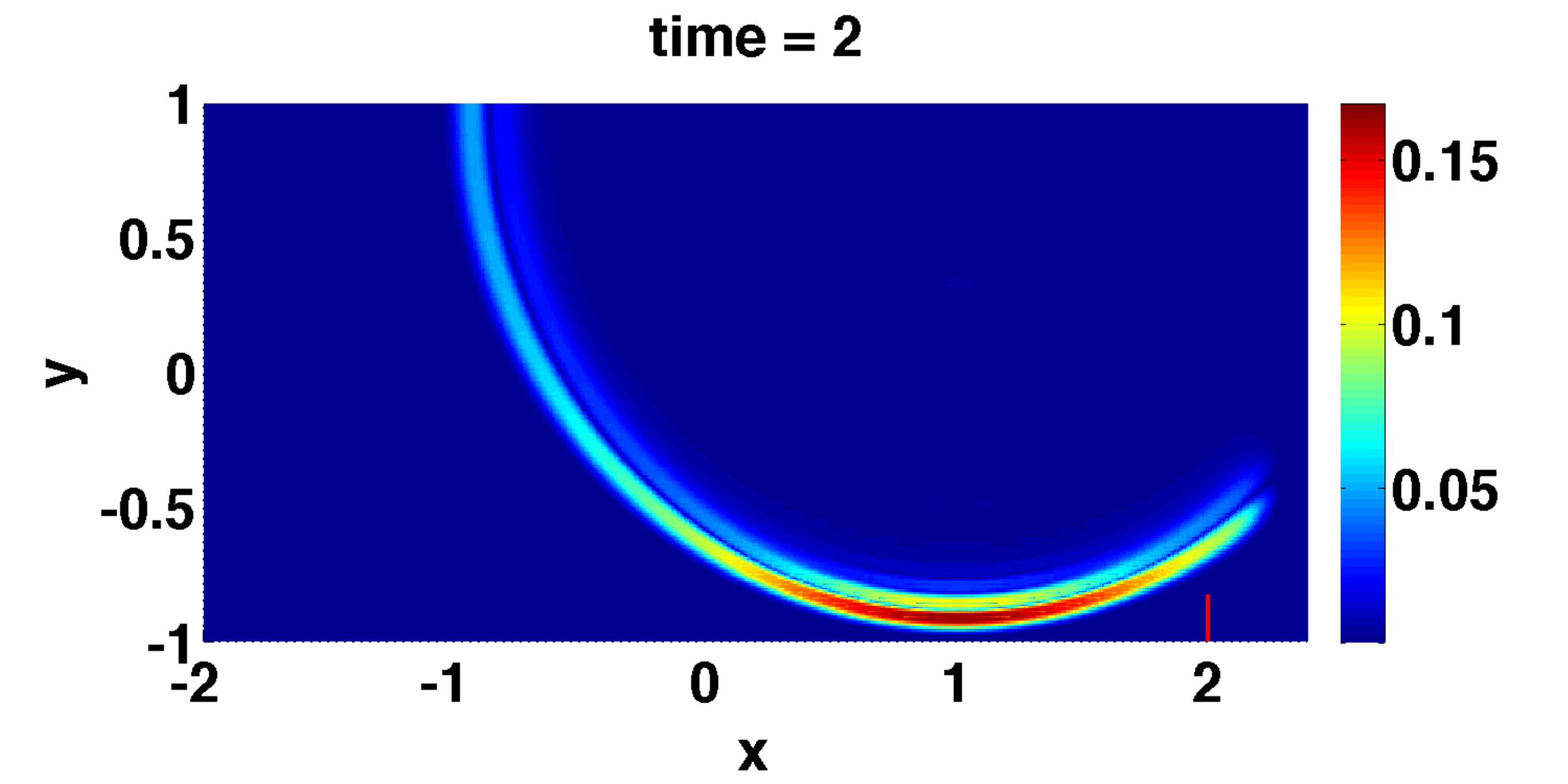}}
{\includegraphics[width=.49\linewidth]{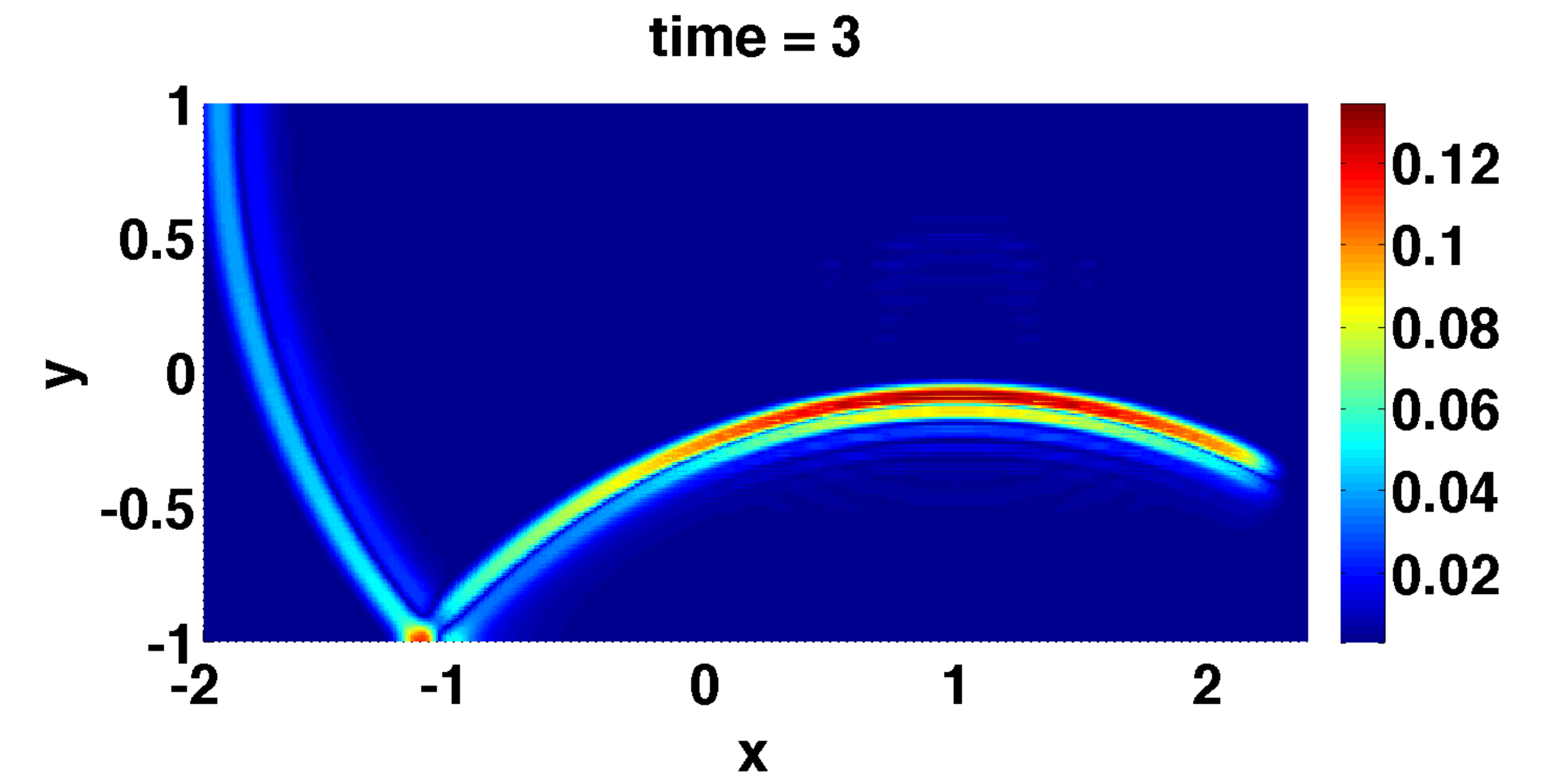}}
{\includegraphics[width=.49\linewidth]{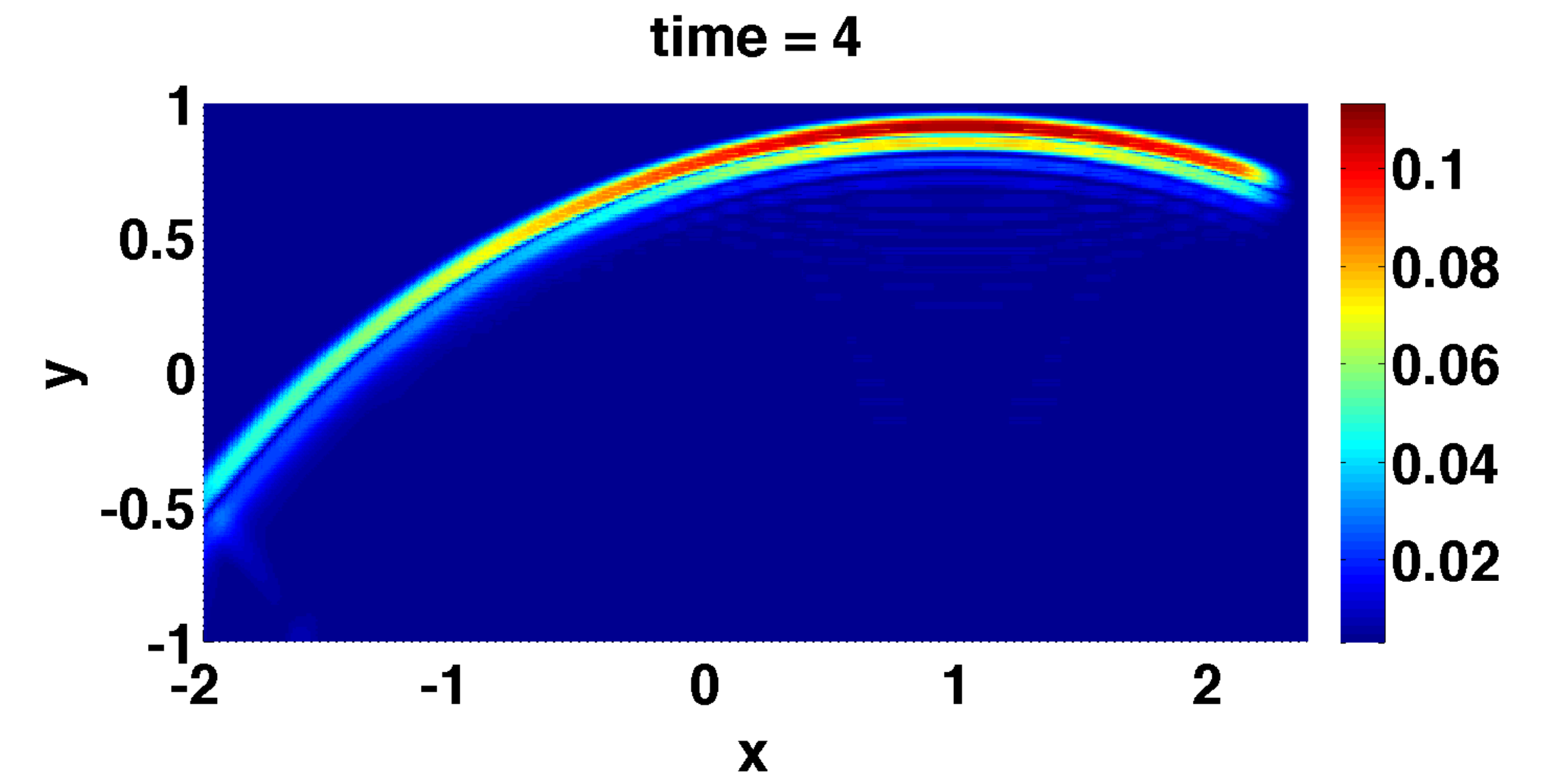}}
\caption{\textit{Snapshots of the electric field at $t = 1,  2, 3, 4$.}}
 \label{fig:Maxwell_Solutions_Waveguide}
\end{figure}

 \begin{table}[h!]
\centering
\begin{tabular}{c c c | c c   | c c c}
 {}&\multicolumn{2}{l}{$6$th Order}&\multicolumn{2}{l}{$4$th Order}&\multicolumn{2}{l}{$2$nd Order}\\
\cline{2-7}
 $h$&$\textbf{error}$&\textbf{rate}&$\textbf{error}$&\textbf{rate}&$\textbf{error}$&\textbf{rate}\\
\hline
0.04&1.08$\times 10^{-3}$&--&1.64$\times 10^{-3}$&--&2.63$\times 10^{-3}$&--\\
0.02& 6.22$\times 10^{-6}$&7.44&5.03$\times 10^{-6}$&8.35&3.61$\times 10^{-5}$&6.19\\
0.01&1.93$\times 10^{-7}$&5.01&8.22$\times 10^{-8}$&5.93&7.96$\times 10^{-8}$&8.82\\
0.005&7.42$\times 10^{-9}$&4.70&6.67$\times 10^{-9}$&3.62&7.43$\times 10^{-9}$&3.42\\
\hline
\end{tabular}
\caption{Total PML errors and the convergence of the error in a waveguide using SBP operators of interior order 2, 4, 6 for different grid resolutions.}
\label{tab:Error_PML_Waveguide}
\end{table}
\begin{figure} [h!]
 \centering
\subfigure[Second order accuracy]{\includegraphics[width=.48\linewidth]{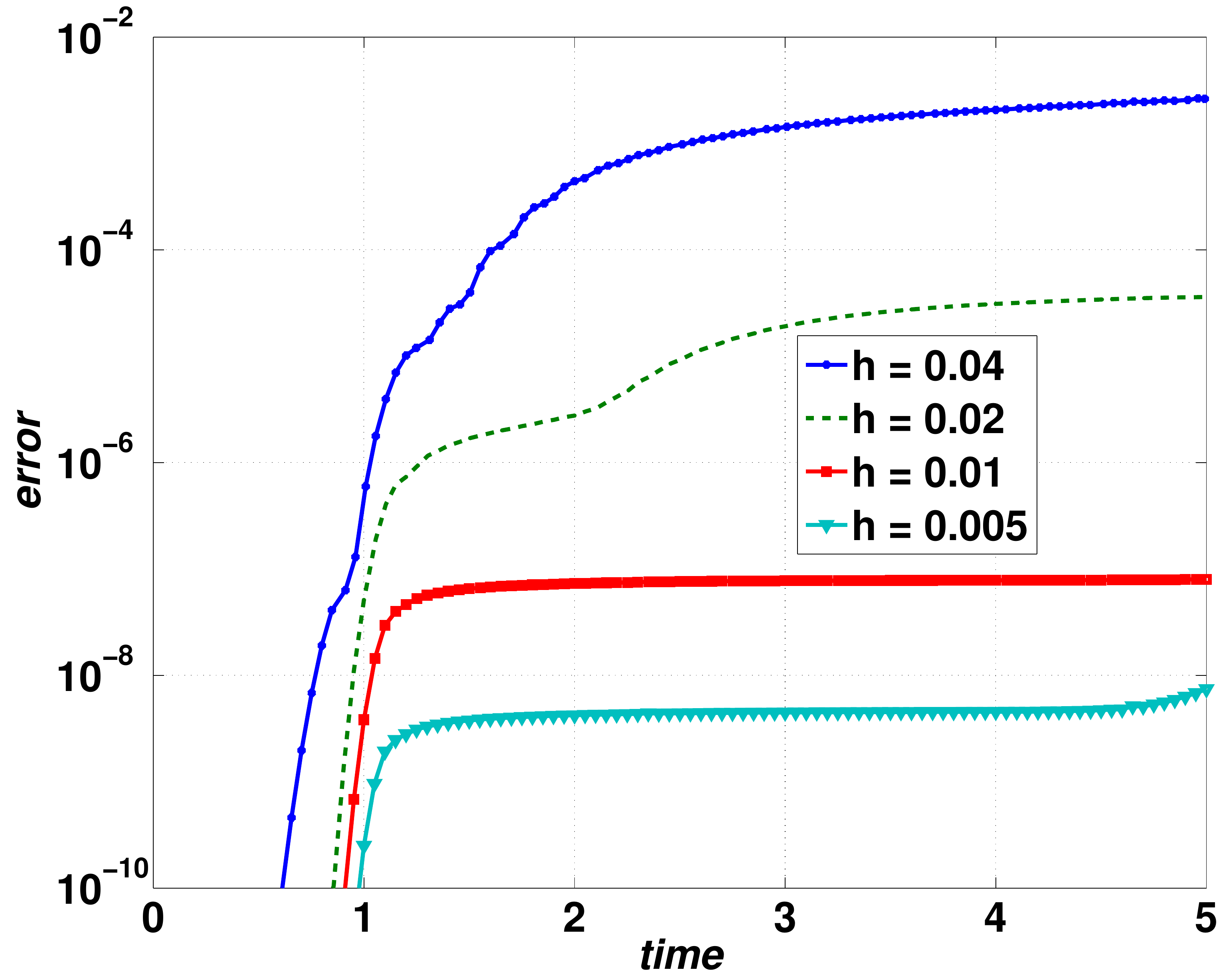}\label{fig:Error_SecondOrder_WaveGuide}}
\,
\subfigure[Fourth order accuracy]{\includegraphics[width=.48\linewidth]{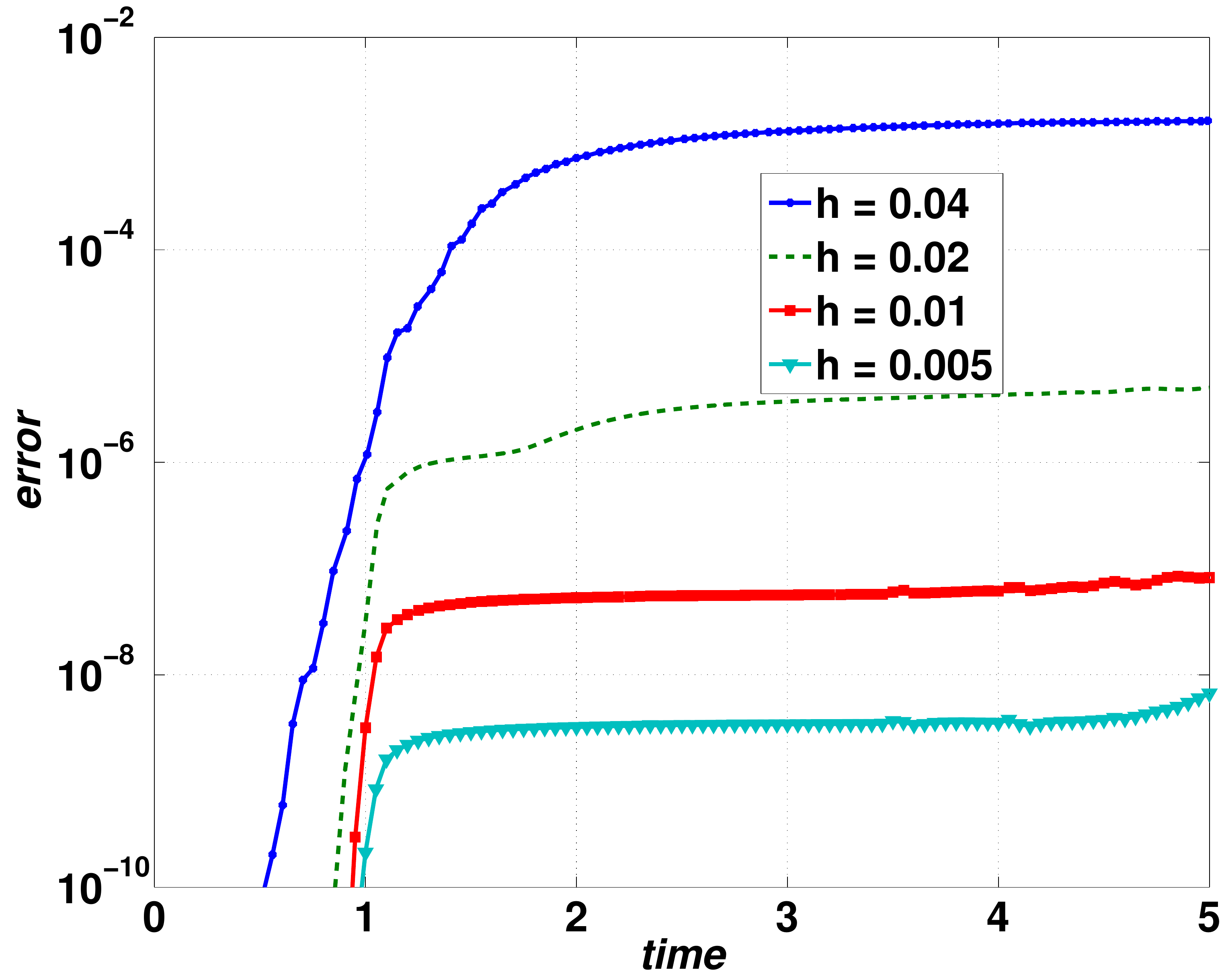}\label{fig:Error_FourthOrder_WaveGuide}}
\,
\subfigure[Sixth order accuracy]{\includegraphics[width=.48\linewidth]{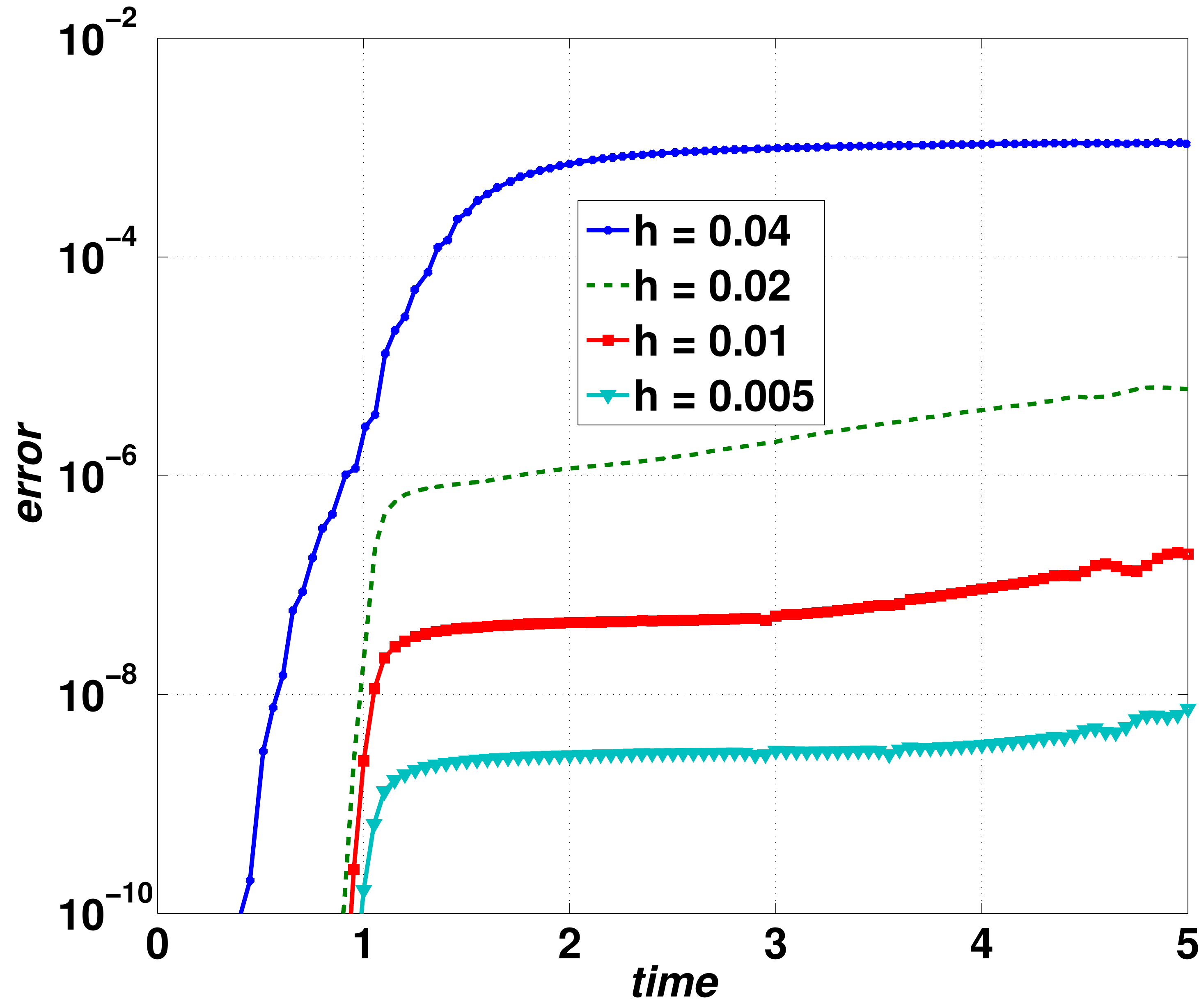}\label{fig:Error_SixthOrder_WaveGuide}}
 \caption{\textit{Total PML errors as a function of time using SBP operators of interior order 2, 4, 6 for different grid resolutions.}}
 \label{fig:PMLError_WaveGuide}
\end{figure}
\section{Summary and outlook}

In this paper, we have studied the (long time) temporal energy growth associated with numerical approximations of the  PML in bounded domains. Three different unsplit PML models and the classical split--field PML model are considered.  We proved the   stability of the constant coefficient IBVP for the PML. We also derived energy estimates for the constant coefficient IBVP in the Laplace space and for  variable coefficients  PML in the time domain domain.  Thus, establishing  well--posedness and stability of the IBVP corresponding to the PML. 

We have demonstrated in theory and numerical experiments that a numerical method which is stable in the absence of the PML can become unstable when the PML is introduced. Here, the discrete stability of the PML depends only on the numerical implementation of boundary conditions; the interior approximations do not play any roles.   We developed high order accurate and stable numerical approximation using SBP finite difference operators to approximate spatial derivatives and weak enforcement of boundary conditions using penalties. By constructing analogous discrete energy estimates we obtain a bound on the energy growth of the  solution  at any future time. Further, we showed that the corresponding constant coefficient PML problem can not support temporally growing solution.  We  presented numerical experiments demonstrating high order accuracy and longtime stability of the PML. 

We note that a stable numerical boundary treatment for a  particular PML model may not necessarily yield stable numerical solutions for another PML modeling the same physical phenomena. When the underlying physical model changes, the corresponding PML model will be different. We also expect different stable numerical boundary treatments.

Although the analysis here focuses on the Maxwell's equations the results obtained can be straightforwardly applied to equations of acoustics.

We are currently extending the techniques described in this paper  to other numerical methods such as finite element methods, spectral element methods, dG finite element/volume methods and mixed spectral-finite element methods. However, these  require substantially more details and  analysis. Our preliminary results are quite promising and will be reported in another paper \cite{Amler2014}.   It will be interesting to compare the results obtained from a dG method to the results obtained in this paper using high order finite difference SBP operators.

\appendix
\section{Some useful lemmata}
 Introduce the complex number $z = x + iy$ and define  the branch cut of $\sqrt{z}$ by
\begin{align}
-\pi < \arg{\left(x + iy\right)} \le \pi, \quad \arg{\sqrt{x + iy}} = \frac{1}{2}\arg{{\left(x + iy\right)}}.
\end{align}
The following Lemma was adapted from Lemma 6 in \cite{HeinzPetersson2011}.
\begin{lemma}\label{Lem:Heinz_Anders}
Let $k_x$ be a real number and let $s = a + ib$ be a complex number where $a > 0$. Consider the relation
\begin{equation}
\kappa = \sqrt{s^2 + k_x^2}.
\end{equation}
There are positive real numbers $\beta_0 \ge 1$, $0<\epsilon_0\le 1$ such that
$\Re{\kappa} = \beta_0a$, $\Im{\kappa} = \epsilon_0 b$.
\end{lemma}

We can also prove the following lemma
\begin{lemma}\label{Lem:Duru1}
Let  $\sigma \ge 0$, $k_x$ be  real numbers and let $s = a + ib$ be a complex number where $a > 0$. Consider the relation
\begin{equation}
\widetilde{\kappa} = \sqrt{s^2 + \left(\frac{k_x}{1+\frac{\sigma}{s}}\right)^2}.
\end{equation}
There are positive real numbers $\beta_{0} \ge 1$, $0<\epsilon_{0}\le 1$ such that
\[
\Re\left(\widetilde{\kappa}\right) = \beta_{0}a+ \sigma \frac{\left(\beta_0-\epsilon_0\right) b^2}{\left(a + \sigma\right)^2 + b^2} > 0.
\]
\end{lemma}
\begin{proof}
Consider
\begin{displaymath}
\begin{split}
&\widetilde{\kappa} = \sqrt{s^2 + \left(\frac{k_x}{1+\frac{\sigma}{s}}\right)^2} = \frac{s}{s+\sigma}\sqrt{\left(s+\sigma\right)^2 +k_x^2} =  \frac{a\left(a + \sigma\right) + b^2 + ib\sigma}{\left(a + \sigma\right)^2 + b^2}\left( \beta_{0} \left(a + \sigma\right) + i \epsilon_{0} b\right), \\
&\implies \Re\left(\widetilde{\kappa}\right) = \beta_{0}a+ \sigma \frac{\left(\beta_0-\epsilon_0\right) b^2}{\left(a + \sigma\right)^2 + b^2} > 0.
\end{split}
\end{displaymath}
\end{proof}
\begin{lemma}\label{Lem:Duru2}
Let  $\sigma \ge 0$, $k_x$ be  real numbers and let $s = a + ib$ be a complex number where $a > 0$. Consider the relation
\begin{equation}
\kappa = \sqrt{s^2 + k_x^2}.
\end{equation}
There are positive real numbers $\beta_{0} \ge 1$, $0<\epsilon_{0}\le 1$ such that
\[
\Re\left(\left(1 + \frac{\sigma}{s}\right){\kappa}\right) = \left(\beta_{0}a\left(1 + \frac{\sigma a}{a ^2 + b^2} \right) +  \epsilon_{0}\frac{\sigma b^2}{a ^2 + b^2}\right) > 0,
\]
\end{lemma}
\begin{proof}
Consider
\begin{displaymath}
\begin{split}
&\left(1 + \frac{\sigma}{s}\right){\kappa} = \left(1 + \frac{\sigma a}{a ^2 + b^2}- i\frac{\sigma b}{a ^2 + b^2}\right) \left( \beta_{0} a + i \epsilon_{0} b\right), \\
&\implies  \Re\left(\left(1 + \frac{\sigma}{s}\right){\kappa}\right) = \left(\beta_{0}a\left(1 + \frac{\sigma a}{a ^2 + b^2} \right) +  \epsilon_{0}\frac{\sigma b^2}{a ^2 + b^2}\right) > 0.
\end{split}
\end{displaymath}
\end{proof}
\begin{lemma}\label{Lem:Duru3}
Let  $\sigma \ge 0$, be  a nonnegative real number and  $s = a + ib$ a complex number where $a > 0$. Consider the relations
\begin{equation}
S_x = 1 + \frac{\sigma}{s} = \frac{a\left(a+\sigma\right) + b^2 - i\sigma b}{a^2 + b^2}, \quad A_0 = \frac{a^2 + b^2}{\left(a\left(a+\sigma\right) + b^2\right)^2 + \left(\sigma b\right)^2 } > 0.
\end{equation}
We have
\begin{align*}
 \Re{\left(\frac{1}{S_x}\right)} &= A_0\left(a\left(a+\sigma\right) + b^2\right)  > 0, \notag\\
\Re{\left(\frac{(sS_x)^*}{S_x}\right)} &= A_0\left(\left( a +\sigma\right)\left(a\left(a+\sigma\right) + b^2\right) + \sigma b^2\right) > 0,\notag\\
 \Re\left(\frac{s^*}{S_x}\right) &= A_0\left(a \left(a\left(a+\sigma\right) + b^2\right) + \sigma b^2\right)> 0.
\end{align*}
\end{lemma}
\begin{proof}
Consider
\begin{displaymath}
\begin{split}
&\frac{1}{S_x} = \frac{a^2 + b^2}{a\left(a+\sigma\right) + b^2 - i\sigma b} = A_0\left(a\left(a+\sigma\right) + b^2 + i\sigma b\right), \\
&\implies  \Re{\left(\frac{1}{S_x}\right)} = A_0\left(a\left(a+\sigma\right) + b^2\right) > 0.
\end{split}
\end{displaymath}
Note that 
\[
sS_x = s + \sigma = a+\sigma + ib \implies \left(sS_x\right)^* = a+\sigma - ib.
\]
Thus,  we have
\[
\frac{\left(sS_x\right)^*}{S_x} =  A_0\left(a\left(a+\sigma\right) + b^2 + i\sigma b\right)\left(a+\sigma - ib\right) \implies \Re{\left(\frac{(sS_x)^*}{S_x}\right)} = A_0\left(\left( a +\sigma\right)\left(a\left(a+\sigma\right) + b^2\right) + \sigma b^2\right) > 0.
\]
Consider now 
\[
\frac{s^*}{S_x} =  A_0\left(a\left(a+\sigma\right) + b^2 + i\sigma b\right)\left(a - ib\right) \implies \Re{\left(\frac{s^*}{S_x}\right)} = A_0\left( a\left(a\left(a+\sigma\right) + b^2\right) + \sigma b^2\right) > 0.
\]
The proof of the lemma is complete.
\end{proof}
\section{Proof of Theorem \ref{Theorem:Wellposedness}}\label{proof:Wellposedness}
Multiply  equation (\ref{eq:Maxwell_PML_WaveGuide_2nd}) by $ \partial { E_z}/\partial t$ and integrate  over the whole domain. 
Integration by parts yields
\begin{equation}\label{eq:wave1PML_Weak}
\begin{split}
&\left(\frac{\partial { E_z}}{\partial t}, \frac{\partial^2 { E_z}}{\partial t^2}\right)+ \left(\frac{\partial { E_z}}{\partial t}, \sigma\frac{\partial { E_z}}{\partial t}\right) \\
&=-\left(\frac{\partial^2 { E_z}}{\partial x \partial t}, \frac{\partial { E_z}}{\partial x} +\sigma H_y \right)-\left(\frac{\partial^2 { E_z}}{\partial y \partial t}, \frac{\partial { E_z}}{\partial y} + \sigma H_x\right) \\
&+  \int_{-y_0}^{y_0}\left(\frac{\partial  { E_z}}{\partial t}\left(\frac{\partial { E_z}}{\partial x} + \sigma H_y\right)\Big|_{-\left(x_0+\delta\right)}^{x_0+\delta}\right)dy +  \int_{-\left(x_0+\delta\right)}^{x_0+\delta}\left(\frac{\partial  { E_z}}{\partial t}\left(\frac{\partial { E_z}}{\partial y} + \sigma H_x\right)\Big|_{-y_0}^{y_0}\right)dx.
\end{split}
\end{equation}
Imposing the boundary conditions (\ref{eq:wall_bc_y1})  at $y = \pm y_0$ and  (\ref{eq:wall_bc_x1}) at $x = \pm \left(x_0+\delta\right) $ , we arrive at
\begin{equation}\label{eq:EnergyEstimate_first_continuous}
\begin{split}
&\frac{1}{2}\frac{d}{dt}\left(\Big\|\frac{\partial { E_z}}{\partial t} \Big\|^2 + \left(\frac{\partial { E_z}}{\partial x}, \frac{\partial { E_z}}{\partial x} \right) + 2 \left(\frac{\partial { E_z}}{\partial x},  \sigma H_y \right) + \left(\frac{\partial { E_z}}{\partial y}, \frac{\partial { E_z}}{\partial y} \right) + 2 \left(\frac{\partial { E_z}}{\partial y},  \sigma H_x \right) + \int_{\Gamma_x}\gamma_y\sigma\left( E_z^2(x, y_0, t) + E_z^2(x, -y_0, t)\right)dx\right) \\
&=    \left(\frac{\partial { E_z}}{\partial x}, {\sigma}\frac{\partial { H_y}}{\partial t}\right) + \left(\frac{\partial { E_z}}{\partial y}, {\sigma}\frac{\partial { H_x}}{\partial t}\right)  - \left(\frac{\partial { E_z}}{\partial t}, \sigma \frac{\partial { E_z}}{\partial t}\right) - \gamma_y\int_{\Gamma_x}\left(\frac{\partial  { E_z^2}}{\partial t}(x, -y_0, t) + \frac{\partial  { E_z^2}}{\partial t}(x, y_0, t)\right)dx\\
&-  \gamma_x\int_{\Gamma_y}\left(\frac{\partial  { E_z^2}}{\partial t}(-x_0-\delta, y, t) + \frac{\partial  { E_z^2}}{\partial t}(x_0+\delta, y, t) \right)dy.
\end{split}
\end{equation}
By adding 
\[
\frac{d}{dt}\left(\sigma H_y, \sigma H_y\right)  +  \frac{d}{dt}\left(\sigma H_x, \sigma H_x\right),
\]
to both sides of equation (\ref{eq:EnergyEstimate_first_continuous}) we have
\begin{equation}\label{eq:EnergyEstimate_second_continuous}
\begin{split}
&\frac{1}{2}\frac{d}{dt}\left( \Big\|\frac{\partial { E_z}}{\partial t} \Big\|^2  +  \Big\|\frac{\partial{ E_z}}{\partial x} + \sigma H_y\Big\|^2 + \Big\|\frac{\partial{ E_z}}{\partial y}+ \sigma H_x\Big\|^2 + \Big\|\sigma H_y\Big\|^2 + \Big\| \sigma H_x\Big\|^2  + \gamma \sigma \|E_z\left(-y_0, \mathrm{t}'\right)\|^2_{\Gamma_y} +\gamma \sigma \|E_z\left(y_0, \mathrm{t}'\right)\|^2_{\Gamma_x}\right) \\
&= \left(\frac{\partial { E_z}}{\partial x}, {\sigma}\frac{\partial { H_y}}{\partial t}\right) + \left(\frac{\partial { E_z}}{\partial y}, {\sigma}\frac{\partial { H_x}}{\partial t}\right) + 2\left(\frac{\partial { H_y}}{\partial t}, \sigma^2 { H_y} \right)  +2\left(\frac{\partial { H_x}}{\partial t}, \sigma^2 { H_x} \right)  - \left(\frac{\partial { E_z}}{\partial t}, \sigma \frac{\partial { E_z}}{\partial t}\right) \\
&- \gamma_y\int_{\Gamma_x}\left(\frac{\partial  { E_z}}{\partial t}^2(x, -y_0, t) + \frac{\partial  { E_z}}{\partial t}^2(x, y_0, t)\right)dx -  \gamma_x\int_{\Gamma_y}\left(\frac{\partial  { E_z}}{\partial t}^2(-x_0-\delta, y, t) + \frac{\partial  { E_z}}{\partial t}^2(x_0+\delta, y, t) \right)dy.
\end{split}
\end{equation}
By using \eqref{eq:Maxwell_PML_WaveGuide_x} and \eqref{eq:Maxwell_PML_WaveGuide_y} in  the right hand side of (\ref{eq:EnergyEstimate_second_continuous}), we obtain
\begin{equation}\label{eq:EnergyEstimate_third_continuous}
\begin{split}
\frac{1}{2}\frac{d}{dt}\mathcal{E}^{\left(1\right)}_{\sigma}\left(t\right) &= -\left(\frac{\partial { E_z}}{\partial x}, {\sigma}\left(\frac{\partial{ E_z}}{\partial x} + \sigma H_y \right)\right) + \left(\frac{\partial { E_z}}{\partial y}, {\sigma}\frac{\partial{ E_z}}{\partial y}  \right) 
- 2\left(\left(\frac{\partial{ E_z}}{\partial x} + \sigma\left(x\right) H_y\right), \sigma^2 { H_y} \right)  +2\left(\frac{\partial{ E_z}}{\partial y} , \sigma^2 { H_x} \right)
 \\
&
 - \left(\frac{\partial { E_z}}{\partial t}, \sigma \frac{\partial { E_z}}{\partial t}\right),
\end{split}
\end{equation}
where the energy $\mathcal{E}_{\sigma}^{\left(1\right)}$ is define in \eqref{eq:EnergyNorm_continuous}.
Note that if the boundary conditions are  (\ref{eq:wall_bc_y1_pec})  at $y = \pm y_0$ and  (\ref{eq:wall_bc_x1_pec}) at $x = \pm (x_0+\delta) $, from \eqref{eq:EnergyEstimate_first_continuous} we have
\begin{equation}\label{eq:EnergyEstimate_third_continuous_pec}
\begin{split}
\frac{1}{2}\frac{d}{dt}\mathcal{E}^{\left(0\right)}_{\sigma}\left(t\right) &= -\left(\frac{\partial { E_z}}{\partial x}, {\sigma}\left(\frac{\partial{ E_z}}{\partial x} + \sigma H_y \right)\right) + \left(\frac{\partial { E_z}}{\partial y}, {\sigma}\frac{\partial{ E_z}}{\partial y}  \right) 
- 2\left(\left(\frac{\partial{ E_z}}{\partial x} + \sigma\left(x\right) H_y\right), \sigma^2 { H_y} \right)  +2\left(\frac{\partial{ E_z}}{\partial y} , \sigma^2 { H_x} \right)
 \\
&
 - \left(\frac{\partial { E_z}}{\partial t}, \sigma \frac{\partial { E_z}}{\partial t}\right),
\end{split}
\end{equation}
Thus, Cauchy--Schwartz inequality and   the facts $\sqrt{\mathcal{E}^{\left(1\right)}_{\sigma}\left(t\right)} \ge \|{\partial { E_z}}/{\partial t}\|$, 
$\sqrt{\mathcal{E}^{\left(1\right)}_{\sigma}\left(t\right)} \ge \| {\partial { E_z}}/{\partial x}+\sigma{{ H_y}}\|$, 
$\sqrt{\mathcal{E}^{(1)}_{\sigma}} \ge  \| \sigma{{ H_y}}\|$, $\sqrt{\mathcal{E}^{\left(1\right)}_{\sigma}\left(t\right)} \ge \| {\partial { E_z}}/{\partial y}+\sigma{{ H_x}}\|$, 
$\sqrt{\mathcal{E}^{\left(1\right)}_{\sigma}\left(t\right)} \ge  \| \sigma{{ H_x}}\|$ and
$\sqrt{\mathcal{E}^{\left(0\right)}_{\sigma}\left(t\right)} \ge \|{\partial { E_z}}/{\partial t}\|$, 
$\sqrt{\mathcal{E}^{\left(0\right)}_{\sigma}\left(t\right)} \ge \| {\partial { E_z}}/{\partial x}+\sigma{{ H_y}}\|$, 
$\sqrt{\mathcal{E}^{(0)}_{\sigma}} \ge  \| \sigma{{ H_y}}\|$, $\sqrt{\mathcal{E}^{\left(0\right)}_{\sigma}\left(t\right)} \ge \| {\partial { E_z}}/{\partial y}+\sigma{{ H_x}}\|$, 
$\sqrt{\mathcal{E}^{\left(1\right)}_{\sigma}\left(t\right)} \ge  \| \sigma{{ H_x}}\|$
  lead to
\begin{equation}\label{eq:EnergyEstimate_third}
\begin{split}
\frac{d}{dt}\sqrt{\mathcal{E}^{\left(1\right)}_{\sigma}\left(t\right)} \le \sigma_{\infty}\sqrt{\mathcal{E}^{\left(1\right)}_{\sigma}\left(t\right)}, \quad \frac{d}{dt}\sqrt{\mathcal{E}^{\left(0\right)}_{\sigma}\left(t\right)} \le \sigma_{\infty}\sqrt{\mathcal{E}^{\left(0\right)}_{\sigma}\left(t\right)}  .
\end{split}
\end{equation}
The proof of the theorem is complete.
\hfill$\square$
\section{Proof of  Theorem \ref{Theorem:Discrete_Stability}}\label{proof:Discrete_Stability}
Set $\theta = 1$  in  \eqref{eq:Maxwell_PML_WaveGuide_Discrete_first_1} we have 
\begin{subequations}\label{eq:Maxwell_PML_WaveGuide_Discrete_first_appendix}
    \begin{alignat}{2}
      \frac{\mathrm{d^2}{\mathbf{E}_z}}{\mathrm{d t^2}}  &= -\left(\mathrm{P}_x^{-1}D_x^T \otimes I_y\right)\left(\mathrm{P}_x \otimes I_y\right)\left(\left(D_x \otimes I_y\right)\mathbf {E}_z +  \mathbf{\sigma}\mathbf {H}_y\right) - \left( I_x\otimes\mathrm{P}_y^{-1}D_y^T \right)\left(I_x \otimes \mathrm{P}_y\right)\left(\left(I_x \otimes D_y\right)\mathbf {E}_z +  \mathbf{\sigma}\mathbf {H}_x\right) 
  - \sigma\frac{\mathrm{d}{\mathbf{E}_z}}{\mathrm{d t}} \notag \\ 
  &  - \left(\mathrm{P}_x^{-1}\left(E_{Rx}+E_{Lx}\right)\otimes I_y + \left(I_x \otimes \mathrm{P}_y^{-1}\left(E_{Ry}+E_{Ly}\right)\right)\right)\frac{\mathrm{d}{\mathbf{E}_z}}{\mathrm{d t}}  - \sigma\left(I_x \otimes \mathrm{P}_y^{-1}\left(E_{Ry}+E_{Ly}\right)\right)\mathbf{E}_z ,\label{eq:Maxwell_PML_WaveGuide_Discrete_first_1_appendix} \\
      \frac{\mathrm{d}{\mathbf{H}_y}}{\mathrm{d t}}  &= -\left(D_x \otimes I_y\right)\mathbf {E}_z- \mathbf{\sigma}\mathbf {H}_y , \label{eq:Maxwell_PML_WaveGuide_Discrete_first_2_appendix} \\
     \frac{\mathrm{d}{\mathbf{H}_x}}{\mathrm{d t}}  &= \left(I_x \otimes D_y\right)\mathbf {E}_z .\label{eq:Maxwell_PML_WaveGuide_Discrete_first_3_appendix}
    \end{alignat}
  \end{subequations}
Multiplying equation \eqref{eq:Maxwell_PML_WaveGuide_Discrete_first_1_appendix} with $\left({\mathrm{d} { \mathbf{E}_z}}/{\mathrm {d t}}\right)^T{\mathbf{P}_{xy}}$ from the left, we have
\begin{equation}\label{eq:discrete_estimate_1}
\begin{split}
&\frac{1}{2}\frac{\mathrm{d} }{\mathrm {d t}}\Big\langle \frac{\mathrm{d} { \mathbf{E}_z}}{\mathrm {d t}}, \frac{\mathrm{d} { \mathbf{E}_z}}{\mathrm {d t}}\Big\rangle_{\mathbf{P}_{xy}} + \Big\langle \frac{\mathrm{d} { \mathbf{E}_z}}{\mathrm {d t}}, \left(\sigma+\left(E_{Rx}+E_{Lx}\right)\otimes \mathrm{P}_y + \left(\mathrm{P}_x \otimes \left(E_{Ry}+E_{Ly}\right)\right)\right)\frac{\mathrm{d} { \mathbf{E}_z}}{\mathrm {d t}}\Big\rangle =\\
& -\Big\langle\left(\left(D_x \otimes I_y\right)\mathbf {E}_z +  \mathbf{\sigma}\mathbf {H}_y\right), \left(D_x \otimes I_y\right)\frac{\mathrm{d} { \mathbf{E}_z}}{\mathrm {d t}}\Big\rangle_{\mathbf{P}_{xy}}-\Big\langle\left(\left(I_x \otimes D_y\right)\mathbf {E}_z +  \mathbf{\sigma}\mathbf {H}_x\right), \left(I_x \otimes D_y\right)\frac{\mathrm{d} { \mathbf{E}_z}}{\mathrm {d t}}\Big\rangle_{\mathbf{P}_{xy}} \\
& - \frac{\mathrm{d} { \mathbf{E}_z}}{\mathrm {d t}}^T\left(\mathrm{P}_x \otimes \sigma\left(E_{Ry}+E_{Ly}\right)\right)\mathbf{E}_z.
\end{split}
\end{equation}
By adding 
\begin{equation}
\frac{d}{dt}\Big\langle \sigma\mathbf{H}_y, \sigma\mathbf{H}_y\Big\rangle_{\mathbf{P}_{xy}}  + \frac{d}{dt}\Big\langle \sigma\mathbf{H}_x, \sigma\mathbf{H}_x\Big\rangle_{\mathbf{P}_{xy}}
\end{equation}
to both sides of equation (\ref{eq:discrete_estimate_1}) we have
\begin{equation}\label{eq:discrete_estimate_2}
\begin{split}
\frac{1}{2}\frac{d}{dt} \mathcal{E}_{h\sigma}^{\left(1\right)}\left(t\right) &= \Big\langle\left(D_x \otimes I_y\right)\mathbf {E}_z , {\sigma}\frac{d {\mathbf {H}_y}}{d t}\Big\rangle_{\mathbf{P}_{xy}} + \Big\langle\left(I_x \otimes D_y\right)\mathbf {E}_z , {\sigma}\frac{d {\mathbf {H}_x}}{d t}\Big\rangle_{\mathbf{P}_{xy}} + 2\Big\langle \sigma\frac{d \mathbf{H}_y}{dt}, \sigma\mathbf{H}_y\Big\rangle_{\mathbf{P}_{xy}}  + 2\Big\langle \sigma\frac{d \mathbf{H}_x}{dt}, \sigma\mathbf{H}_x\Big\rangle_{\mathbf{P}_{xy}} ,
\end{split}
\end{equation}
where the energy $\mathcal{E}_{h\sigma}^{\left(1\right)}\left(t\right)$ is defined in  \eqref{eq:Energy_Discrete_first}.
Eliminating the time derivatives on the right hand side  of  \eqref{eq:discrete_estimate_2}  using \eqref{eq:Maxwell_PML_WaveGuide_Discrete_first_2}--\eqref{eq:Maxwell_PML_WaveGuide_Discrete_first_3}, we obtain
\begin{equation}\label{eq:discrete_estimate_3}
\begin{split}
\frac{1}{2}\frac{d}{dt} \mathcal{E}_{h\sigma}^{\left(1\right)}\left(t\right) &= -\Big\langle\left(D_x \otimes I_y\right)\mathbf {E}_z , {\sigma}\left(\left(D_x \otimes I_y\right)\mathbf {E}_z+ \mathbf{\sigma}\mathbf {H}_y \right)\Big\rangle_{\mathbf{P}_{xy}} + \Big\langle\left(I_x \otimes D_y\right)\mathbf {E}_z , {\sigma}\left(I_x \otimes D_y\right)\mathbf {E}_z \Big\rangle_{\mathbf{P}_{xy}}\\
& - 2\Big\langle \sigma\left(\left(D_x \otimes I_y\right)\mathbf {E}_z+ \mathbf{\sigma}\mathbf {H}_y\right), \sigma\mathbf{H}_y\Big\rangle_{\mathbf{P}_{xy}}  + 2\Big\langle \sigma\left(I_x \otimes D_y\right)\mathbf {E}_z , \sigma\mathbf{H}_x\Big\rangle_{\mathbf{P}_{xy}} . 
\end{split}
\end{equation}
As before, Cauchy--Schwartz inequality and   the facts 
\[
\sqrt{\mathcal{E}_{h\sigma}^{\left(1\right)}} \ge \Big\|\frac{\mathrm{d} { \mathbf{E}_z}}{\mathrm {d t}} \Big\|_{\mathbf{P}_{xy}}, \quad
\sqrt{\mathcal{E}_{h\sigma}^{\left(1\right)}} \ge \Big\|\left(D_x \otimes I_y\right)\mathbf {E}_z+ \mathbf{\sigma}\mathbf {H}_y\Big\|_{\mathbf{P}_{xy}}, \quad  \sqrt{\mathcal{E}_{h\sigma}^{\left(1\right)}} \ge  \Big\|\sigma\mathbf {H}_y\Big\|_{\mathbf{P}_{xy}},
\]
\[
\sqrt{\mathcal{E}_{h\sigma}^{\left(1\right)}} \ge \Big\|\left(I_x \otimes D_y\right)\mathbf {E}_z+ \mathbf{\sigma}\mathbf {H}_x\Big\|_{\mathbf{P}_{xy}}, \quad  \sqrt{\mathcal{E}_{h\sigma}^{\left(1\right)}} \ge  \Big\|\sigma\mathbf {H}_x\Big\|_{\mathbf{P}_{xy}},
\]
 lead to
\begin{equation}\label{eq:discrete_estimate_4}
\begin{split}
\frac{d}{dt}\sqrt{\mathcal{E}_{h\sigma}^{\left(1\right)}} \le \sigma_{\infty}\sqrt{\mathcal{E}_{h\sigma}^{\left(1\right)}} .
\end{split}
\end{equation}
$\hfill\square$
\bibliographystyle{elsarticle-num}

\end{document}